\documentclass[11pt,makeidx]{amsart}
\usepackage{amsfonts,amssymb,amsmath,amscd, bbm,}
\usepackage{tikz, tikz-cd}
\usetikzlibrary{positioning}
\usepackage{amscd}
\usepackage{hyperref}
\usepackage[all]{xy}
\font\cyr=wncyr10

\newcommand{\A}{{\mathbb A}}

\newcommand{\bb}{{{\mathbbm b}}}
\newcommand{\C}{{\mathbb C}}
\newcommand{\F}{{\mathbb F}}
\newcommand{\ff}{{\mathbbm f}}
\newcommand{\G}{{\mathbb G}}

\newcommand{\Q}{{\mathbb Q}}
\newcommand{\Ql}{{\Q_\ell}}
\newcommand{\Qp}{{\Q_p}}
\newcommand{\bQ}{{\overline\Q}}

\newcommand{\bQp}{\bQ_p}
\newcommand{\R}{{\mathbb R}}

\newcommand{\Z}{{\mathbb Z}}
\newcommand{\Zl}{{\Z_\ell}}
\newcommand{\Zp}{{\Z_p}}
\newcommand{\bK}{\mathbb{K}}

\newcommand{\gra}{{\mathfrak a}}

\newcommand{\grb}{{\mathfrak b}}

\newcommand{\grh}{{\mathfrak h}}

\newcommand{\grm}{{\mathfrak m}}

\newcommand{\grO}{{\mathfrak O}}
\newcommand{\grp}{{\mathfrak p}}
\newcommand{\p}{{\mathfrak p}}
\newcommand{\grq}{{\mathfrak q}}

\newcommand{\eps}{{\epsilon}}

\newcommand{\kap}{{\kappa}}

\newcommand{\cA}{{\mathcal A}}

\newcommand{\CalD}{{\mathcal D}}
\newcommand{\CI}{{\mathcal I}}
\newcommand{\cE}{\mathcal{E}}
\newcommand{\cF}{\mathcal{F}}
\newcommand{\cG}{\mathcal{G}}

\newcommand{\cK}{{\mathcal K}}
\newcommand{\CL}{{\mathcal L}}
\newcommand{\cL}{{\mathcal L}}
\newcommand{\CM}{{\mathcal M}}

\renewcommand{\O}{{\mathcal O}}
\newcommand{\cO}{\mathcal{O}}
\newcommand{\cP}{{\mathcal P}}

\newcommand{\CU}{{\mathcal U}}

\newcommand{\CX}{{\mathcal X}}

\newcommand{\ac}{\mathrm{ac}}
\newcommand{\Ad}{\mathrm{Ad}}
\newcommand{\alg}{\mathrm{alg}}

\newcommand{\Aut}{\mathrm{Aut}}
\newcommand{\BK}{\mathrm{BK}}
\newcommand{\bv}{{\bar v}}
\newcommand{\can}{\mathrm{can}}

\newcommand{\coker}{\mathrm{coker}}
\newcommand{\comment}[1]{}
\newcommand{\cris}{cris}

\renewcommand{\div}{\mathrm{div}}
\newcommand{\ds}{\displaystyle}
\newcommand{\End}{\mathrm{End}}
\newcommand{\FBK}{\cF_{\BK}}
\newcommand{\frob}{\mathrm{frob}}

\newcommand{\Gal}{\mathrm{Gal}}
\newcommand{\GL}{\mathrm{GL}}
\newcommand{\Hida}{\mathrm{Hida}}
\newcommand{\Hom}{\mathrm{Hom}}

\newcommand{\isoarrow}{\stackrel{\sim}{\rightarrow}}
\newcommand{\isom}{\cong}

\newcommand{\loc}{\mathrm{loc}}
\newcommand{\M}{\mathrm{M}}

\newcommand{\ord}{\mathrm{ord}}
\newcommand{\ov}{\overline}

\newcommand{\ra}{\rightarrow}
\newcommand{\xra}[1]{\xrightarrow{#1}}
\newcommand{\hra}{\hookrightarrow}
\newcommand{\rank}{\mathrm{rank}}
\newcommand{\rCM}{\mathrm{CM}}
\newcommand{\rec}{\mathrm{rec}}

\newcommand{\res}{\mathrm{res}}
\newcommand{\Res}{\mathrm{Res}}
\newcommand{\rhobar}{\overline{\rho}}

\newcommand{\Sel}{\mathrm{Sel}}
\newcommand{\Sha}{\hbox{\cyr X}}

\newcommand{\tor}{\mathrm{tor}}

\newcommand{\ur}{\mathrm{ur}}

\newcommand{\XN}{{X_{N^+,N^-}}}
\newcommand{\XNs}{{X_{N^+,N^-}^*}}
\newcommand{\wB}{\widehat{B}}
\newcommand{\wZ}{\widehat{\Z}}

\newcommand{\vbar}{\overline{v}}
\newcommand{\SelGr}{\cF_{\Gr}}

\newcommand{\SelMin}{\cF_{\ac}}

\newcommand{\Xac}{X_{\ac}}
\newcommand{\fac}{f_{\ac}}

\renewcommand{\H}{\HH}

\DeclareMathOperator{\ab}{ab}

\DeclareMathOperator{\chr}{char}
\DeclareMathOperator{\cont}{cont}

\DeclareMathOperator{\ddiv}{div}

\DeclareMathOperator{\Frob}{Frob}

\DeclareMathOperator{\Gr}{Gr}
\DeclareMathOperator{\GU}{GU}
\DeclareMathOperator{\HH}{H}

\DeclareMathOperator{\Reg}{Reg}
\DeclareMathOperator{\rk}{rk}
\DeclareMathOperator{\tors}{tors}

\makeatletter
\@addtoreset{equation}{subsection}

\makeatother

\newtheorem{coro}[subsubsection]{\bf Corollary}
\newtheorem{lem}[subsubsection]{\bf Lemma}
\newtheorem{thm}[subsubsection]{\bf Theorem}
\newtheorem{prop}[subsubsection]{\bf Proposition}
\newtheorem{conj}[subsubsection]{\bf Conjecture}

\theoremstyle{definition}
\newtheorem{rmk}[subsubsection]{\it Remark}

\newcommand{\authnote}[2][]{\noindent {\if!#1!  {\bf TODO} \else {\small \bf #1} \fi: #2} \vspace{0.1in}}

\newcounter{assumptionnumber}
\newcommand{\assumption}[2][]{%
  \addtocounter{assumptionnumber}{1}%
  \begin{center}%
  \framebox[1.1\width]{\begin{minipage}{0.9\textwidth}%
  \textbf{Assumption \arabic{assumptionnumber}} \textit{\if!#1!\else (#1)\fi}: {#2}%
  \end{minipage}}%
  \end{center}%
}

\author{Dimitar Jetchev}
\address{
\'Ecole Polytechnique F\'ed\'erale de Lausanne\\
FSB MATHGEOM GR-JET\\
B\^atiment MA C3 605\\
CH-1015 Lausanne\\
Switzerland}
\email{dimitar.jetchev@epfl.ch} 

\author{Christopher Skinner}
\address{
Department of Mathematics\\
Princeton University\\
Fine Hall, Washington Road \\
Princeton, NJ 08544-1000\\
USA}
\email{cmcls@princeton.edu}
 
\author{Xin Wan}
\address{
Department of Mathematics\\
Columbia University\\
New York, NY 10027\\
USA 
}
\email{xw222@math.columbia.edu}

\title[The BSD formula for the rank one case]
{The Birch and Swinnerton-Dyer Formula for Elliptic Curves
of Analytic Rank One
}

\begin{document}

\setcounter{section}{0}
\begin{abstract} Let $E/\Q$ be a semistable elliptic curve such that $\ord_{s=1}L(E,s) = 1$.
We prove the $p$-part of the Birch and Swinnerton-Dyer formula for $E/\Q$ for each prime $p\geq 5$ of good reduction such that $E[p]$ is irreducible: 
$$
\ord_p \left (\frac{L'(E,1)}{\Omega_E\cdot\Reg(E/\Q)} \right ) = \ord_p \left (\#\Sha(E/\Q)\prod_{\ell\leq \infty} c_\ell(E/\Q) \right ).
$$
This formula also holds for $p=3$ provided $a_p(E)=0$ if $E$ has supersingular reduction at $p$.
\end{abstract}
\maketitle

%
%
\section{Introduction}\label{bsd-intro}

\subsection{The Birch and Swinnerton-Dyer conjecture}\label{subsec:bsd}
Let $E/\Q$ be an elliptic curve.  The Birch and Swinnerton-Dyer Conjecture for $E$, as stated by Tate \cite[Conj.~4]{tate:arithmetic} (see also \cite{tate:bsd} 
and \cite[(I)-(IV)]{birch:bsd}) is the following:

\begin{conj}[Birch and Swinnerton-Dyer Conjecture]\label{conj:bsd}\hfill
\begin{itemize}
\item[(a)] The order of the zero at $s=1$ of the Hasse--Weil $L$-function $L(E,s)$ is equal to the rank $r$ of the Mordell--Weil group $E(\Q)$.
\item[(b)] Let $\Reg(E/\Q)$ be the regulator of $E(\Q)$ (the discriminant of the N\'eron-Tate height-pairing on $E(\Q)$),
$\Sha(E/\Q)$ the Tate-Shafarevich of $E$, $c_\ell(E/\Q)$ the Tamagawa number of $E$ at a prime $\ell$, and 
$\Omega_E = \int_{E(\R)} |\omega_E|$ for $\omega_E$ a N\'eron differential $($a $\Z$-basis for the differentials of the
N\'eron model of $E$ over $\Z$$)$. Then
\begin{equation}\label{eq:bsd-formula}
\frac{L^{(r)}(E,1)}{r! \cdot \Omega_E \cdot \Reg(E/\Q)} = \frac{\#\Sha(E/\Q)}{(\#E(\Q)_{\tor})^2}\prod_{\ell} c_\ell(E/\Q),
\end{equation}
\end{itemize}
\end{conj}
\noindent We call the formula \eqref{eq:bsd-formula} in (b) the {\it BSD formula for $E$}.

\subsection{Main result}
Let $E / \Q$ be a semistable elliptic curve of conductor $N$ (so $N$ is square-free). 
Let $p \geq 3$ be a prime of good reduction (i.e., $p\nmid N$) such that the mod $p$ Galois representation 
$\rhobar_{E, p} \colon \Gal(\overline{\Q} / \Q) \ra \Aut(E[p])$ is irreducible.
Suppose that $\ord_{s = 1} L(E/\Q, s) = 1$. 
In this case the work of Gross and Zagier \cite{gross-zagier} and of Kolyvagin \cite{kolyvagin:euler_systems, kolyvagin:structure_of_selmer, kolyvagin:structureofsha} (see also \cite{gross:kolyvagin}) implies that $\text{rk}_{\Z} E(\Q) = 1$ and $\# \Sha(E/\Q) < \infty$. In particular, part (a) of Conjecture \ref{conj:bsd} holds for $E$.
In this paper we prove that the $p$-part of the BSD formula \eqref{eq:bsd-formula} holds for $E$:

\begin{thm}[$p$-part of the Birch and Swinnerton-Dyer formula]\label{thm:main} 
If $p\geq 5$, then 
\begin{equation}\label{eq:bsd}
\ord_p \left (\frac{L'(E,1)}{\Reg(E/\Q)\cdot \Omega_{E}} \right ) = \ord_p \left (\#\Sha(E/\Q)\prod_{\ell\nmid \infty} c_\ell(E/\Q) \right ). 
\end{equation}
If $p = 3$, then \eqref{eq:bsd} holds provided $a_p(E)=0$ when $E$ has supersingular reduction at~$p$. 
\end{thm} 

\noindent It is a consequence of the Gross--Zagier formula that $\frac{L'(E/\Q,1)}{\Reg(E/\Q)\cdot \Omega_{E}} \in \Q^\times$ 
(see \cite[Thm.7.3]{gross-zagier}), so the $p$-adic valuation of the left-hand side of \eqref{eq:bsd} makes sense.

Particular cases of Theorem~\ref{thm:main} have been obtained
by Zhang \cite{zhang:selheeg} and by Berti, Bertolini, and Venerucci \cite{berti-bertolini-venerucci}, 
for $p\geq 5$ a prime of good ordinary reduction. In particular, \cite[Thm.~7.3]{zhang:selheeg} also has the extra assumption that for any $\ell \mid\mid N$ for which $\ell \equiv \pm 1 \bmod p$, $\rhobar_{E, p}$ is ramified at $\ell$ (equivalently,
that $p\nmid c_\ell(E/\Q)$), and   
\cite[Thm.~A]{berti-bertolini-venerucci} also assumes that $p$ is not anomalous (that is, $p\nmid\#E(\F_p)$)
and that $p$ does not divide any of the Tamagawa factors $c_\ell(E/\Q)$.  
In contrast, Theorem~\ref{thm:main} is general  -- including both the ordinary and supersingular cases, as well as the cases where there are Tamagawa numbers divisible by $p$ -- 
aside from the extra hypothesis that $a_p(E)=0$ when $E$ has supersingular reduction at $p$ if $p = 3$.

If $p$ is a prime of supersingular reduction for $E$ and $a_p(E)=0$ (which is always the case for supersingular $p\geq 5$), then the conclusion of Theorem~\ref{thm:main}
follows from combining the work of Kobayashi on the $p$-adic Gross--Zagier formula and the non-vanishing of the $p$-adic height of the Heegner point 
\cite[Thm.~1.1 and Cor.~4.9]{kobayashi:gz} with Wan's recent work on Kobayashi's supersingular variant of the Iwasawa main conjecture for $E$ \cite{wan:rankin-ss}; cf.~\cite[Cor.~1.3]{kobayashi:gz}. This argument does not apply in the case where
$E$ has ordinary reduction at $p$ as the $p$-adic height of the Heegner point is then not known to be non-zero. The proof of Theorem \ref{thm:main} in this paper 
treats the ordinary and supersingular cases the same and also applies to $\GL_2$-type modular abelian varieties. 

If $E$ has complex multiplication (in which case $E$ is not semistable and $N$ is not square-free), the equality \eqref{eq:bsd} similarly follows from the $p$-adic Gross--Zagier formula together with 
the non-vanishing of the $p$-adic height of the Heegner point and the Iwasawa main conjectures for $E$; see \cite[Cor.~1.4]{kobayashi:gz}. 

When $\ord_{s=1}L(E,s) = 0$ (that is, $L(E,1)\neq 0$) then the $p$-part of the Birch and Swinnerton-Dyer conjectural formula is also known. In fact, the equality in this case is an important 
ingredient in our proof of Theorem \ref{thm:main} as well as the proofs of all the results described above. For more on what is known in this case see Theorem~\ref{thm:rank0} below.

\subsection{Outline of the proof}\label{subsec:outline}
Theorem \ref{thm:main} is proved in this paper by separately establishing the upper and lower bounds predicted by \eqref{eq:bsd-formula} for the order $\#\Sha(E/\Q)[p^\infty]$
of the $p$-primary part of the Tate--Shafarevich group.

To prove the exact lower bounds we use anticyclotomic Iwasawa theory. 
For a suitable imaginary quadratic field $\cK' = \Q(\sqrt{D'})$ of discriminant $D' < 0$, using arguments similar to \cite[\S 4]{greenberg:iwasawa-elliptic} we prove a control theorem, Theorem~\ref{thm:control}, that compares a specialization of a certain $\Lambda$-cotorsion Selmer group (here $\Lambda = \Z_p[[\Gal(\cK'_\infty / \cK')]]$ is the Iwasawa algebra
for the anticyclotomic $\Z_p$-extension $\cK'_\infty$ of $\cK'$)
to a certain $\Z_p$-cotorsion Selmer group that is closely related to the $p$-primary part $\Sha(E/\cK')[p^\infty]$ of the Tate--Shafarevich group.
This comparison is stated as an explicit formula involving the Tamagawa numbers of $E$ at primes that split in $\cK'$. 
The $\Lambda$-cotorsion Selmer group is the Selmer group related via an 
anticyclotomic main conjecture to a $p$-adic $L$-function recently constructed for 
modular curves by Bertolini, Darmon and Prasanna \cite{bertolini-darmon-prasanna} 
and extended to Shimura curves by Brooks \cite{brooks:shimura} and Liu, Zhang, and 
Zhang \cite{liu-zhang-zhang}. This $p$-adic $L$-function is also a specialization of a 
two-variable $p$-adic $L$-function constructed by Hida, as explained in 
Section~\ref{bsd-padicLfunc}. We note that the anticyclotomic main conjecture that we 
use is different from the classical anticyclotomic main conjecture of Perrin-Riou formulated in 
\cite{perrin-riou:mc} (see also \cite[p.3]{howard:heeg}): the $p$-adic $L$-functions in the 
two cases have different ranges of interpolation. Using an extension of the methods of \cite{skinner-urban:gl2}, Wan \cite{wan:rankin, wan:rankin-ss} has proved that this 
$p$-adic $L$-function divides the characteristic ideal of the $\Lambda$-cotorsion Selmer group up to a power of $p$ (see Section~\ref{bsd-padicLfunc} for the definitions of the relevant $p$-adic $L$-functions and Section~\ref{bsd-acmc} for the precise statement of the main conjecture). 
To extend this to an unambiguous divisibility in $\Lambda$ and to pass from this divisibility to lower bounds on $\#\Sha(E/\cK')[p^\infty]$ in terms
of the index of a suitable Heegner point, we need two key ingredients: 1) a recent result of Burungale \cite{burungale:mu2} on the vanishing of the corresponding analytic $\mu$-invariant, and 2) a central value formula due to Brooks \cite[Thm.1.1]{brooks:shimura} generalizing a recent formula of Bertolini, Darmon and Prasanna relating the $p$-adic $L$-function at a point outside of the range of interpolation to the $p$-adic logarithm of a Heegner point $z_{\cK'}\in E(\cK')$. The result is the inequality
\begin{equation}\label{bsd-inequality}
\ord_p(\#\Sha(E/\cK')[p^\infty])\geq
\ord_p\left([E(\cK'):\Z\cdot z_{\cK'}]^2\prod_{{w\mid N}\atop{\text{$w$ a split prime of $\cK'$}}} c_w(E/\cK')\right)
\end{equation}
The Heegner point $z_{\cK'}$ appearing in this formula (and also appearing in the formula of Brooks) 
comes from a parameterization of $E$ by a Shimura curve $X_{N^+, N^-}$, $N^+ N^- = N$, of level $N^+$ attached to a quaternion algebra 
$B = B_{N^-}$ of discriminant $N^-$. In order to appeal to the known results about the anticyclotomic main conjecture of interest, 
it is necessary to take $N^->1$ (so this is not the classical Heegner point setting).  To pass from the inequality \eqref{bsd-inequality}
to one where the right-hand side is replaced with an $L$-value, we combine the inequality with the general 
Gross--Zagier formula
for the point $z_\cK'$, due to Zhang \cite{zhang:gz-shimura} and Yuan, Zhang, and Zhang \cite{yuanzhangzhang:gz}.  The result is the exact lower bound
for $\#\Sha(E/\cK')[p^\infty]$ predicted by the BSD formula for $E/\cK'$ (see Conjecture \ref{conj:bsd-gen}).
We note that the Tamagawa numbers at the non-split primes of $\cK'$ now appear, coming into the Gross--Zagier formula as the ratio of the degree of 
the usual modular paramaterization and the degree of the Shimura curve parameterization (this is essentially due to Ribet and Takahashi \cite{ribet-takahashi}).
Finally, to obtain the expected lower bound on $\#\Sha(E/\Q)[p^\infty]$ we express both sides of the resulting inequality for $\#\Sha(E/\cK')[p^\infty]$
in terms of $E/\Q$ and its $\cK'$-twist $E^{D'}/\Q$. We then exploit the fact that Kato has proved that the predicted {\it upper bound} holds
for $\#\Sha(E^{D'}/\Q)[p^\infty]$ (as $\cK'$ is chosen so that $L(E^{D'},1)\neq 0$).
 
Upper bounds on $\#\Sha(E/\Q)[p^\infty]$ are typically achieved by an Euler system argument via the method developed by Kolyvagin 
\cite{kolyvagin:euler_systems}, which applies since $E/\Q$ is assumed to have analytic rank one. 
Kolyvagin's method uses the Euler system constructed from Heegner points that are the images of CM moduli via the usual modular parameterization $X_0(N)\ra E$.
As explained in \cite{jetchev:tamagawa}, this 
has a drawback in the sense that it will only give the precise upper bounds on the $p$-primary part of the Tate-Shafarevich group (for $E$ over
a suitable imaginary quadratic field) if at most one Tamagawa number of $E$ is divisible by $p$. To get around this problem we again consider a parameterization of $E$ by a general Shimura curve $X_{N^+,N^-}$, $N^+N^-=N$.
As explained in Section~\ref{subsubsec:upper}, it is possible to choose this parametrization so that no Tamagawa number at a prime 
dividing $N^+$ is divisible by $p$. We then choose an imaginary quadratic field $\cK'' = \Q(\sqrt{D''})$, $D''<0$, such that each prime
dividing $N^+$ splits in $\cK''$ and each prime dividing $N^-$ is inert in $\cK''$ (this will generally be a different field than 
the $\cK'$ used to establish the lower bound). Kolyvagin's method applied to the Heegner points obtained from 
the parameterization by $X_{N^+, N^-}$ and the field $\cK''$ then yields an inequality in the direction opposite of \eqref{bsd-inequality}:
\begin{equation}\label{eq:bsd-inequality 2}
\ord_p(\#\Sha(E/\cK'')[p^\infty])\leq
\ord_p\left([E(\cK''):\Z\cdot z_{\cK''}]^2\right). 
\end{equation}
Appealing to the general Gross--Zagier formula for $z_{\cK''}$ we then get the upper bound on $\#\Sha(E/\cK'')[p^\infty]$ predicted
by the BSD formula for $E/\cK''$. To pass from this to the expected upper bound for $\#\Sha(E/\Q)[p^\infty]$ we make use of the fact
that the predicted {\it lower bound} for $\#\Sha(E^{D''}/\Q)[p^\infty]$ is known.
This lower bound follows from the proved cases of the cyclotomic main conjectures for $E^{D''}$ (see \cite[Thm.2(a)]{skinner-urban:gl2} and \cite[Thm.C]{skinner:multred}
for the ordinary case and \cite[Thm.1.3]{wan:kobayashi} for the supersingular case).

It is only at the final step, where we invoke the $p$-part of the BSD formula for $L(E^{D''},1)$, that we need to assume that $a_p(E)=0$ 
if $E$ has supersingular reduction at $p$, which is only a real condition when $p=3$. Furthermore, most of our arguments apply
more generally to the situation where $E$ is replaced by a newform of weight 2, square-free level, and trivial character (see also Section \ref{subsubsec:remarks} for additional comments on the general case).

\subsection{Organization of the paper}
In section~\ref{bsd-prelim} we recall some relevant background on Galois representations, local conditions and Selmer modules (generalizations of classical Selmer groups),
as well as the specific cases arising from newforms and modular abelian varieties. 
In Section~\ref{bsd-control} we prove the control theorems for the relevant anticyclotomic Selmer groups. In Section~\ref{bsd-cmpts} we include the relevant background on quaternion algebras, Shimura curves, CM points, and the Kolyvagin system coming from a 
Shimura curve, and we recall the upper bounds on $\#\Sha(E/\cK)[p^\infty]$ obtained from Kolyvagin's argument in the setting of Shimura curves. 
Section~\ref{bsd-padicLfunc} is about $p$-adic $L$-functions and various comparisons. In it we recall the $p$-adic anticyclotomic 
$L$-function constructed by Brooks \cite{brooks:shimura} and compare it to a specialization of Hida's two variable 
$p$-adic $L$-function \cite{hida:padic2}. Section~\ref{bsd-padicLfunc} also includes the statement of Burungale's result on the vanishing of the analytic 
anticyclotomic $\mu$-invariants
and the statement of Brooks' result expressing a certain value of the anticyclotomic $p$-adic $L$-function in terms of a $p$-adic logarithm of a 
Heegner point. In Section~\ref{bsd-acmc} we discuss the relevant anticyclotomic main 
conjectures and recent progress on proving them. 
We complete the proof of our main result, Theorem~\ref{thm:main}, in Section~\ref{bsd-mainthm}.

%
%
\section{Preliminaries}\label{bsd-prelim}
Let $\bQ\subset\C$ be the algebraic closure of $\Q$. For any number field $F \subset\bQ$, let $G_F = \Gal(\bQ/F)$. For each place $v$ of $F$, fix an algebraic closure $\ov F_v$ of $F_v$ and an $F$-embedding $\iota_v \colon \bQ \hookrightarrow \ov F_v$ to get an identification of  $G_{F_v} = \Gal(\ov F_v/F_v)$ with a
decomposition group for $v$ in $G_F$. For each finite place $v$, let $I_v\subset G_{F_v}$ be the inertia subgroup and
$\Frob_v \in G_{F_v}/I_v$ a geometric\footnote{Throughout this paper we take geometric normalizations (e.g., for Frobenius elements, for the reciprocity maps of class field theory, for Hodge--Tate weights, for $L$-functions of Galois representations).}
Frobenius element. Let $\F_v$ be the residue field of $v$ and $\overline\F_v$ an algebraic closure of $\F_v$. 
Then there is a canonical identification $G_{F_v}/I_{v}= \Gal(\overline \F_v/\F_v)$.

Throughout, let $p\geq 3$ be a fixed prime and let $\eps \colon G_\Q\rightarrow \Z_p^\times$ be the $p$-adic cyclotomic character. 

\subsection{Galois representations}\label{subsec:galoisreps}
Let $F$ be a number field (for much of this paper $F$ will be either $\Q$ or an imaginary quadratic field $\cK$).
By a $p$-adic Galois representation of $G_F$ we will always mean a finite-dimensional vector space $V$ over a finite extension
$L/\Q_p$ that is equipped with a continuous $L$-linear $G_F$-action. Such a representation
$V$ will be understood to come with a scalar field $L$.  We will always assume that 
\begin{equation}\label{geom}
\text{$V$ is geometric}
\tag{{geom}}
\end{equation}
in the sense introduced by Fontaine and Mazur \cite{fontainemazur:geom}: $V$ is unramified away from a finite set of 
places and potentially semistable at all places $w\mid p$ of $F$. We will further assume that 
\begin{equation}\label{pure}
\text{$V$ is pure.}
\tag{{pure}}
\end{equation}
In particular, this means that there is some integer $m$ such that for any finite place $w$ of 
$F$ at which $V$ is unramified, all the eigenvalues of $\Frob_w$ are Weil numbers
of absolute value $(\# \F_w)^{m/2}$. More generally, \eqref{pure} means that for all finite places $w$,  the Frobenius semi-simplification of the Weil--Deligne representation $WD_w(V)$ 
associated\footnote{For $w\mid p$
this was defined by Fontaine via $p$-adic Hodge theory:
Let
$$D_{pst}(V) = \bigcup_{E/F_w} (V\otimes B_{st})^{G_E},$$
with $E$ running over
all finite extensions of $F_w$ and $B_{st}$ being Fontaine's ring of semistable $p$-adic periods.
This is a free $L\otimes_\Qp F_w^{\ur}$-module of rank two with an induced action of the monodromy operator
$N$ and Frobenius $\varphi$ of $B_{st}$.
The Weil-Deligne representation associated to $V|_{G_{F_w}}$ by Fontaine is
$WD_w(V) = D_{pst}(V)\otimes_{L\otimes_\Qp F_w^{\ur}} \bQp$ (chose any embedding $F_w^\ur\hookrightarrow \bQp$)
with the induced action of $N$. The action of the Weil group $W_{F_w}\subset G_{F_w}$ is defined
by twisting its $L$-linear, $F_{w}^\ur$-semilinear action $r_{sl}$ on $D_{pst}(V)$. An element $g\in W_{F_w}$ acts on $WD_w(V)$
as $r_{sl}(g)\varphi^{\nu(g)}$, where $\nu:W_{F_w} \rightarrow \Z$ is the normalized valuation map.}
with $V|_{G_{F_w}}$ is pure of weight $m$ in the sense defined by Taylor and Yoshida \cite{tayloryoshida:local-global}.

Let $\cO\subset L$ be the ring of integers of $L$. We will generally choose a $G_F$-stable $\cO$-lattice $T\subset V$
and let $W = V/T$. The latter is a discrete $\cO$-divisible $G_F$-module; it is canonically identified with 
$T\otimes_\cO L/\cO$. Note that the isomorphism class of $T$,  and hence also that of $W$, is not necessarily
uniquely determined. In what follows we will often fix such a triple $(V,T,W)$.

Let $\grm\subset \cO$ be the maximal ideal and $\kap = \cO/\grm$ the residue field. Let $\overline{V}$ be the 
semi-simplification of the finite-dimensional $\kap$-representation $T/\grm T$. Then $\overline{V}$ is uniquely-determined
up to isomorphism; it is independent of the choice of $T$. Furthermore, if $\overline{V}$ is irreducible, then all lattices
$T\subset V$ are homothetic (and so the isomorphism of class of $T$ is unique). Note that while our definition of $\overline{V}$
commutes with extension of the scalars $\kap$, the property of being irreducible may not. 

\subsection{Local conditions}
Let $F$ be a number field and let $(V,T,W)$ be as in Section~\ref{subsec:galoisreps}. 
Let $M$ be an $\cO$-module with a continuous $G_F$-action. 
By a local condition for $M$ at a place $w$ of $F$ we mean a subgroup of the local cohomology group $\H^1(F_w, M)$. 
We discuss several local conditions that will be used throughout. 

\subsubsection{The unramified local condition}
For a finite place $w$ of $F$, the unramified local condition is defined as 
$$
\H^1_{\ur}(F_w, M) = \ker \{ \H^1(F_w,  M) \ra \H^1(I_w,  M)\},
$$ 
where $I_w \subset G_{F_w}$ is the inertia group at $w$. Note that we also have 
$$
\H^1_{\ur}(F_w,M) = \H^1(\F_w,M^{I_w}).
$$
  
\subsubsection{The finite local condition} \label{subsubsec:finite}
Following Bloch and Kato \cite{bloch-kato}, for a finite place $w$ of $F$ we define the finite local condition for $V$ to be
$$
\H^1_{f}(F_w, V) = 
\begin{cases}
\H_{\ur}^1(F_w, V) & w \nmid p \infty \\
\ker \{ \H^1(F_w, V) \ra \H^1(F_w, V \otimes_{\Q_p} B_{\cris}) \} & w \mid p \\ 
0 & w \mid \infty.
\end{cases}
$$

The finite local condition for $V$ can be propagated to $T$ and $W$ via the exact sequence $0 \ra T \ra V \ra W \ra 0$. The resulting local conditions are
$$
\H^1_{f}(F_w, T) = \text{preimage of $\H^1_{f}(F_w, V)$ under the map $\H^1(F_w, T) \ra \H^1(F_w, V)$}
$$
and 
$$
\H^1_{f}(F_w, W) = \text{the image of $\H^1_{f}(F_w, V)$ under the map $\H^1(F_w, V) \ra \H^1(F_v, W).$}
$$  
Note that for $w\nmid p$ we have 
$$
\H^1_{\ur}(F_w,T) \subset \H^1_f(F_w,T) \ \ \ \text{and} \ \ \ \H^1_f(F_w,W)\subset H^1_{\ur}(F_w,W)
$$
and that neither inclusion need be an equality.

We note for later use that if $\dim_L V = 2$ and $V$ is pure of weight different from~0 and~1, then for $w\nmid p$, $\H^1_f(F_w,V) = \H^1_\ur(F_w,V) = 0$, and so
$\H^1_f(F_w,W)=0$ and $\H^1_f(F_w,T)=\H^1_f(F_w,T)_\tor$. In fact, in this case we have $\H^0(F_w,V) = 0 = \H^0(F_w,V^\vee(1)) = 0$,
so $\H^1(F_w,V) = 0$ and hence $\H^1(F_w,T)$ and $\H^1(F_w,W)$ both have finite order.

\subsubsection{The anticyclotomic local condition}\label{subsubsec:ac}
Suppose now that $F = \cK$ is an imaginary quadratic field. Suppose also that $p$ splits 
in $\cK$ as $p = v \vbar$. We define the anticyclotomic local condition for $V$ as 
$$
\H^1_{\ac}(\cK_w, V) = 
\begin{cases}
\H^1(\cK_{\vbar}, V) & \text{ if } w = \vbar, \\ 
\H^1_{f}(\cK_w, V) & \text{ if } w \nmid p\infty \text{ is split in } \cK, \\
0 & \text{ else.}  
\end{cases}
$$  
Note that this definition involves the choice of a prime $v$ of $\cK$ above $p$.

We propagate the anticyclotomic local condition via $0\rightarrow T\rightarrow V\rightarrow W\rightarrow 0$, getting anticyclotomic local conditions for $T$ and $W$.
In particular,
$$
\H^1_{\ac}(\cK_w, W) = 
\begin{cases}
\H^1(\cK_{\vbar}, W)_{\div} & \text{ if } w = \vbar, \\ 
\H^1_{f}(\cK_w, W) & \text{ if } w \nmid p\infty \text{ is split in } \cK, \\
0 & \text{ else.}  
\end{cases}
$$  

As noted in Section~\ref{subsubsec:finite}, if $\dim_LV=2$ and $V$ is pure of weight different from $0$ or $1$, then 
the conditions at $w\nmid p\infty$ agree in the split and non-split cases (the local condition is just $0$).

\subsection{Selmer structures and Selmer modules}
Following \cite[Ch.2]{mazur-rubin:kolyvagin_systems},  a {Selmer structure} $\cF$ on $M$ is a choice 
of a local condition $\H^1_{\cF}(F_w, M) \subseteq \H^1(F_w, M)$ for each place $w$ of $F$ such that for all but finitely many $w$, $\H^1_{\cF}(F_w, M) = \H^1_{\ur}(F_w, M)$. 
A Selmer structure $\cF$ on $M$ has an associated {Selmer module} defined as 
$$
\H^1_{\cF}(F, M) := \ker \left \{\H^1(F, M) \ra \bigoplus_{w}\H^1(F_w, M) / 
\H^1_{\cF}(F_w, M) \right \}, 
$$
where the sum is taken over all places $w$ of $F$. 

If $\cF$ is a Selmer structure on $M$, then we define the dual Selmer structure $\cF^*$ on 
$$
M^* = \Hom_{\cont}(M,\Q_p/\Zp(1)),
$$
the arithmetic dual of $M$, as
$$
\H^1_{\cF^*}(F_w,M^*) = \text{ the annihilator of $\H^1_{\cF}(F_w,M)$ via local duality.}
$$

For the purpose of this paper, if $S$ and $S'$ are two finite sets of finite places of $\cK$ for which $S \cap S' = \emptyset$ then 
$\cF^{S'}_S$ will denote the Selmer structure obtained from $\cF$ by replacing the local conditions at the places in $S$ with the trivial local conditions and the local conditions at the places in $S'$ with the relaxed local conditions (i.e., $\H^1_{\cF^{S'}_S}(\cK_w, M) = \H^1(\cK_w, M)$ for $w \in S'$ and $\H^1_{\cF^{S'}_S}(\cK_w, M) = 0$ for $w \in S$).

\subsubsection{A consequence of Poitou-Tate duality}
If $\cF$ and $\cG$ are 
two Selmer structures on $M$, we write $\cF \preceq \cG$ 
if $\H^1_{\cF}(F_w, M) \subseteq \H^1_{\cG}(F_w, M)$ for every place $w$ of 
$F$. If $\cF \preceq \cG$, there is a perfect bilinear pairing 
$$
\frac{\H^1_{\cG}(F_w, M)}{\H^1_{\cF}(F_w, M)} \times 
\frac{\H^1_{\cF^*}(F_w, M^*)}{\H^1_{\cG^*}(F_w, M^*)} \ra \Q /\Z
$$
that is induced from the Tate local pairing. 

The following theorem is a consequence of the Poitou--Tate global duality theorem (see~\cite[Thm.1.7.3]{rubin:book}, \cite[Thm.I.4.10]{milne:duality} 
and~\cite[Thm.3.1]{tate:duality}):  

\begin{thm}\label{thm:poitou-tate}
Let $\cF \preceq \cG$ be two Selmer structures on $M$ and consider the exact 
sequences 
$$
0 \ra \H^1_{\cF}(F, M) \hra \H^1_{\cG}(F, M) \xra{\loc^{\cG}_{\cF}} \left ( \bigoplus_{w} 
 \frac{\H^1_{\cG}(F_w, M)}{\H^1_{\cF}(F_w, M)} \right )
$$
and 
$$
0 \ra \H^1_{\cG^*}(F, M^*) \hra \H^1_{\cF^*}(F, M^*) \xra{\loc^{\cF^*}_{\cG^*}} \left (  \bigoplus_{w} 
\frac{\H^1_{\cF^*}(F_w, M^*)}{\H^1_{\cG^*}(F_w, M^*)} \right ),  
$$ 
where $\loc^{\cG}_{\cF}$ and 
$\loc^{\cF^*}_{\cG^*}$ are the natural restriction maps 
and the sum is over all places $w$, for which $\H^1_{\cF}(F_w, M) \subsetneq 
\H^1_{\cG}(F_w, M)$. The images of 
$\loc^{\cG}_{\cF}$ and $\loc^{\cF^*}_{\cG^*}$ are orthogonal complements with respect to the pairing
$\ds \sum_w \langle -, -\rangle_w$ obtained from the local Tate pairings on the local cohomology groups.     
\end{thm}

\subsubsection{The finite and anticyclotomic Selmer structures}\label{subsubsec:selstruct}
Let $(T,V,W)$ be as in Section \ref{subsec:galoisreps} and let $M$ be one of $T$, $V$, or $W$. The finite (or Bloch-Kato) Selmer structure $\cF_{\BK}$ is defined by the finite local conditions
$$
\H^1_{\cF_{\BK}}(F_w,M) = \H^1_f(F_w,M).
$$
Note that $\cF_{\BK}^*$ is just the finite Selmer structure on $M^*$. If $F=\cK$ is an imaginary quadratic field in which $p$ splits as $p = v\bar v$, then the anticyclotomic Selmer structure $\SelMin$ is defined by the anticyclotomic local conditions
$$
\H^1_{\SelMin}(\cK_w,M) = \H^1_\ac(\cK_w,M).
$$

\subsubsection{Iwasawa-theoretic Selmer structures.}\label{subsubsec:iwasawa-selstruct}
Let $\cK$ be an imaginary quadratic field and let $F_\infty$ denote either the anticyclotomic $\Z_p$-extension or the $\Z_p^2$-extension of $\cK$. 
As before, we assume that $p$ splits in $\cK$, i.e., $p = v \vbar$. Let $R = \cO[[\Gal(F_\infty / \cK)]]$ be the associated Iwasawa algebra
and consider the $R$-module $M = T \otimes_\cO \widehat{R}$ where 
$\widehat{R} = \Hom_{\cO,\cont}(R, L/\cO)$. The module $M$ is equipped with a $G_{\cK}$-action given by $\rho \otimes \Psi^{-1}$ where $\Psi \colon G_{\cK} \ra R^\times$ is the character naturally 
defined by the projection $G_{\cK} \ra \Gal(F_\infty / \cK)$.  

We define two Selmer structures $\SelMin$ and $\SelGr$ on $M$ that we refer to as the anticyclotomic and the Greenberg Selmer structures, respectively. The anticyclotomic Selmer 
structure $\H^{1}_{\SelMin}(\cK_w, M)$ is defined by
\begin{equation*}
\H^{1}_{\SelMin}(\cK_w, M) = 
\begin{cases}
\H^1(\cK_{\vbar}, M) & \text{ if } w = \vbar \\ 
\H^1_{\ur}(\cK_w, M) & \text{ if } w \nmid p\infty \text{ is split}, \\
0 & \text{ else}.  
\end{cases}
\end{equation*}
The Greenberg Selmer structure is defined by
\begin{equation*}
\H^{1}_{\SelGr}(\cK_w, M) = 
\begin{cases}
\H^1(\cK_{\vbar}, M) & \text{ if } w = \bar v, \\
\H^1_{\ur}(\cK_w, M) & \text{ else}.   
\end{cases}
\end{equation*}

\begin{rmk} When $F_\infty$ is the anticyclotomic $\Zp$-extension of $\cK$, 
the anticyclotomic Selmer structure gives rise to a global Selmer module 
$\H^1_{\SelMin}(\cK, M)$ whose Pontrjagin dual $X_{\ac}(M) = \Hom_{\cO} (\H^1_{\SelMin}(\cK, M), L / \cO)$ appears in the statement of the anticyclotomic main conjecture that will be relevant to our argument (see Section~\ref{subsec:mainconj} for more details).
\end{rmk}

\begin{rmk}
For the particular modules $M=T\otimes_\cO \widehat{R}$ that we will consider later, the $G_\cK$-module $T\otimes_\cO R$ can be viewed as 
a $p$-adic family of motivic $p$-adic Galois representations. Moreover, the inductions of these representations to $G_\Q$ satisfy
the Panchishkin condition at $p$ (see \cite[p.~211]{greenberg:defmotives}). In this case, the Greenberg Selmer condition above can be identified with a special case of the Selmer conditions
defined by Greenberg \cite{greenberg:defmotives}. When $F_\infty$ is the $\Z_p^2$-extension of $\cK$, the Iwasawa--Greenberg conjecture \cite[Conj.~4.1]{greenberg:defmotives}
relates the characteristic ideal of $X_{\Gr}(M) = \Hom_{\cO}(\H^1_{\Gr}(\cK,M),L/\cO)$ to the two-variable $p$-adic $L$-function constructed by Hida \cite{hida:padic2} (see Section~\ref{subsec:twovarpadicLfunc} for details). 
\end{rmk}

\subsection{Newforms, their Galois representations, and modular abelian varieties}\label{subsec:new-galrep-modabvar}
We introduce notation for and recall basic properties of newforms and their associated $p$-adic Galois representations.

\subsubsection{Newforms}\label{subsubsec:newforms}
Let $f\in S_{2k}(\Gamma_0(N))$ be a normalized newform of weight $2k$, level $N$, and trivial Nebentypus. 
Let $f = \sum_{n=1}^\infty a_n(f) q^n$ be its $q$-expansion at the cusp $\infty$. The Fourier coefficients
$a_n(f)$ generate a finite extension $\Q(f)\subset\C$ of $\Q$. Fix an embedding
$\Q(f)\hookrightarrow \bQp$ and let $L\subset\bQp$ be a finite extension of $\Qp$ containing
the image $\Q(f)$.  Let $\O$ be the ring of integers of $L$.

\subsubsection{Galois representations associated to newforms}\label{subsubsec:modulargalrep}
Associated to $f$, $L$ and the fixed embedding $\Q(f)\hookrightarrow L$ is a
two-dimensional $L$-space $V_f$ with a continuous, absolutely irreducible $L$-linear $G_\Q$-action that is characterized by the equality of $L$-functions\footnote{Following
our convention of using geometric normalizations, the Euler factors of a Galois representation (or a
Weil-Deligne representation) are defined
via the action of geometric Frobenius on inertia invariants; Euler factors at $p$ are defined using the Weil-Deligne
representation associated with $V_f^\vee|_{G_\Qp}$.}:
$$
L(V_f^\vee,s) = L(f,s),
$$
where $V_f^\vee$ is the $L$-dual\footnote{We have adopted the conventions here so if $f$ is associated with
an elliptic curve $E/\Q$, then $V_f$ is just the $\Qp$-Tate module of $E$ (or an extension of scalars thereof).}
of the $G_\Q$-representation $V_f$. This equality of $L$-functions can be refined as follows.
Let $\pi=\otimes_{\ell\leq\infty} \pi_\ell$ be the cuspidal automorphic
representation such that $L(\pi,s-1/2) = L(f,s)$. Then the Frobenius-semisimplification of the
Weil--Deligne representation $WD_\ell(V_f^\vee)$ associated with each local 
Galois representation $V_f^\vee|_{G_\Ql}$ is, after extending scalars from $L$ to $\bQp$ 
and fixing a $\Q(f)$-isomorphism $\C\cong\bQp$, the Weil--Deligne representation 
associated with $\pi_\ell\otimes|\cdot|_{\ell}^{-1/2}$ via the local Langlands correspondence.
An important feature of the Galois representation $V_f$ is that it is geometric and pure 
of weight $1-2k$. 

\subsubsection{Modular abelian varieties.}\label{subsubsec:modabvar}
Associated to a newform $f \in S_2(\Gamma_0(N))$ of weight $2$ is an isogeny class of abelian varieties
whose endomorphism rings contain an order in the ring of integers $\Z(f)$ of $\Q(f)$.
Let $A_f$ be an abelian variety in this isogeny class such that $\Z(f)\hookrightarrow \End_\Q A_f$ (such an $A_f$ always exists). 
The $p$-adic Tate module $T_pA_f$ is a free $\Z(f)\otimes\Zp$-module of rank two.
Let $\grp$ be a prime of $\Z(f)$ containing $p$ and let $T_\grp A_f = T_pA_f\otimes_{\Z(f)\otimes\Zp}\Z(f)_\grp$;
this is the $\grp$-adic Tate-module of $A_f$. Let $V_\grp A_f = T_\grp A_f\otimes_{\Z(f)_\grp} \Q(f)_\grp$. The
quotient $V_\grp A_f/T_\grp A_f\cong T_\grp A_f\otimes_{\Z(f)_\grp}\Q(f)_\grp/\Z(f)_\grp$
is naturally identified with $A_f[\grp^\infty]$; this identifies $A_f[\grp^n]$ with $T_\grp A_f/\grp^nT_\grp A_f$
for each $n\geq 1$. For $L=\Q(f)_\grp$ and $\O=\Z(f)_\grp$, $V_\grp A_f$ is just the $V_f$ of 
Section \ref{subsubsec:modulargalrep}, and 
$(V_\grp A_f, T_\grp A_f, A_f[\grp^\infty])$
is an example of a triple $(V,T,W)$ as in Section \ref{subsec:galoisreps}. 

\subsubsection{Selmer groups of newforms and modular abelian varieties}
Let $f\in S_{2k}(\Gamma_0(N)$ be a newform and $V_f$ an associated $p$-adic Galois representation
as in Section \ref{subsubsec:modulargalrep}.
Let $T_f\subset V_f$ be an $G_\Q$-stable $\cO$-lattice and let $W_f = V_f/T_f$.

For a number field $F$, let $L(V^\vee_f/F,s)$ be the $L$-function of the $G_F$-representation $V^\vee_f$. 
Then $L(V^\vee_f/\Q,s)$ is just the usual $L$-function $L(f,s)$, and more generally $L(V^\vee_f/F,s)$ is the value
at $s-1/2$ of the $L$-function of the formal base change to $\GL_2/F$ of the automorphic representation $\pi$ of 
Section \ref{subsubsec:modulargalrep}. The Bloch--Kato conjectures connect the central value
$L(V^\vee_f/F,k)$ with the order of the Selmer module\footnote{Note the Tate-twist: if $V = V_f(1-k)$, then $L(V^\vee_f/F,s) = L(V^*/F,s-k)$, 
so the central critical value is just $L(V^*,0)$.}
$\H^1_{\cF_{\BK}}(F,W_f(1-k))$.

Suppose that $f$ has weight 2 and $(V_f,T_f,W_f) = (V_\grp A_f, T_\grp A_f, A_f[\grp^\infty])$. Then 
for any number field $F$, $\H^1_{\cF_\BK}(F,W_f)$ is just the usual $\grp$-adic Selmer group of $A_f$:
\begin{equation}
\H^1_{\cF_\BK}(F,W_f) = \Sel_{\grp}(A_f/F) = \Sel_p(A_f/F)\otimes_{\Z(f)\otimes \Zp}\Z(f)_\grp.
\end{equation}
Here $\Sel_p(A_f/F)$ is the usual $p$-adic Selmer group of $A_f/F$ and 
$\Sel_\grp(A_f/F)$ is the usual $\grp$-adic Selmer group of $A_f/F$.

%
%
\section{An Anticyclotomic Control Theorem}\label{bsd-control}
Let $\cK/\Q$ be an imaginary quadratic field such that $p$ splits in $\cK$:
\begin{equation}\label{split}
p = v\bar v. \tag{{split}}
\end{equation}
Let $\tau\in \Gal(\cK/\Q)$ be the nontrivial automorphism.
Let $\cK_\infty$ be the anticyclotomic $\Zp$-extension of $\cK$ and let $\Gamma=\Gal(\cK_\infty/\cK)$.

Let $(V,T,W)$, and so also $L$ and $\cO$, be as in Section \ref{subsec:galoisreps}.
Let $\Lambda = \cO[\![\Gamma]\!]$ and put 
$$
M = T \otimes_{\cO} \widehat{\Lambda}, \ \ \widehat{\Lambda} = \Hom_{\cont}(\Lambda, \Q_p/\Z_p).
$$ 
We equip $M$ with an action of $G_{\cK}$ via $\rho \otimes \Psi^{-1}$ where the projection $\Psi \colon G_\cK \ra \Gamma$ is viewed as a continuous
 $\Lambda^\times$-valued character of $G_\cK$. 
Given a finite set $\Sigma$ of finite places $w \nmid p$, let 
$$
\ds X_{\ac}^\Sigma(M) = \Hom_{\cO} \left (\H^1_{\SelMin^\Sigma}(\cK, M), L / \cO \right ).
$$

We begin this section  by listing a few assumptions on the Galois representation $V$ and the related Selmer modules that will be 
assumed to be in force in all that follows. Then under these assumptions we relate the order of the Selmer module $\H^1_{\SelMin}(\cK, W)$ to the order of the Shafarevich--Tate group 
$\Sha_\BK(W/\cK)$, defined \`a la Bloch--Kato as 
$$
\Sha_{\BK}(W/\cK) = \H^1_{\FBK}(\cK, W) / \H^1_{\FBK}(\cK, W)_{\ddiv}.
$$ 
We then prove a control theorem providing a connection between the order of $\H^1_{\SelMin}(\cK, W)$ and the 
characteristic ideal of $X_{\ac}^\Sigma(M)$. The latter is linked to $p$-adic $L$-functions via the anticyclotomic main conjectures discussed in Section~\ref{bsd-acmc}.
Finally, we deduce some consequences for the Selmer groups associated with modular forms and modular abelian varieties.

\subsection{A few assumptions on $(V,T,W)$}  In addition to \eqref{geom} and \eqref{pure} we will assume that 
\begin{equation}\label{semistable}
\text{$V$ is semistable as a representation of $G_{\cK_w}$ for all $w\mid p$}
\tag{{sst}}
\end{equation}
and that
\begin{equation}\label{cdual}
V^*\cong V^\tau,
\tag{{$\tau$-dual}}
\end{equation}
where $V^\tau$ denotes the representation with the same underlying space as $V$ but with the $G_\cK$-action
composed with conjugation by (a lift of) $\tau$. 
This last hypothesis forces the weight of $V$ to be $-1$ (that is, $V$ is pure of weight $-1$). 
To slightly simplify some arguments we will additionally assume that
\begin{equation}\label{two dim}
\dim_L V = 2,
\tag{{$2$-dim}}
\end{equation}
that
\begin{equation}\label{HT}
\text{no non-zero Hodge-Tate weight of $V$ is $\equiv 0\,(\mathrm{mod}\, {p-1})$},
\tag{{HT}}
\end{equation}
and that
\begin{equation}\label{irredK}
\text{$\overline{V}$ is an irreducible $\kappa$-representation of $G_{\cK}$.}
\tag{{irred}${}_{\cK}$}
\end{equation}
We assume furthermore that
\begin{equation}\label{crk1}
\H^1_{\FBK}(\cK, W)_\div \cong L/\cO  \ \ \text{and} \ \ \H^1_f(\cK_w,W)\cong L/\cO,  \ w\mid p,
\tag{{corank 1}}
\end{equation}
and
\begin{equation}\label{surjp}
\H^1_{\FBK}(\cK, W)_{\div} \stackrel{\loc_w}{\twoheadrightarrow} H^1_f(\cK_w, W), \ w\mid p.
\tag{{sur}}
\end{equation}

\subsection{Relating $\H^1_{\SelMin}(\cK, W)$ to $\Sha_{\BK}(W/\cK)$}
We will prove the following:

\begin{prop} \label{Sel-Sha prop}
One has 
$$
\#\H^1_{\SelMin}(\cK, W) = \#\Sha_\BK(W/\cK)\cdot (\#\delta_v)^2,
$$
where $\delta_v = \coker\{\H^1_{\FBK}(\cK, T)\stackrel{\loc_v}{\rightarrow} \H^1_f(\cK_v,T)/\H^1(\cK_v,T)_{\tor}\}$.
In particular, $\H^1_{\SelMin}(\cK, W)$ has finite order.
\end{prop}

\begin{proof}
Let $\cF = \cF_{\BK}$. Consider the exact sequence 
\begin{equation}\label{eq:bk-bkv}
0 \ra \H^1_{\cF_v}(\cK, W) \ra \H^1_{\cF}(\cK, W) \ra  \H^1_f(\cK_v, W)
\end{equation}
and the dual exact sequence 
\begin{equation}\label{eq:bk-bkv-dual}
0 \ra \H^1_{\cF}(\cK, T^*) \ra \H^1_{\cF^v}(\cK, T^*) \ra \H^1(\cK_v, T) / \H^1_f(\cK_v, T^*).
\end{equation}
By the assumption \eqref{surjp}, \eqref{eq:bk-bkv} is surjective on the right. It  
then follows from Theorem~\ref{thm:poitou-tate} that the image of the map 
$\H^1_{\cF^v}(\cK, W^*) \ra \H^1(\cK_v, W^*) / \H^1_f(\cK_v, W^*)$ in \eqref{eq:bk-bkv-dual}
is $0$, and so
\begin{equation}\label{eq:h1T=h1vT}
\H^1_{\cF}(\cK, W^*) = \H^1_{\cF^v}(\cK, W^*).
\end{equation}

By \eqref{eq:bk-bkv} and \eqref{surjp} we have a map of short exact sequences

$$
\xymatrix{
0 \ar[r] & \H^1_{\cF_v}(\cK,W)\cap \H^1_\cF(\cK,W)_{\ddiv} \ar[r] \ar[d] & \H^1_\cF(\cK,W)_{\ddiv} \ar[r] \ar[d] & \H^1_f(\cK_v,W) \ar[d]^{=} \ar[r] &  0 \\
0  \ar[r] & \H^1_{\cF_v}(\cK,W) \ar[r] & \H^1_\cF(\cK,W) \ar[r] & \H^1_f(\cK_v,W) \ar[r] & 0 
} 
$$
As $\H^1_{\cF_v}(\cK,W)\cap \H^1_\cF(\cK,W)_{\ddiv} = \ker\{\H^1_\cF(\cK,W)_{\ddiv}\ra \H^1_f(\cK_v,W)\}$,
by applying the snake lemma to the preceding diagram we conclude that 
\begin{equation}\label{eq:v-sha}\begin{split}
\# \H^1_{\cF_v}(\cK, W) & = \# \Sha_{\BK}(W /\cK) \cdot 
\# \ker \{ \H^1_{\cF}(\cK, W)_{\ddiv} \ra \H^1_f(\cK_v, W)\} \\
& = \# \Sha_{\BK}(W /\cK) \cdot \# \delta_v. 
\end{split}
\end{equation}
The last equality follows upon tensoring the short exact sequence
$$
0 \ra \H^1_\cF(\cK,T)/\H^1_\cF(\cK,T)_\tor \ra \H^1_f(\cK_v,T)/\H^1(\cK_v,T)_\tor  \rightarrow \delta_v\rightarrow 0
$$
with $L/\cO$.

It follows from $V$ being pure of weight different from $0$ or $1$ that we have an exact sequence
\begin{equation*}
0 \ra \H^1_{\cF_v}(\cK, W) \ra \H^1_{\SelMin}(\cK, W) \xra{\alpha} \H^1(\cK_{\vbar}, W)_{\ddiv} / \H^1_f(\cK_{\vbar}, W) 
\end{equation*}
(cf.~Section \ref{subsubsec:finite}) with dual exact sequence 
\begin{equation*}
0 \ra \H^1_{(\SelMin)^*}(\cK, W^*) \ra \H^1_{\cF^v}(\cK, W^*) \xra{\beta} \H^1_f(\cK_{\vbar}, W^*) / \H^1(\cK_{\vbar}, W^*)_{\tor}.  
\end{equation*}
We then have 
\begin{equation}\label{eq:ac-bk}
\# \H^1_{\SelMin}(\cK, W) = \# \ker (\alpha) \cdot \# \textrm{im}(\alpha) = \# \ker (\alpha) \cdot \# \coker(\beta),  
\end{equation}
where the second equality follows from Theorem~\ref{thm:poitou-tate}. It then follows from 
\eqref{eq:h1T=h1vT} that 
$$
\coker(\beta) = \coker \{ \H^1_{\cF}(\cK, W^*) \ra \H^1_f(\cK_{\vbar}, W^*) / \H^1(\cK_{\vbar}, W^*)_{\tor}\}.
$$ 
The hypotheses \eqref{cdual} and \eqref{irredK} imply that $W^*\cong T^\tau$,
so conjugating by the automorphism $\tau$ identifies $\coker(\beta)$
with 
$$
\coker\{\H^1_{\cF}(\cK, T) \ra \H^1_f(\cK_{v}, T) / \H^1(\cK_{v}, T)_{\tor}\} = \delta_v.
$$
Hence combining \eqref{eq:v-sha} and \eqref{eq:ac-bk} yields 
\begin{equation*}\label{eq:ac-sha}
\# \H^1_{\SelMin}(\cK, W) = \# \ker (\alpha) \cdot \# \coker{\beta} = \# \H^1_{\cF_v}(\cK, W) \cdot \# \delta_v = 
\# \Sha_{\BK}(W /\cK) \cdot (\# \delta_v)^2. 
\end{equation*}
It follows from \eqref{crk1} and \eqref{surjp} that $\delta_v$ has finite order, whence $\H^1_{\SelMin}(\cK, W)$
has finite order.
\end{proof}

\begin{rmk}\label{rmk:irredK}
(a) Note that neither of the assumptions \eqref{semistable} and \eqref{HT} is used in this proof.
(b) Clearly, \eqref{irredK} is not essential to the proof of the finiteness of $\H^1_{\SelMin}(\cK,W)$.
It is only used to ensure that $W^* \cong T^\tau$, so that $\delta_v$ can be identified with the cokernel $\coker \{ \H^1_{\cF}(\cK, W^*) \ra \H^1_f(\cK_{\vbar}, W^*) / \H^1(\cK_{\vbar}, W^*)_{\tor}\}$.
Under certain circumstances, such as if $T= T_\grp A_f$, we have $W^*\cong T^\tau$ without
assuming \eqref{irredK}.
\end{rmk}

\subsection{The anticyclotomic control theorem}
Let $S$ be a finite set of places of $\cK$ including all those at which $V$ is ramified and let $S_p\subset S$ be the subset of those not dividing $p$. 
Let $\Sigma\subset S_p$.
Fix a topological generator $\gamma\in \Gamma$. We identify $\cO[\![T]\!]$ with $\Lambda=\O[\![\Gamma]\!]$ via the continuous
$\cO$-algebra map sending $1+T\mapsto \gamma$.
We will prove the following theorem:  

\begin{thm}[Anticyclotomic Control Theorem]\label{thm:control}
The $\Lambda$-module $\Xac^{\Sigma}(M)$ is $\Lambda$-torsion, and if $\fac^\Sigma(T)$ is a generator of its characteristic $\Lambda$-ideal 
$\chr (\Xac^\Sigma(M))$, then 
$$
\#\cO/\fac^{\Sigma}(0) = \# \H^1_{\SelMin}(\cK, W) \cdot C^\Sigma(W),  
$$
where 
$$
C^\Sigma(W) = \# \H^0(\cK_v, W) \cdot \# \H^0(\cK_{\overline{v}}, W) \cdot \prod_{\substack{w \in S_p\backslash \Sigma \\ w \text{ split }}} 
\# \H^1_{\ur}(\cK_w, W)  \cdot \prod_{w \in \Sigma}\# \H^1(\cK_w, W). 
$$
\end{thm}

\noindent Our proof of Theorem~\ref{thm:control} follows the arguments of Greenberg in \cite[\S 4]{greenberg:iwasawa-elliptic}. 

\subsubsection{Surjectivity of the localization maps}\label{subsubsec:surjloc}
For a finite set $S'$ of finite places of $\cK$ let 
Let
$$
\cP_{\SelMin}(M; S') =  \prod_{w \in S'} \frac{\H^1(\cK_{w},M)}{\H^1_{\SelMin}(\cK_{w}, M)} \qquad \text{ and } \qquad 
\cP_{\SelMin}(W; S') = \prod_{w \in S'} \frac{\H^1(\cK_{w},W)}{\H^1_{\SelMin}(\cK_{w}, W)}. 
$$
The key to our result relating $\H^1_{\SelMin^S}(\cK_\infty, M)^{\Gamma}$ to $\H^1(\cK^S / \cK, W)$ is understanding the images of 
\begin{equation}\label{restrict-eq1}
\H^1(\cK^S/\cK,M) \stackrel{\loc_S}{\longrightarrow}
\cP_{\SelMin}(M; S).  
\end{equation}
and
\begin{equation}\label{restrict-eq2}
\H^1(\cK^S / \cK,W) \stackrel{\loc_S}{\longrightarrow}
\cP_{\SelMin}(W; S).
\end{equation}
Here $\cK^S/\cK$ is the maximal extension unramified at all finite places not in $S$.

\begin{prop}\label{Selmersurjprop}
The restriction maps \eqref{restrict-eq1} and \eqref{restrict-eq2} are surjective.
\end{prop}

\begin{proof}
Let $M^* = M^\vee(1)$, which is just $T\otimes_\O\Lambda$ with $G_{\cK}$-action given by $\rho_f\otimes\Psi$.
Recall that 
$$
\H^1_{(\SelMin^S)_{\bv}}(\cK, M^*) = \{c\in \H^1_{\SelMin^S}(\cK,M^*) \ : \ \loc_\bv c = 0 \}.
$$
By Theorem \ref{thm:poitou-tate} the dual of the cokernel of \eqref{restrict-eq1} is identified with 
a quotient of  $\H^1_{(\SelMin^S)_{\bv}}(\cK, M^*)$.
Therefore, to prove the desired surjectivity of \eqref{restrict-eq1}, it suffices to show that 
$\H^1_{(\SelMin^S)_{\bv}}(\cK, M^*)=0$.

Note that $M^*/(\gamma-1)M^* \isoarrow T$. We claim that the natural injection
$$
\H^1(\cK^S / \cK ,M^*)/(\gamma-1) \H^1(\cK^S / \cK ,M^*)\hookrightarrow \H^1(\cK^S / \cK,T)
$$
induces an injection
$$
\H^1_{(\SelMin^S)^*}(\cK^S / \cK ,M^*)/(\gamma-1) \H^1_{(\SelMin^S)^*}(\cK^S / \cK ,M^*) \hookrightarrow \H^1_{(\SelMin^S)_\bv}(\cK, T). 
$$
For this, suppose $c\in \H^1_{(\SelMin^S)^*}(\cK^S /\cK ,M^*)$ has trivial image in $\H^1_{(\SelMin^S)_\bv}(\cK, T)$. Then
$c = (\gamma-1)d$ for some $d\in \H^1(\cK^S/\cK ,M^*)$ such that $(\gamma-1)d = 0$ in $\H^1(\cK_\bv,M^*)$.
But the kernel of multiplication by $\gamma-1$ on $\H^1(\cK_\bv,M^*)$ is the image of 
$\H^0(\cK_\bv,T)$. The vanishing of the latter follows from $H^0(\cK_\bv,V)=0$ (which is true as $V$ is pure of 
weight different from $0$ or $1$: $V^{G_K}\neq 0$ would imply that $1$ was an eigenvalue of Frobenius on $WD_\bv(V)^{N=0}$,
which would contradict purity if the weight is not $0$ or $1$).

Next we note that the canonical isomorphism $T\otimes_\O L/\O \cong W$ induces
an injection
$$
\H^1_{(\SelMin^S)_\bv}(\cK, T) \otimes_\O L/\O \hookrightarrow \H^1_{(\SelMin^S)_\bv}(\cK, W).
$$
It follows from Proposition~\ref{Sel-Sha prop} that the right-hand side is finite (reversing the roles
of $v$ and $\bv$ and using that $H^1(K_w,W)$ is finite for $w\nmid p$ (cf.~Section \ref{subsubsec:finite}).  
As \eqref{irredK} implies
$\H^1_{(\SelMin^S)_\bv}(\cK, T)$ is torsion-free, it follows from the finiteness of 
$\H^1_{(\SelMin^S)_\bv}(\cK, W)$ that $\H^1_{(\SelMin^S)_\bv}(\cK, T)=0$. Hence
$$
\H^1_{(\SelMin^S)_\bv}(\cK, M^*)/(\gamma-1)\H^1_{(\SelMin^S)_\bv}(\cK, M^*) = 0,
$$ 
and so, by Nakayama's lemma, $\H^1_{(\SelMin^S)_\bv}(\cK, M^*)=0$.
This completes the proof of the surjectivity of \eqref{restrict-eq1}. The proof of the surjectivity of \eqref{restrict-eq2} is similar: Poitou--Tate duality identifies the cokernel with a quotient of 
$\H^1_{(\SelMin^S)_\bv}(\cK, T)$, which we have already seen to be $0$.
\end{proof}

\subsubsection{Trivial coinvariants}
We now show that $\H^1(\cK^S/ \cK, M)_{\Gamma}$ and $\H^1_{\SelMin^\Sigma}(\cK, M)_\Gamma$ both vanish: 

\begin{lem}\label{lem:trivcoinv}
We have $\H^1(\cK^S/ \cK, M)_{\Gamma} = 0$ and $\H^1_{\SelMin^\Sigma}(\cK, M)_{\Gamma} = 0$. 
\end{lem}

\begin{proof}
The long exact sequence on Galois cohomology associated to the short exact sequence 
$0 \ra W \ra M \xra{\gamma -1} M \ra 0$ yields an injection $\H^1(\cK^S/\cK, M)_\Gamma \hra \H^2(\cK^S / \cK ,W)$.
The local cohomology group $\H^2(\cK_w,W)$ is dual to $\H^0(\cK_w,T)$, and the latter is $0$;
for $w \mid p$ this was explained in the proof of Proposition \ref{Selmersurjprop}, and for $w\nmid p$
this was explained in Section~\ref{subsubsec:finite}. Consequently (using the notation of \cite[\S 4]{milne:duality}),
$$
\H^2(\cK^S / \cK ,W) = \Sha^2_S(\cK, W) = \{ c\in \H^2(\cK^S / \cK ,W) \colon  \loc_w c=0 \ \forall w\in S\}.
$$
By Poitou-Tate duality \cite[Thm.4.10(a)]{milne:duality}, this group is dual to 
$$
\Sha^1_S(\cK, T) = \{ c\in \H^1(\cK^S / \cK ,T) \colon \loc_w c=0 \ \forall w\in S\}, 
$$ 
and the latter is trivial. Indeed, \eqref{irredK} implies that $\Sha^1_S(\cK,T)$ is torsion-free while \eqref{crk1} and \eqref{surjp} imply that 
$\Sha^1_S(\cK,T)$ is torsion and thus, $\H^1(\cK^S/ \cK, M)_{\Gamma} = 0$. 

To show that $\H^1_{\SelMin^\Sigma}(\cK, M)_{\Gamma} = 0$, consider the exact sequence
$$
0 \ra \H^1_{\SelMin^\Sigma}(\cK,M) \ra \H^1(\cK^S/\cK,M) \ra \cP_{\SelMin}(M;S\backslash\Sigma) \ra 0.
$$
The exactness on the right is a consequence of Proposition \ref{Selmersurjprop}. Multiplying by $\gamma-1$, we obtain from the snake lemma 
the exact sequence
$$
\H^1(\cK^S/\cK,W) = \H^1(\cK^S/\cK,M)^\Gamma \rightarrow \cP_{\SelMin}(M;S\backslash\Sigma)^\Gamma
\ra \H^1_{\SelMin^\Sigma}(\cK,M)_{\Gamma} \ra \H^1(\cK^S/\cK,M)_{\Gamma}.
$$
The map $\H^1(\cK^S/\cK,W) \ra \cP_{\SelMin}(M;S\backslash\Sigma)^\Gamma$ is the composite of the restriction map
$\H^1(\cK^S/\cK,W)\ra \cP_{\SelMin}(W;S\backslash\Sigma)$, which is surjective by Proposition \ref{Selmersurjprop}, and
the map $\cP_{\SelMin}(W;S\backslash\Sigma)\ra \cP_{\SelMin}(M;S\backslash\Sigma)^\Gamma$, which is also surjective
(as the maps $\H^1(\cK_w,W)\rightarrow \H^1(\cK_w,M)^\Gamma$ are surjective). It follows that $\H^1_{\SelMin^\Sigma}(\cK,M)_{\Gamma}$ injects
into $\H^1(\cK^S/\cK,M)_{\Gamma}$. But we have already shown the latter to be trivial.
\end{proof}

\subsubsection{Computing $\# \ker(r)$.}\label{subsubsec:local}
We now calculate the order of the kernel of the map
$$
r \colon \cP_{\SelMin}(W; S\backslash \Sigma) \ra \cP_{\SelMin}(M; S \backslash \Sigma)^{\Gamma}.
$$

\begin{prop}\label{prop:kerr} The kernel of $r$ has order
\begin{equation}
\# \ker(r) = \# \H^0(\cK_v, W) \cdot \# \H^0(\cK_{\vbar}, W)\cdot \prod_{\substack{w\in S_p\backslash\Sigma \\ w \text{ split}}} c_{w}^{(p)}(W),
\end{equation}
where $c_{w}^{(p)}(W) := [\H^1_{\ur}(\cK_{w}, W) : \H^1_{f}(K_{w}, W)] = \#\H^1_\ur(\cK_w,W)$ are the $p$-parts of the local Tamagawa numbers. 
\end{prop}

\begin{proof}
Let $w$ be a place of $\cK$. By the long exact sequence on Galois cohomology associated to the short sequence 
$$
0 \ra W \ra M \xra{\times (\gamma -1)} M \ra 0,  
$$
the kernel of the restriction map $r_w:\H^1(\cK_w, W) \ra \H^1(\cK_w, M)^{\Gamma}$ is the image of 
$M^{G_{\cK_w}} / (\gamma -1) M^{G_{\cK_w}}$ under the coboundary map. Let $\ell$ 
be the prime below $w$. Unlike the cyclotomic case over $\Q$ treated in \cite[\S 3]{greenberg:iwasawa-elliptic} 
where every prime $w$ is finitely decomposed in the $\Zp$-extension, 
we need to consider the cases of $\ell$ being split or non-split in $\cK$ separately.   

\smallskip 

\noindent{\bf Case 1(a): $w \nmid p$, $W$ is ramified at $w$ and $\ell$ is split in $\cK$.} 
We have a commutative diagram 
$$
\xymatrix{
0 \ar[r] & \H^1_{\ur}(\cK_w, W) \ar[r] \ar[d] & \H^1(\cK_w, W) \ar[r] \ar[d] & \H^1(I_w, W)^{G_{\cK_w}} \ar[d]^{} \ar[r] &  0 \\
0 \ar[r] & \H^1_{\ur}(\cK_w, M) \ar[r] & \H^1(\cK_w, M) \ar[r] & \H^1(I_w, M) & \\
} 
$$
The kernel of the right-most map is just the image of $(M^{I_w}/(\gamma-1)M^{I_w})^{G_{\cK_w}}$.  
Since $\Psi$ is not ramified at $w$, $M^{I_w}$ is $(\gamma-1)$-divisible, hence this kernel is trivial. 
It follows that the map $\H^1(\cK_w, W) / \H^1_{\ur}(\cK_w, W) \ra \H^1(\cK_w, M) / \H^1_{\ur}(\cK_w, M)$
is injective and hence that the kernel of the map 
$
r_w \colon \H^1(\cK_w, W) / \H^1_f(\cK_w, W) \ra \H^1(\cK_w, M) / \H^1_{\ur}(\cK_w, M),
$
is isomorphic to $\H^1_{\ur}(\cK_w, W) / \H^1_{f}(\cK_w, W)$. But the order of the latter is exactly the $p$-part $c_w^{(p)}(W)$ of the 
Tamagawa number at $w$. 

\smallskip 

\noindent{\bf Case 1(b): $w \nmid p$, $W$ is ramified at $w$ and $\ell$ is not split in $\cK$.} In this case $\ell$ is inert or ramified 
in $\cK$, and by the definition of the anticyclotomic Selmer structure both $\H^1_{\SelMin}(K_w, W)$ and $\H^1_{\SelMin}(K_w, M)$ are trivial. 
Hence 
$\ker(r_w) = \ker\{H^1(\cK_w,W)\twoheadrightarrow H^1(\cK_w,M)[\gamma-1])\}$
and so equals the image of $M^{G_{\cK_w}}/(\gamma-1)M^{G_{\cK_w}}\hookrightarrow H^1(\cK_w,W)$. However, 
since $w$ is not split, $\Psi$ is trivial on $G_{\cK_w}$ and hence, $M^{G_{\cK_w}}$ is divisible by $\gamma-1$. It follows that $\ker(r_w) = 0$. 

\smallskip 

\noindent {\bf Case 2(a): $w \nmid p$, $W$ is unramified at $w$ and $\ell$ is split in $\cK$.} Similarly to Case 1(a), 
we obtain $\ker(r_w) \isom \H^1_{\ur}(K_w, W) / \H^1_f(K_w, W)$ which is trivial since the two local conditions coincide. 

\smallskip

\noindent {\bf Case 2(b): $w \nmid p$, $W$ is unramified at $w$ and $\ell$ is not split in $\cK$.} Exactly the same argument
as employed for Case 1(b) shows that $\ker(r_{w}) = 0$. 

\smallskip 

\noindent {\bf Case 3(a): $w = \bv$.} We have $\ker(r_{\vbar}) = \H^1(\cK_{\vbar}, W) / \H^1(\cK_{\vbar}, W)_{\ddiv}$. 
By Tate local duality $\H^1(\cK_\bv,W)/\H^1(\cK_\bv,W)_\div$ is dual to $\H^1(\cK_\bv,T)_\tor$, which is just
$\ker\{\H^1(\cK_\bv,T)\rightarrow \H^1(\cK_\bv,V)\}\cong \H^0(\cK_\bv,W)/\H^0(\cK_\bv,W)_\div$. But $\H^0(\cK_\bv,V) = 0$ (as noted in the proof of Proposition \ref{Selmersurjprop}, this is a consequence of being pure of
weight different from $0$ or $1$), so $\H^0(\cK_\bv,W)_\div =0$ and $\# \ker(r_{\bv}) = \#\H^0(\cK_\bv,W)$.

\smallskip
\noindent{\bf Case 3(b): $w=v$.}
In this case the map is $r_v \colon \H^1(\cK_v, W) \ra \H^1(\cK_v, M)^\Gamma$ and 
we have 
$$
\ker(r_v) \isom M^{G_{\cK_{v}}}/(\gamma-1)M^{G_{\cK_{v}}}\hookrightarrow \H^1(\cK_{v},W).
$$
Let $P_{v} = \ker\Psi|_{G_{\cK_{v}}}$ and $\Gamma_{v} = G_{\cK_{v}}/P_{v}\hookrightarrow \Gamma$, where $\Psi$ is as in the beginning of Section~\ref{bsd-control}.   
Then $\Gamma_{v}$ has finite index in $\Gamma$, and the image of $I_{v}$ in $\Gamma_{v}$ also has finite index.
Let $\gamma_{v} \in \Gamma_{v}$ be a topological generator.
Let $T^\vee = \Hom_\Zp(T,\Zp)$ and let $T^\vee_{P_{v}}$ be its $P_{v}$-coinvariants.
Then $M  = T\otimes_\O\widehat{\Lambda} \cong \Hom_{\cont}(T^\vee\otimes_\O\Lambda,\Qp/\Zp)$ and so
$M^{P_{v}} \cong \Hom_{\cont}(T^\vee_{P_{v}}\otimes_\O \Lambda,\Qp/\Zp)$.
If $\#T^\vee_{P^{v}}$ is finite, then $M^{G_{\cK_{v}}}$ is therefore isomorphic to a submodule of
$\Hom_\Zp(T^\vee_{P_{v}}\otimes_\O\Lambda/(\gamma_{v}^{p^t}-1),\Qp/\Zp)$, which has finite order.
Here $t>0$ is such that $\gamma_{v}^{p^t}$ acts trivially on~$T^\vee_{P_{v}}$. This shows that
$M^{G_{\cK_{v}}}$ has finite order if $\#T^\vee_{P_{v}}$ is finite, in which case
$$
\#M^{G_{\cK_{v}}}/(\gamma-1)M^{G_{\cK_{v}}} = \#M^{G_{\cK_{v}}}[\gamma-1] =  \#M[\gamma-1]^{G_{\cK_{v}}} = \#\H^0(\cK_v, W).
$$

It remains to show that $\#T^\vee_{P_{v}}$ is finite, which we will do by arguing by contradiction. 
Assume that $\#T^\vee_{P_{v}}$ is not finite. Then $T^{P_{v}}\neq 0$
and hence $V^{P_v}\neq 0$. As $V$ is two-dimensional and semistable, there are two cases
to consider: (i) $V$ is crystalline and (ii) $V$ is non-crystalline and hence a non-split extension of the form
$0 \ra L(\eps\alpha)\ra V \ra L(\alpha) \ra 0$ with $\alpha$ unramified. 
In case (i), $V^{P_v}$ would have to be, possibly after a finite extension of scalars, a sum of 
one-dimensional crystalline representations of weight $-1$. Such a crystalline character must be of the form $\eps^a\alpha$ for some
integer $a$ and some unramified character $\alpha$, while the condition of being weight $-1$ means that
$\alpha(\Frob_v)$ is a Weil number of absolute value $p^{-1/2+a}$.  
However, since $\eps^a\alpha$ factors through a pro-$p$-group,
so must $\eps^a|_{I_v} = \eps^a\alpha|_{I_v}$. But this only happens if $a\equiv 0 \pmod{p-1}$. As $a$ is a Hodge-Tate
weight of $V$, it follows from \eqref{HT} that we must have $a=0$. Then $\alpha$ must factor through the 
quotient of $\Gamma_v$ by the image of $I_v$. This quotient is finite, so $\alpha$ must have finite order, contradicting
it being a character of weight $-1$. 
On the other hand, if $V$ is as in case (ii), 
then $V$ is also a non-split extension of $P_v$ and so if $V^{P_v}\neq 0$, then again $\eps\alpha$, $\alpha$ unramified,  must factor through $\Gamma_v$,
which we have already seen to be impossible. This contradiction completes the last case and hence proves the lemma.
\end{proof}

\begin{rmk}\label{rm:HT}
The proof of Case 3(b) above is the only place where the hypothesis \eqref{HT} is invoked.
\end{rmk}

\subsubsection{An application of the snake lemma}
There is a commutative diagram
$$
\xymatrix{
0 \ar[r] & \H^1_{\SelMin^{\Sigma}}(\cK, W) \ar[r] \ar[d]^{s} & \H^1(\cK^S/\cK, W) \ar[r] \ar[d]^{h} & \cP_{\SelMin}(W;S\backslash\Sigma) \ar[r] \ar[d]^{r} & 0 \\
0 \ar[r] & \H^1_{\SelMin^{\Sigma}}(\cK, M)^{\Gamma} \ar[r] & \H^1(\cK^S/\cK, M)^{\Gamma} \ar[r] & \cP_{\SelMin}(M;S\backslash\Sigma)^{\Gamma}& \\
} 
$$
(note that exactness of the top row follows from Proposition \ref{Selmersurjprop}), and the snake lemma yields an exact sequence
$$
0 \ra \ker(s) \ra \ker(h) \ra \ker(r) \ra \coker(s) \ra \coker(h) \ra \coker(r). 
$$
However, $\coker(h)=0$ and hence
there is an exact sequence
\begin{equation*}\label{eq:snake2}
0 \ra \ker(s) \ra \ker(h) \ra \ker(r) \ra \coker(s) \ra 0. 
\end{equation*}
It follows that 
$$
\frac{\#\H^1_{\SelMin^{\Sigma}}(\cK, M)^{\Gamma}}{\# \H^1_{\SelMin^{\Sigma}}(\cK, W)} = \frac{\# \coker(s)}{\# \ker(s)} = \frac{\# \ker(r)}{ \# \ker(h)}. 
$$
The order of $\ker(r)$ was computed in Proposition \ref{prop:kerr} while 
$$
\ker(h) = M^{G_\cK}/(\gamma-1)M^{G_\cK} = M^{G_\cK} = \H^0(\cK,W)
$$ 
which vanishes by \eqref{irredK}. It follows that 
\begin{equation}\label{eq:selmin1}
\#\H^1_{\SelMin^\Sigma}(\cK,M)^\Gamma = \#\H^1_{\SelMin^\Sigma}(\cK,W)\cdot
\# \H^0(\cK_v, W) \cdot \# \H^0(\cK_{\vbar}, W)\cdot \prod_{\substack{w\in S_p\backslash\Sigma \\ w \text{ split}}} c_{w}^{(p)}(W).
\end{equation}

It also follows from the surjectivity of \eqref{restrict-eq2} that there is an exact sequence
$$
0 \ra \H^1_{\SelMin}(\cK,W) \ra \H^1_{\SelMin^\Sigma}(\cK,W) \ra \cP_{\SelMin}(W;\Sigma) \ra 0.
$$
As $\H^1_f(\cK_w,W) = 0$ (see Section \ref{subsubsec:finite}), we have
$$
\#\H^1_{\SelMin^\Sigma}(\cK,W) = \#\H^1_{\SelMin}(\cK,W)\cdot \prod_{w\in\Sigma} \#\H^1(\cK_w,W).
$$
Combining this with \eqref{eq:selmin1} yields
\begin{equation}\label{eq:selmin2}
\#\H^1_{\SelMin^\Sigma}(\cK,M)^\Gamma = \#\H^1_{\SelMin}(\cK,W)\cdot C^\Sigma(W).
\end{equation}

\subsubsection{Proof of Theorem \ref{thm:control}: Torsionness of $X^\Sigma_{\ac}(M)$}
It follows from Proposition \ref{Sel-Sha prop} and \eqref{eq:selmin2} that $\H^1_{\SelMin^\Sigma}(\cK,M)^\Gamma$ has finite order.
Hence $X^\Sigma_{\ac}(M)_\Gamma = \Hom_\cO (\H^1_{\SelMin^\Sigma}(\cK,M)^\Gamma, L/\cO)$ has finite order. It follows
easily that $X^\Sigma_\ac(M)$ must therefore be $\Lambda$-torsion.

\subsubsection{Proof of Theorem \ref{thm:control}: determining $\#\cO/\fac^\Sigma(0)$}
To complete the proof of the control theorem, we use the relation between 
$\H^1_{\SelMin}(\cK, W)$ and $\H^1_{\SelMin^{\Sigma}}(\cK, M)^{\Gamma}$ to relate $\# \H^1_{\SelMin}(\cK, W)$ to 
$f_{\ac}(0)$. The key to the comparison is the following proposition on the non-existence of pseudo-null submodules of the 
$\Lambda$-module $X^{\Sigma}_{\ac}(M)$: 

\begin{prop}\label{nopseudonull-prop}
The Selmer module $\H^1_{\SelMin^{\Sigma}}(\cK, M)$ has no proper $\Lambda$-submodule 
of finite order. Equivalently, the $\Lambda$-module $X^{\Sigma}_{\ac}(M)$ has no non-trivial pseudo-null $\Lambda$-submodule. 
\end{prop}

\begin{proof}
Recall that $\H^1_{\SelMin^\Sigma}(\cK, M)_{\Gamma}$ by Lemma~\ref{lem:trivcoinv}. If $X\subset X^{\Sigma}_{\ac}(M)$ is a 
$\Lambda$-submodule of finite order, then its dual $X^*$ is a finite order quotient
of $\H^1_{\SelMin^\Sigma}(\cK, M)$, so $X^*/(\gamma-1)X^*$ is a quotient of $\H^1_{\SelMin^\Sigma}(\cK, M)_\Gamma$ and therefore $0$. 
But $X^*$ is a finite $\Lambda$-module and so, by Nakayama's Lemma, $X^*=0$ and hence $X=0$.
\end{proof}

We can now establish the comparison result: 

\begin{lem}\label{SelWorder-lem}
We have
$$
\#\cO/\fac^\Sigma(0) = \#\Lambda/(T,\fac^\Sigma(T)) = \#\H^1_{\SelMin}(\cK,W)\cdot C^\Sigma(W).
$$
\end{lem}

\begin{proof}
Let $X = X^{\Sigma}_{\ac}(M)$. Then $\#X_{\Gamma} = \#\H^1_{\SelMin^\Sigma}(\cK, M)^{\Gamma}$. Since $X$ is $\Lambda$-torsion, 
$X$ is pseudoisomorphic to a $\Lambda$-module $\ds Y = \prod_{i=1}^r \Lambda/(f_i)$ for distinguished polynomials $f_i$. 
By Proposition~\ref{nopseudonull-prop}, $X$ has no pseudo-null submodule so we obtain an exact sequence 
$0 \ra X \rightarrow Y \rightarrow K \rightarrow 0$, where $K$ is a pseudo-null $\Lambda$-module. 
Applying the snake lemma to the commutative diagram
\[
\xymatrix{
0 \ar[r] & X \ar[d]^{\gamma - 1} \ar[r] & Y \ar[r] \ar[d]^{\gamma -1}& K \ar[r] \ar[d]^{\gamma - 1}& 0 \\
0 \ar[r] & X \ar[r] & Y \ar[r] & K \ar[r] & 0, 
}
\]
we obtain
$$
\#X_{\Gamma} = \#Y_{\Gamma} = \prod\#\Lambda/(T,f_i(T)) = \prod\#\O/(f_i(0)) = \#\Lambda/(T,f_{\ac}^{\Sigma}(T)), 
$$
where $f_{\ac}^{\Sigma}(T) \sim f_1(T) \cdots f_r(T)$ is a generator of the characteristic ideal of $X$. We thus get that 
$$
\#\cO/\fac^\Sigma(0) = \# \H^1_{\SelMin^{\Sigma}}(\cK, W) = \# \H^1_{\SelMin}(\cK, W) \cdot C^{\Sigma}(W), 
$$
which proves the claim. 
\end{proof}

This completes the proof of Theorem \ref{thm:control}.

\subsection{Comparison with Selmer groups over $\Z_p^2$-extensions}\label{subsubsec:bigsel}
In order to deduce what we will need from the existing theorems in Iwasawa theory, it is necessary to consider Selmer groups for Galois
extensions larger than $\cK_\infty/\cK$.  Let $\bK_\infty/\cK$ be the composite of all $\Zp$-extensions of $\cK$ in $\bQ$.  
Let $\Gamma_{\cK} = \Gal(\bK_\infty/\cK) \cong \Z_p^2$. The Galois group $\Gal(\cK/\Q)$ acts
on $\Gamma_{\cK}$, which decomposes under this action as $\Gamma_{\cK} = \Gamma_{\cK}^+\oplus\Gamma_{\cK}^-$ with the superscript $\pm$ denoting the 
subgroup on which $\tau\in\Gal(\cK/\Q)$ acts at $\pm 1$; then $\Gamma_{\cK}^+$ is the kernel of the
natural surjection $\Gamma_{\cK}\twoheadrightarrow \Gamma$. 
Let $\Lambda_{\cK} = \O[\![\Gamma_\cK]\!]$. Let $\Psi_\cK \colon G_\cK\twoheadrightarrow \Gamma_\cK$ be the natural surjection; 
this is also a continuous $\Lambda_\cK^\times$-valued character of $G_\cK$.
Let
$
\CM = T\otimes_\O\Lambda_\cK^*
$
with $G_\cK$-action given by $\rho\otimes\Psi_\cK^{-1}$. Consider the Selmer module 
$\H^1_{\SelGr^{\Sigma}}(\cK, \CM)$ and its Pontrjagin dual $X^{\Sigma}_{\Gr}(\CM) = \Hom_\cO(\H^1_{\SelGr^{\Sigma}}(\cK, \CM),L/\cO)$
which is a finite $\Lambda_\cK$-module. Let
$\chr \left ( X^{\Sigma}_{\Gr}(\CM)  \right) \subset \Lambda_\cK$
be its $\Lambda_{\cK}$-characteristic ideal. Since $\Lambda_{\cK}$ is a UFD, $\ds \chr \left ( X^{\Sigma}_{\Gr}(\CM)  \right)$ is a principal ideal, 
and it is nonzero if and only if $X^{\Sigma}_{\Gr}(\CM)$ is a torsion $\Lambda_{\cK}$-module.

Let $\gamma_\pm\in\Gamma^\pm_{\cK}\cong\Zp$ be fixed topological generators; we assume that $\gamma_-\mapsto \gamma$
under the projection $\Gamma_{\cK}\twoheadrightarrow\Gamma$. We have $M = \CM[\gamma_+-1]$ and, by \eqref{irredK},
$$
\H^1(\cK^S /\cK,M) \isoarrow \H^1(\cK^S/\cK,\CM)[\gamma_+-1],
$$
which induces maps
\begin{equation}\label{bigSelmercontrol-eq}
\H^1_{\SelMin^\Sigma}(\cK,M) \hookrightarrow \H^1_{\SelGr^\Sigma}(\cK,\CM)[\gamma_+-1] \ \ \text{and} \ \  X^{\Sigma}_{\Gr}(\CM)/(\gamma_+-1)X^{\Sigma}_{\Gr}(\CM) \rightarrow X^{\Sigma}_{\ac}(M).
\end{equation}

\begin{lem}\label{bigSelmercontrol-lem}
Suppose $\Sigma$ contains all the finite places $w\nmid p$ at which $V$ is ramified. Then
the maps in \eqref{bigSelmercontrol-eq} have finite cokernel and kernel, respectively.
\end{lem}

\begin{proof}
The identification $M = \CM[\gamma_+-1]$ yields a short exact sequence
$$
\CM^{G_{\cK_{v}}}/(\gamma_+-1)\CM^{G_{\cK_{v}}} \hookrightarrow \H^1(\cK_{v},M)\twoheadrightarrow \H^1(\cK_{v},\CM)[\gamma_+-1].
$$
Since the image of $G_{\cK_\bv}$ in $\Gamma_{\cK}$ has finite index,
the argument used in the Case 3(b) of the proof of Proposition~\ref{prop:kerr} to show that $M^{G_{\cK_{v}}}$ has finite order can be easily 
adapted to prove that $\CM^{G_{\cK_{v}}}$ has finite order by replacing $P_{v}$ with $\ker\Psi_K|_{G_{\cK_{v}}}$. It follows that
$$
\# \CM^{G_{\cK_{v}}}/(\gamma_+-1)\CM^{G_{\cK_{v}}} = \#\CM^{G_{\cK_{v}}}[\gamma_+-1] < \infty.
$$
Since the cokernel of the map $\H^1_{\SelMin^\Sigma}(\cK, M) \hra \H^1_{\SelGr^\Sigma}(\cK, \CM)[\gamma_+ - 1]$ is a quotient of
$\CM^{G_{\cK_{v}}}/(\gamma_+-1)\CM^{G_{\cK_{v}}}$, it is therefore finite, and so, too, is the kernel of
the dual map.
\end{proof}

\begin{coro}\label{bigSelmercontrol-cor}
Suppose $\Sigma$ contains all the finite place $w\nmid p$ at which $V$ is ramified.
Then
$$
\chr \left ( X^{\Sigma}_{\ac}(M)  \right)  \subset \chr \left ( X^{\Sigma}_{\Gr}(\CM)  \right)  \mod\, (\gamma_+-1).
$$
\end{coro}

\begin{proof}
Let $F^\Sigma_{\Gr}(\CM)$ be the $\Lambda_K$-fitting ideal of $X^\Sigma_{\Gr}(\CM)$ and $F^\Sigma_\ac(M)$ the $\Lambda$-fitting ideal of 
$X^\Sigma_\ac(M)$. Then
$F^{\Sigma}_{\Gr}(\CM)\subset \chr \left ( X^{\Sigma}_{\Gr}(\CM)  \right) $. Since the kernel of the surjection map 
$X^{\Sigma}_{\Gr}(\CM)/(\gamma_+-1)X^{\Sigma}_{\Gr}(\CM) \twoheadrightarrow X^{\Sigma}_{\ac}(M)$
has finite order and since the source has $\Lambda$-Fitting ideal equal to $F^{\Sigma}_{\Gr}(\CM)$ modulo $(\gamma_+-1)$,
there is some $c>0$ such that
$F^{\Sigma}_{\ac}(M)\grm_\Lambda^c \subset F^{\Sigma}_{\Gr}(\CM) \mod \, (\gamma_+-1)$,
where $\grm_\Lambda = (\gamma-1,\grm)$ is the maximal ideal of $\Lambda$.
It follows that
$$
F^{\Sigma}_{\ac}(M) \grm_\Lambda^c\subset \chr \left ( X^{\Sigma}_{\Gr}(\CM)  \right)  \mod\, (\gamma_+-1).
$$
Since the right-hand side is a principal ideal, it follows that $F^{\Sigma}_{\ac}(M)$
is contained in $\chr \left ( X^{\Sigma}_{\Gr}(\CM)  \right)  \mod\, (\gamma_+-1)$. And since
$\chr \left ( X^{\Sigma}_{\ac}(M)  \right) $ is the smallest principal ideal containing $F^\Sigma_\ac(M)$,
we must also have that $\chr \left ( X^{\Sigma}_{\ac}(M)  \right) $ is contained in
$\chr \left ( X^{\Sigma}_{\Gr}(\CM)  \right)  \mod\, (\gamma_+-1)$.
\end{proof}

\subsection{Applications to newforms and modular abelian varieties $A_f$}\label{subsec:newform-app}
We return to the notation of Section \ref{subsec:new-galrep-modabvar}. 
Let $f\in S_{2k}(\Gamma_0(N))$ be a newform, and let $(V_f,T_f,W_f)$ be
as in Section \ref{subsubsec:modulargalrep}.

Let
$$
(V,T,W) = (V_f(1-k),T_f(1-k),W_f(1-k)).
$$
Note that $V^*\cong V$.
It is a theorem of Saito \cite{saito}, building on work
of Deligne, Langlands, Carayol, and others, that $V$ is geometric and pure of weight $-1$.
So \eqref{geom}, \eqref{pure}, \eqref{cdual}, and \eqref{two dim} all hold for $V$. 
Furthermore, it follows that if 
\begin{equation}\label{p-sst}
\ord_p(N)\leq 1,
\tag{{$p$-sst}}
\end{equation}
then $V$ is semistable at $p$, that is, \eqref{semistable} also holds. 
As the Hodge-Tate weights of $V$ are $k-1$ and $-k$, \eqref{HT} holds if
and only if neither $k$ nor $k-1$ (when $k\neq 1$) are divisible by $p-1$;
in particular, \eqref{HT} always holds for $2k=2$.

Suppose now that $f$ has weight $2$ and that 
$$
(V,T,W) = (V_f,T_f,W_f) = (V_\grp A_f, T_\grp A_f, A_f[\grp^\infty])
$$
as in Section \ref{subsubsec:modabvar}.
The Mordell--Weil group $A_f(F)$ is a $\Z(f)$-module and
$$
\rank_{\Z} A_f(F) = [\Q(f):\Q]\cdot\rank_\O (A_f(F)\otimes_{\Z(f)}\O).
$$
Suppose
\begin{equation}\label{rank1}
\rank_\O (A_f(\cK)\otimes_{\Z(f)}\O) = 1 \tag{{rank 1}}
\end{equation}
or, equivalently, $\rank_\Z A_f(\cK) = [\Q(f):\Q]$. The Tate-Shafarevich group $\Sha(A_f/\cK)$ is also a $\Z(f)$-module and satisfies
$\Sha(A_f/\cK)[p^\infty]\otimes_{\Z(f)}\cO = \Sha(A_f/\cK)[\grp^\infty]$.
Suppose
\begin{equation}\label{Sha-finite}
\# \Sha(A_f/\cK)[\grp^\infty] < \infty \tag{{$\Sha$ $\grp$-finite}}.
\end{equation}
Note that is this case $\Sha_{BK}(W_f/\cK) = \Sha(A_f/\cK)[\grp^\infty]$.
Under the assumptions \eqref{rank1} and \eqref{Sha-finite}, it follows from the exact sequence
\begin{equation}\label{FES}
A_f(F)\otimes_{\Z(f)}L/\O\hookrightarrow \Sel_{\grp^\infty}(A_f/F)\twoheadrightarrow \Sha(A_f/F)[\grp^\infty],
\tag{{FES}}
\end{equation}
that
\eqref{crk1} holds for $\H^1_{\cF_\BK}(\cK,W_f) = \Sel_{\grp^\infty}(A_f/F)$. 
If furthermore 
\begin{equation}\label{Afirred}
\text{$A_f[\grp]$ is an irreducible $G_\cK$-representation}, 
\tag{{$\grp$-irred}}
\end{equation}
then
$$
A_f(\cK)\otimes_{\Z(f)}\O \isoarrow \H^1_f(\cK,T)\isom \O.
$$

If $w\mid p$ is a place of $\cK$ then $A_f(\cK_w)$ is a finitely generated $\Zp$-module of rank equal to $[\cK_w:\Qp]\cdot[\Z(f):\Z]$ (there
is a natural $\Z(f)$-injection of
$\Z(f)\otimes\O_{\cK_w}$ into the compact Lie group $A_f(\cK_w)$, with image having finite index; this is just the
$\exp$ map from a neighborhood of zero in the tangent space at the origin). In particular, if $p$ splits in $\cK$
(so $\Qp\isoarrow \cK_w)$, we have that $A_f(\cK_w)\otimes_{\Z(f)}L/\O \cong L/\O$.
Furthermore, since $A_f(\cK)\subset A_f(\cK_w)$, if \eqref{rank1} holds, then
we also have that
$$
A(\cK)\otimes_{\Z(f)}L/\O\cong L/\O \twoheadrightarrow A(\cK_w)\otimes_{\Z(f)}L/\O \cong L/\O.
$$
That is, \eqref{surjp} also holds. This shows that if \eqref{split}, \eqref{rank1}, \eqref{Sha-finite}, and \eqref{Afirred} hold,
then so do \eqref{crk1}, \eqref{surjp}, and \eqref{irredK}.

Still assuming that $p$ splits in $\cK$, we also have
$A_f(\cK_v)\otimes_{\Z(f)\otimes\Zp}\O\isoarrow \H^1_f(\cK_v,T)$ is a finite $\O$-module of rank one.
Let $A_f(\cK_v)_{/\tor} = A_f(\cK_v)/A_f(K_v)_\tor$ (this is then a free $\Z(f)\otimes\Zp$-module), so
$A_f(\cK_v)_{/\tor} \otimes_{\Z(f)\otimes\Zp}\O\isoarrow \H^1_f(\cK_v,T)/\H^1(\cK_v,T)_{\tor}$.
Then
\begin{equation}\label{eq:deltav}\begin{split}
\#\delta_v & = \#\coker\{\H^1_f(\cK,T)\stackrel{\loc_v}{\rightarrow} \H^1_f(\cK_v,T)/\H^1(\cK_v,T)_{\tor}\} \\
& = \#[A_f(\cK_v)_{/\tor}\otimes_{\Z(f)\otimes\Zp}\O : A_f(\cK)\otimes_{\Z(f)}\O].
\end{split}
\end{equation}
Let $P\in A_f(\cK)$ be any point of infinite order. The $\O$-module $\O\cdot P \subset A_f(\cK)\otimes_{\Z(f)}\O$
generated by $P$ has finite index, and it follows that
\begin{equation*}\label{delta-eq}
\#\delta_v = \frac{[A_f(\cK_v)_{/\tor}\otimes_{\Z(f)\otimes\Zp}\O : \O\cdot P]}{[A_f(\cK)\otimes_{\Z(f)}\O : \O\cdot P]}.
\end{equation*}
In particular, it then follows from Proposition \ref{Sel-Sha prop}
that 
\begin{equation}\label{eq:modabvar-sha}
\#\H^1_{\SelMin}(\cK,W_f) = \#\Sha_\BK(W_f/\cK) \cdot \frac{[A_f(\cK_v)_{/\tor}\otimes_{\Z(f)\otimes\Zp}\O : \O\cdot P]^2}{[A_f(\cK)\otimes_{\Z(f)}\O : \O\cdot P]^2}.
\end{equation}

Suppose now that 
\begin{equation}\label{good}
p\nmid N.
\tag{good}
\end{equation}
Then $A_f$ has good reduction at $p$ and so extends to an abelian scheme over $\Z_{(p)}$.
Let $A_f^1(\Qp)\subset A_f(\Qp)$ be the kernel of reduction modulo $p$.
Let $\Omega^1(A_f/\Z_p)^\vee=\Hom_\Zp(\Omega^1(A_f/\Zp),\Zp)$. 
The formal group logarithm defines a $\Z(f)\otimes\Zp$-isomorphism $\log:A_f^1(\Qp)\rightarrow p\Omega^1(A_f/\Zp)^\vee$, which extends to an
injective $\Z(f)\otimes\Zp$-homomorphism $\log:A_f(\Qp)_{/\tor} \hra \Omega^1(A_f/\Zp)^\vee$. 
Recall that $\cO = \Z(f)_\grp$ for a chosen prime $\grp\mid p$ of $\Z(f)$.
Let $\omega_f\in \Omega^1(A_f/\Zp)\otimes_\Zp\cO$ be an $\cO$-basis element such that the
action of each $\alpha\in\Z(f)$ on $A_f$ induces $\alpha^*\omega_f = \alpha\cdot\omega_f$ (multiplication
by the scalar $\alpha\in\cO$). Then composition of $\log$ with evaluation on $\omega_f$ defines
an $\cO$-homomorphism 
$$
\log_{\omega_f}:A_f(\Qp)_{/\tor}\otimes_\Zp\cO \rightarrow \cO
$$
that maps $A^1(\Qp)\otimes_\Zp\cO$ surjectively onto $p\cO$.
By the choice of $\omega_f$, 
the map $\log_{\omega_f}$ factors through $A_f(\Qp)_{/\tor}\otimes_{\Z(f)\otimes\Zp}\cO$.
The induced homomorphism $A_f(\Qp)_{/\tor}\otimes_{\Z(f)\otimes\Zp}\cO \hra \cO$ is injective 
and maps $A_f^1(\Qp)\otimes_{\Z(f)\otimes\Zp}\cO$ isomorphically onto
$p\cO$; this follows from $A_f(\Qp)_{/\tor}\otimes_{\Z(f)\otimes\Zp}\cO$ and $A_f^1(\Qp)\otimes_{\Z(f)\otimes\Zp}\cO$
both being free $\cO$-modules of rank one and the surjective mapping of the latter onto $p\cO$.
As $\cK_v = \Qp$ (since $p$ splits in $\cK$) it easily follows from this that 
\begin{equation*}\begin{split}
[A_f(\cK_v)_{/\tor}\otimes_{\Z(f)\otimes\Zp}\O : \O\cdot P] & = \frac{\#\cO/(\log_{\omega_f}P)}{\#\cO/(\log_{\omega_f}(A_f(\Qp)_{/\tor}\otimes_\Zp\cO))} \\
& = \frac{\#\cO/(\log_{\omega_f}P) \cdot\#(A_f(\Qp)_{/\tor}/A_f^1(\Qp) \otimes_{\Z(f)\otimes\Zp}\cO)}
{\#\cO/p\cO} \\
& = \frac{\#\cO/(\log_{\omega_f}P) \cdot\#(A_f(\Qp)/A_f^1(\Qp) \otimes_{\Z(f)}\cO)}{\#\cO/p\cO\cdot \#(A_f(\Qp)_\tor\otimes_{\Z(f)}\cO)}.
\end{split}
\end{equation*}
Since reduction modulo $p$ yields an isomorphism $A_f(\Qp)/A^1_f(\Qp)\isoarrow A_f(\F_p)$, we have
$A_f(\Qp)/A^1_f(\Qp)\otimes_{\Z(f)}\cO \isoarrow A_f(\F_p)\otimes_{\Z(f)}\cO = A_f[\grp^\infty](\F_p)$. The latter group is trivial unless $f$ is ordinary
with respect to $\grp$ (that is, $\grp\nmid a_p$) in which case it is isomorphic to $\cO/(1-a_p+p)$. Hence
$$
\#A_f(\Qp)/A^1_f(\Qp)\otimes_{\Z(f)}\cO = \#\cO/(1-a_p+p).
$$
Also, $A_f(\Qp)_{\tor}\otimes\Zp= \H^0(\Qp,A_f[p^\infty])$ so $A_f(\Qp)\otimes_{\Z(f)\otimes\Zp}\cO = \H^0(\Qp,W_f)$, hence
$$
\#A_f(\Qp)_\tor\otimes\Zp = \#\H^0(\Qp,W_f),
$$
Putting this together with the preceding formula for $[A_f(\cK_v)_{/\tor}\otimes_{\Z(f)\otimes\Zp}\O : \O\cdot P]$ we find 
\begin{equation*}
\begin{split}
[A_f(\cK_v)_{/\tor}\otimes_{\Z(f)\otimes\Zp}\O : \O\cdot P] & = \frac{\#\cO/(\log_{\omega_f}P)\cdot\#\cO/(1-a_p+p)}{\#\cO/p\cO \cdot \#\H^0(\cK_v,W_f)} \\
& = \frac{\#\cO/(({\frac{1-a_p+p}{p}})\log_{\omega_f} P)}{\#H^0(\cK_v,W_f)}.
\end{split}
\end{equation*}
Combining this last equality with \eqref{eq:deltav} we get
\begin{equation*}\label{eq:deltav2}
\#\delta_v = \frac{\#\cO/(({\frac{1-a_p+p}{p}})\log_{\omega_f} P)}{[A_f(\cK)\otimes_{\Z(f)}\cO:\cO\cdot P] \cdot \# \H^0(\cK_v,W_f)},
\end{equation*}
which, when substituted into \eqref{eq:modabvar-sha}, yields
\begin{equation}\label{eq:modabvar-sha2}
\#\H^1_{\SelMin}(\cK,W_f) = \#\Sha_\BK(W_f/\cK)\cdot 
\left ( \frac{\#\cO/(({\frac{1-a_p+p}{p}})\log_{\omega_f} P)}{[A_f(\cK)\otimes_{\Z(f)}\cO:\cO\cdot P] \cdot \# \H^0(\cK_v,W_f)}\right)^2
\end{equation}

In the special case that $A_f = E$ is an elliptic curve (i.e., $f$ has rational coefficients), $p$ is a prime of good reduction, $\cO=\Zp$, and $P\in E(\cK)$ has infinite order, 
we can rewrite the formula for $\#\delta_v$ as
\begin{equation*}\label{delta-eq-E}
\#\delta_v = \frac{\#\Zp/(({\frac{1-a_p(E)+p}{p}})\log_{\omega_E} P)}{[E(\cK):\Z\cdot P]_p\cdot\#\H^0(\cK_v,E[p^\infty])},
\end{equation*}
where $[-,-]_p$ denotes the $p$-part of the index and we have take for $\omega_f$ the N\'eron differential $\omega_E \in \Omega^1(E/\Z_{(p)})$.
In this case \eqref{eq:modabvar-sha2} then becomes
\begin{equation}\label{eq:modabvar-ec}
\#\H^1_{\SelMin}(\cK,E[p^\infty]) = \#\Sha(E/\cK)[p^\infty]
\cdot \left(\frac{\#\Zp/(({\frac{1-a_p(E)+p}{p}})\log_{\omega_E} P)}{[E(\cK):\Z\cdot P]_p\cdot\#\H^0(\cK_v,E[p^\infty])}\right)^2.
\end{equation}

%
%
\section{CM Points, CM Periods and Upper Bounds on $\Sha$}\label{bsd-cmpts}
In this section we recall the definition of Heegner points on certain Shimura curves and their Jacobians. 
The Shimura curves considered are moduli spaces for false elliptic curves and the Heegner points correspond to 
false elliptic curves with complex multiplication. These Heegner points give rise to cohomology
classes for the Tate module of an optimal quotient of the Jacobian of the Shimura curve. The Euler system method of Kolyvagin yields
an upper bound on the order of the Tate--Shafarevich group of this optimal quotient in terms of the indices of certain of these classes. 
This upper bound is recalled in Section \cite{jetchev:tamagawa}. It is an essential ingrediant in our proof of the main theorem
of this paper. For use in comparing $p$-adic $L$-functions in Section \ref{bsd-padicLfunc}, we also
explain how the complex and $p$-adic periods of false CM elliptic curves are identified with complex and $p$-adic periods
of (true) CM elliptic curves.  

Let $\cK$ be an imaginary quadratic field such that \eqref{split} holds. Let $-D_\cK < 0$ be the discriminant of $\cK$.

\subsection{The Heegner hypothesis}\label{subsec:heeghyp} 

Let $N$ be an integer. We will also assume that $N$ satisfies at least one of the two 
Heegner-type hypotheses recalled below. The first of these is:
\begin{equation}\label{H}
\begin{split}
\bullet & \ \text{$N=N^+N^-$ with $(N^+,N^-)=1$}; \\
\bullet & \ \text{$\ell\mid N^+$ if and only if $\ell$ splits in $\cK$;} \\
\bullet & \ \text{$\ell\mid N^-$ if and only if $\ell$ is inert in $\cK$;} \\
\bullet & \ \text{$N^-$ is squarefree with an even number of prime factors}.
\end{split}\tag{{H}}
\end{equation}
Implicit in the second of these assumptions is that the discriminant $-D_{\cK}<0$ of $\cK$ satisfies
\begin{equation}\label{coprime}
(N,D_{\cK})=1.\tag{{coprime}}
\end{equation}
If $N^-=1$, then \eqref{H} is just the usual Heegner hypothesis.
For some of the arguments that follow, we also need a more general Heegner-type hypothesis: 
\begin{equation}\label{gen-H}
\begin{split}
\bullet & \ \text{$N=N^+N^-$ with $(N^+,N^-)=1$}; \\
\bullet & \ \text{$\ell \mid N^+$ if and only if $\ell$ is split or ramified in $\cK$;} \\ 
\bullet & \ \text{$\ell \mid N^-$ if and only if $\ell$ is inert in $\cK$;} \\ 
\bullet & \ \text{$N^-$ is squarefree with an even number of prime factors}.
\end{split}\tag{{gen-H}}
\end{equation}
Note that the only difference between the hypotheses \eqref{H} and \eqref{gen-H} is in the latter the primes dividing $N^+$ are allowed  
to be ramified in $\cK$ (so \eqref{coprime} may not hold).

Let $f\in S_{2k}(\Gamma_0(N)$ be a newform and $\pi$ the associated cuspidal representation of $\GL_2(\A)$ as in Section \ref{subsubsec:modulargalrep}.
The third and forth assumption of either Heegner-type hypothesis implies that the epsilon factor $\eps(\pi,\cK,s)$ of the base change of $\pi$ to $\GL_2(\A_{\cK})$ satisfies
\begin{equation}\label{sign-1}
\eps(\pi,{\cK},1/2) = -1.\tag{{sign $-1$}}
\end{equation}

\subsection{The quaternion algebra $B$ and the Shimura curve $X_{N^+,N^-}$}\label{subsec:shimcurve}
Suppose $N$ satisfies \eqref{gen-H}. 
Let $B$ be the indefinite quaternion algebra of discriminant $N^-$.
Let $\O_B$ be a maximal order of $B$ and let $R \subset\O_B$ be an Eichler order of level $N^+$.
Fix an isomorphism of $\R$-algebras
\begin{equation}\label{BisoR}
B\otimes\R \isoarrow \M_2(\R)
\end{equation}
and an isomorphism of $\A^{\infty N^-}$-algebras
\begin{equation}\label{BisoAf}
B\otimes\A^{\infty N^-}\isoarrow \M_2(\A_f^{N-})
\end{equation}
that identifies $R\otimes\widehat\Z^{N^-}$ with the order
$\{\left(\smallmatrix a & b \\ c & d \endsmallmatrix\right)\in \M_2(\wZ^{N^-}) \ : \ N^+\mid c \}$.
Here, as usual, $\A^{\infty N^-}$ (resp.~$\wZ^{N^-}$) denotes the restricted product over $\Q_\ell$ (resp.~$\Zl$),
$\ell\nmid N^-$. 

In order to compare later statements with results in \cite{brooks:shimura}, we assume (without loss of generality) 
that the isomorphism \eqref{BisoR} and the fixed isomorphisms $R\otimes\Zl\isom\M_2(\Zl)$ for $\ell\mid N^+p$
all arise from the choice of a real quadratic field $M=\Q(\sqrt{p_0}) \subset B$, $p_0\nmid pND$,
in which the primes $\ell\mid pN^+$ split and from an identification 
$$
\iota_M \colon B\otimes M\isoarrow \M_2(M).
$$ 
In particular, the chosen isomorphisms are induced from $\iota_M$ by fixing an inclusion $M\hookrightarrow \R$ and a prime of
$M$ above each $\ell\mid pN^+$. Furthermore, we fix an idempotent $e\in \O_B\otimes \O_M[\frac{1}{2p_0}]$ as in 
\cite[p.7]{brooks:shimura}
such that 
$$
\iota_M(e) = \left(\smallmatrix 1 & 0 \\ 0 & 0 \endsmallmatrix\right).
$$

Let $G$ be the algebraic group over $\Q$ such that $G(S) = (B\otimes S)^\times$ for each $\Q$-algebra $S$.
Using the identification \eqref{BisoR} we define a homomorphism,
$$
h_0 \colon \Res_{\C/\R}(\G_m)\rightarrow G_{/\R}, \ \ x+iy \mapsto \left(\smallmatrix x & y \\ -y & x\endsmallmatrix\right),
$$
and let $X$ be the $G(\R)$-conjugacy class of $h_0$. The set $X$ has a natural complex structure and the map
$X\isoarrow \grh^\pm:=\C-\R$, $\Ad(g)h_0 \mapsto g(i)$, is a $G(\R)$-equivariant holomorphic isomorphism. The action of $G(\R)$
on $\grh^\pm$ is via \eqref{BisoR} and the usual action of $\GL_2(\R)$.

Let $X_{N^+,N^-}$ be the Shimura curve associated with the Shimura datum $(G,X)$ and the open compact subgroup
$K=(R\otimes\wZ)^\times\subset \wB^\times=(B\otimes\A^\infty)^\times = G(\A^\infty)$. This curve has a canonical model over $\Q$ with 
complex uniformization
\begin{equation}\label{Xcompunif}
 X_{N^+,N^-}(\C) = B^\times\backslash (X\times \wB^\times/K)
\end{equation}
It even has the structure of a coarse moduli space \cite{KatzMazur}, \cite{buzzard:integral}, \cite{Helm}; the solution of the moduli problem yields a smooth regular model over $\ds \Z[1/N^-]$ which, in particular, is smooth at $p$ 
if 
\begin{equation}\label{good}
p\nmid N
\tag{{good}}
\end{equation}
holds.

To describe the moduli problem represented by $X_{N^+,N^-}$ we recall that
a {\it false elliptic curve} over a scheme $S$ is a pair $(A,\iota)$ with $A$ an abelian surface over $S$ and
$\iota:\O_B\hookrightarrow \End_S(A)$ an injective homomorphism. A full level-$N^+$ structure on $(A,i)$ is an isomorphism
of group schemes $t \colon A[N^+]\isoarrow \O_B\otimes(\Z/N\Z)_S$ commuting with the $\O_B$-action on both sides,
and a $K$-level structure is a $K$-equivalence class of such full level $N^+$ structures. Then $X_{N^+,N^-}$ represents
the course moduli scheme over $\Z[\frac{1}{DN}]$ for the moduli problem classifying isomorphism classes of triples
$(A,\iota,t)$ where $(A,\iota)$ is a false elliptic curve over $S$ and $t$ is a $K$-level structure
\footnote{When $N^-=1$ (so $B=\M_2(\Q)$ and $G=\GL_2$)
the usual moduli interpretation of the modular curve $X_{1,N}= X_0(N)$ classifies elliptic curves $E$ together with 
a $\Gamma_0(N)$-equivalence class of isomorphisms $\alpha \colon E[N]\isoarrow (\Z/N\Z)^2_S$.
The two moduli problems are isomorphic, with the class of $(E,\alpha)$ being identified with the class of $A = E\times E$ together with
the obvious action of $\O_B=M_2(\Z)$ and the $K$-level structure $\alpha\times\alpha$.}.
In terms of the complex uniformization,
$[h,1]$, $h\in X$, represents the isomorphism class of the triple $(A_h,\iota_{\can}, t_{\can})$, where
\begin{itemize}
 \item $A_h = \left ( \O_B\otimes\R \right ) /\O_B$, where the complex structure on $\O_B\otimes\R$ is defined by right multiplication by
 $h(z)$, $z\in\C$;
 \item $\iota_\can$ is the action of $\O_B$ arising from left multiplication;
 \item $t_\can$ is the $K$-equivalence class of the canonical isomorphism $A_h[N^+] = \frac{1}{N^+}\O_B/\O_B$.
\end{itemize}

When $N^-=1$ (so $B=\M_2(\Q)$) $X_{N^+,N-}$ is not proper, but a proper model $X_{N^+.N^-}^*$ is obtained by adding cusps; this extends to a regular
model over $\Z_{(p)}$, which is still smooth if \eqref{good} holds. If $N^-\neq 1$, then $X_{N^+,N^-}$ is already proper, but to unify notation
we also write $X_{N^+.N^-}^*$ for $X_{N^+,N^-}$ in this case.

Let $\ell_0\nmid Npp_0$ be a prime (to be chosen later). Associated with $\ell_0$ is the Hecke correspondence $T_{\ell_0}$
for $\XNs$. The degree of this correspondence is $\ell_0+1$, so for any $x\in \XNs$ the divisor
$$
(T_{\ell_0}-\ell_0-1)[x] \in \mathrm{Div}^0(\XNs)
$$
has degree $0$.

Let $J(\XNs)_{/\Q}$ be the Jacobian of $\XNs$. We define a finite morphism 
\begin{equation}\label{eq:cohtriv}
\iota_{N^+,N^-}: \XNs \ra J(XNs), \ \ x\mapsto (T_{\ell_0}-\ell_0-1)[x].
\end{equation}

\begin{rmk}\label{rmk:cohtriv}
To compare with formulas in \cite{cai-tian:gz} (and ultimately with those in \cite{yuanzhangzhang:gz}), we also consider a different embedding.
Let $\delta_{N^+,N^-} \in \mathrm{Pic}(\XNs)$ be defined as follows. If $N^-=1$, then let
$m$ be an integer that annihilates the cuspidal subgroup of $\mathrm{Pic}^0(\XNs)$ (which is finite by the the theorem of Manin-Drinfeld)
and $\delta_{N^+,N^-} = m[\infty]$, where $[\infty]$ is the divisor of the cusp at infinity.
If $N^-\neq 1$, then let $\delta_{N^+,N^-}=\widetilde{\xi(\XNs)}$
be the Hodge class defined in \cite[\S 6.2]{zhang:gross-zagier} and let $m\geq 1$ be its degree. In all cases, the action of any Hecke
correspondences on $\delta_{N^+,N^-}$ is just multiplication by the degree of the correspondence (in particular, the action is
Eisenstein). We define a finite morphism (over $\Q$):
\begin{equation}\label{eq:cohtriv1}
\tilde\iota_{N^+,N^-} \colon \XNs\rightarrow J(\XNs), \ \ x \mapsto m[x]-\delta_{N^+,N^-}.
\end{equation}
As $\delta_{N^+,N^-}$ is Eisenstein we then have
\begin{equation}
\label{eq:cohtriv-relation}
(T_{\ell_0}-\ell_0-1)\cdot \tilde\iota_{N^+,N^-}(x) = m\cdot \iota_{N^+,N^-}(x).
\end{equation}
\end{rmk}

\subsection{CM points on $\XNs$}\label{subsec:cmpts}
Let $\iota_{\cK} \colon \cK\hookrightarrow B$ be an optimal embedding with respect to $R$, in the sense
that $\iota_{\cK}(\cK)\cap R = \O_{\cK}$. Each of the hypotheses \eqref{H} and \eqref{gen-H} ensures that such an embedding exist.

There exists a unique point $h\in X$ such that $h\in \grh^+$ and $\iota_{\cK}({\cK}^\times)$ fixes $h$. Moreover, the subgroup of
$B^\times$ fixing $h$ is just $\iota_{\cK}({\cK}^\times)$. Replacing the choice if $\iota_\cK$ with $\iota_\cK\circ\tau$ if necessary (this does
not change the point $h$), we may assume that the homomorphism
$h \colon \C\rightarrow G(\R)$ is such that $h(i) = \iota_{\cK}(\sqrt{-D})\otimes\frac{1}{\sqrt{D}}$
and that the action of $\iota_{\cK}(k)$, $k\in \cK^\times$, on $\ds \left[\smallmatrix h \\ 1 \endsmallmatrix\right]$ is just multiplication by 
$k$.

The set of CM points of $\XNs$ is, in terms of the complex uniformization \eqref{Xcompunif},
$$
\rCM(\XNs) = \{[h,b]\in \XN(\C) \ : \ b\in\wB^\times\}.
$$
This set does not depend on $\iota_{\cK}$ since any two embeddings ${\cK}\hookrightarrow B$ are $B^\times$-conjugate. 
In addition, the fixed optimal embedding $\iota_{\cK}$ induces a bijection 
\begin{equation}\label{eq:cm-bij}
\rCM(\XNs) \simeq \cK^\times \backslash G(\mathbb{A}_f) / \widehat{R}^\times.
\end{equation}
Shimura's reciprocity law shows that 
$\rCM(\XN)\subset \XN(\cK^{ab})$ and that the action of $\Gal({\cK}^{ab}/{\cK})$
is described in terms of the reciprocity map $\cK^\times\backslash \A_{{\cK},f}^\times\stackrel{\rec_{\cK}}{\rightarrow} \Gal({\cK}^{ab}/{\cK})$
by
$$
\rec_{\cK}(t)[h,b] = [h,\iota_{\cK}(t)b].
$$
In particular, since $\iota_{\cK}$ is an optimal embedding, the point
$x=[h,1]$
is defined over the Hilbert class field $H$ of $\cK$. Let
$$
x_{\cK}^{N^+, N^-} = \sum_{\sigma\in \Gal(H/\cK)} \iota_{N^+,N^-}(x)^\sigma \in J(\XNs)(\cK).
$$
Note that $x_{\cK}^{N^+, N^-}$ depends on the auxiliary prime $\ell_0$ (to be chosen later).

Let $f\in S_2(\Gamma_0(N)$ be a a newform. Then the modular abelian variety $A_f$ is a quotient of $J(\XNs)$. Let 
$\pi:J(\XNs)\rightarrow A_f$ be such a quotient map (in later applications we will choose a $\pi$ with nice properties at $p$).
We then obtain a Heegner point
$$
y_\cK^{N^+,N^-}  = \pi(x_\cK^{N^+,N^-}) \in A_f(K).
$$
Let $\grp$ be a prime of $\Z(f)$ containing $p$.
Suppose that
\begin{equation}\label{eq:l0good}
\grp\nmid (a(\ell_0)-\ell_0-1).
\tag{{$\ell_0$-good}}
\end{equation}
Then 
\begin{equation*}\label{eq:heegner-z}
z_\cK^{N+,N^-} = \frac{1}{a(\ell_0)-\ell_0-1}y_\cK^{N^+,N^-} \in A_f(\cK)\otimes_{\Z(f)}\Z(f)_\grp
\end{equation*}
is independent of the choice of $\ell_0$.

\begin{rmk}\label{rmk:l0good}
 If \eqref{irredK} holds for the residual representation modulo the prime $\grp$, then 
there is a positive proportion of primes $\ell_0\nmid Npp_0$ for which \eqref{eq:l0good} holds.
\end{rmk}

\subsection{Upper bounds on $\# \Sha(E/\cK)[p^\infty]$} Suppose now that $A_f=E$ is an elliptic equivalently (equivalenty, that
$\Q(f) = \Q$). Here, we recall a consequence of Kolyvagin's theorem on the structure of 
$\Sha(E/\cK)[p^\infty]$ extended to the case of elliptic curves quotients of Jacobians of Shimura curves. The following result is a direct consequence of \cite[Thm.3.2]{nekovar:euler}:   

\begin{thm}\label{thm:shaK}
Suppose that $E[p]$ is an irreducible $G_\Q$-representation.
Suppose that hypothesis \eqref{H} holds for the imaginary quadratic field $\cK$ and the conductor $N$ of $E$.
Then 
\begin{equation}
\# \Sha(E/\cK)[p^\infty] \leq p^{2m_0^{N^+, N^-}}, 
\end{equation}
where $m_0^{N^+, N^-} = \ord_p [E(\cK) : \Z y_{\cK}^{N^+, N^-}]$. 
\end{thm}

\begin{rmk}
In fact, one can deduce from the work of \cite{kolyvagin:structureofsha} (see also \cite{mccallum:sha}) and its generalizations to Shimura curves by Nekov\'a\v{r} \cite{nekovar:euler} a more precise result: 
\begin{equation}\label{eq:shaK-formula}
\# \Sha(E/\cK)[p^\infty] = p^{2(m_0^{N^+, N^-} - m_\infty^{N^+, N^-})}, 
\end{equation}
where $m_{\infty}^{N^+,N^-}$ is the maximal non-negative integer $m$ such that $p^{m}$ divides all the cohomology 
classes constructed by Kolyvagin (see \cite{jetchev:tamagawa} for more details in the case of modular curves). 
\end{rmk}

\subsection{Comparisons of CM periods}\label{subsec:compperiods}
In the following we recall the complex and $p$-adic periods of CM elliptic curves and explain how they can be identified with
periods of `false' CM elliptic curves.

\subsubsection{Complex and $p$-adic periods of CM elliptic curves}\label{subsubsec:cmec-per}
Suppose that $p$ splits in $\cK$. 
Let $H / \cK$ be the Hilbert class field of $\cK$ and let $F$ be a finite extension of $H$ and let $E_0 / F$ be an elliptic curve with 
complex multiplication (CM) by an order $\grO$ in the imaginary quadratic field $\cK$. We assume that $\grO\otimes\Z_{(p)}$ is a maximal 
$\Z_{(p)}$-order.
Fix a complex uniformization 
$E_0(\C) \isom \C / \mathfrak{a}$, where $\mathfrak{a}$ is a non-zero ideal of $\grO = \End(E_0)$. 
By enlarging $F$ if necessary we may assume that $E_0$ has good reduction at all primes above $p$ and so extends to an
elliptic scheme $\cE_0/\cO_{F,(p)}$.  
We are interested in $\Omega^1(\cE_0 / \cO_{F, (p)})$, which is a free $\cO_{F,(p)}$-module of rank one.
Write $\Omega^1(\cE_0 / \cO_{F, (p)}) = \cO_{F, (p)}\cdot \omega_{E_0}$ for some differential $\omega_{E_0} \in \Omega^1(\cE_0 / \cO_{F, (p)})$.

We define the complex period of $E_0$ as follows. We consider $\C$ to be an $F$-algebra via the fixed inclusion $\iota_\infty:\bQ\hra \C$. Then
$$
\Omega^1(\cE_0 / \cO_{F, (p)}) \otimes_{\cO_{F, (p)}} \C = \Omega^1(E_0 / \C).
$$ 
As the invariant differential $\omega_{\C} = dz$ on the complex torus $E_0(\C) \isom \C / \mathfrak a$ gives a generator for $\Omega^1(E_0 / \C)$, 
there exists a scalar $\Omega_\infty \in \C^\times$ such that $\omega_{E_0} = \Omega_\infty \cdot (2\pi i) \cdot \omega_\C$. 
This agrees with the definition of the complex period of the CM elliptic curve that is given in, e.g., \cite[p.1132]{bertolini-darmon-prasanna}. 

We define the $p$-adic period of $E_0$ in much the same way, following \cite[p.1134]{bertolini-darmon-prasanna}. 
Recall that $p=v\bv$ splits in $\cK$. 
We take the place $v$ to be that determined by the fixed embedding $\iota_p:\bQ\hookrightarrow \bQp$
and $\iota_v:\bQ\hra \overline{\cK}_v = \bQp$ to just be $\iota_p$. 
Denote also by $v$ the place of $F$ determined by the embedding $\iota_v$; so $\iota_v$ identifies $F_v$ with a finite extension of 
$\cK_v = \Qp$ in $\bQp$. Let $F_v^\ur$ be the maximal unramified extension of $F_v$ in $\bQp$ and let $R$ be the ring of integers of $F_v^\ur$.
Considering the good integral model $\cE_0 / R$, let $\widehat{\cE}_0$ be the formal completion over the identity section. 
As we are working over $R$, $\widehat{\cE}_0$ is non-canonically isomorphic to $\widehat{\G}_m$. 
We fix such an isomorphism $\widehat{\cE}_0 \isoarrow \widehat{\G}_m$.  
The latter is equivalent to fixing an isomorphism of $p$-divisible groups 
$\cE_0[p^\infty]^0 = \cE_0[\p_v^\infty] \isom \mu_{p^\infty}$, which is uniquely determined up to the action of an element of $\Z_p^\times$.   
Here $\p_v$ is the prime ideal of $\grO$ corresponding to $v$. The pullback of $dt/t$ under the the fixed isomorphism
$\widehat{\cE}_0 \isoarrow \widehat{\G}_m$ is an element
$$
\omega_\can \in \Omega^1(\cE_0/R) = \Omega^1(\cE/\cO_{F,(p)})\otimes_{\cO_{F,(p)}} R.
$$
We then define the $p$-adic period $\Omega_p \in R^\times$ by $\omega_{E_0} = \Omega_p \cdot \omega_{\can}$.

\begin{rmk} Both $\Omega_\infty$ and $\Omega_p$ depend on the choice of $\omega_{E_0}$. However, another choice of $\omega_{E_0}$
only replaces $\Omega_\infty$ and $\Omega_p$ with their multiples by the {\it same} scalar in $\cO_{F,(p)}^\times$.
The definition of $\Omega_\infty$ also depends on the choice of the complex uniformization. Composing with multiplication
by a non-zero scalar $\alpha\in \C^\times$:
$$
E_0(\C)\isoarrow \C/\gra \isoarrow \C/\alpha\gra,
$$
the complex period $\Omega_\infty$ gets multiplied by the scalar $\alpha$. Similarly, $\Omega_p$ also depends on the choice
of the isomorphism $\widehat{\cE}_0\cong \widehat{\G}_m$; changing this isomorphism replaces $\Omega_p$ with a $\Z_p^\times$-multiple.
\end{rmk}

\subsubsection{Complex and $p$-adic periods of ``false" CM elliptic curves}
To understand the complex and $p$-adic periods that appear in the formulas in \cite{brooks:shimura}, 
let $A$ be a ``false elliptic curve'' as in Section \ref{subsec:shimcurve}. We take $A$ and its endomorphisms to all be defined over a 
finite extension $F$ of $\cK$ containing the real quadratic field $M$. We assume that $A$ has good reduction at all places of $F$ over $p$ and we 
consider a good integral model $\cA  / \cO_{F, (p)}$. Then
$\Omega^1(\cA / \cO_{F, (p)})$ is free of rank two over $\cO_{F, (p)}$. 
Let $e\in \O_B\otimes\O_{M,(p)}$ be an idempotent as in Section \ref{subsec:shimcurve} and let $\omega_A$ be an $\cO_{F,(p)}$-generator of 
$e \Omega^1(A / \cO_{F, (p)})$, i.e., $e \Omega^1(\cA / \cO_{F, (p)}) = \cO_{F, (p)} \cdot \omega_A$. To define a complex period, consider 
$\Omega^1(A/\C) = \Omega^1(\cA/ \cO_{F, (p)})\otimes_{\cO_{F,(p)}}\C$. 
If we have a canonical basis element $\omega_{A, \C} \in e \Omega^1(A/\C)$, we will then be able to compare it against  
$\omega_{A}$ to define a complex period. We would similarly get a $p$-adic period from comparison with a canonical $R$-basis element
of $e\Omega^1(\cA/R)$.
In the next section, we explain how to obtain such canonical elements 
when $A$ is a false elliptic curve obtained from a CM elliptic curve via Serre's tensor product construction.   

\begin{rmk}
For the above, the only hypotheses on the real quadratic field $M$ that is needed is that $B\otimes M$ is split and $p$ splits in $M$.
No additional hypotheses are needed on  
the idempotent $e$. The extra hypotheses imposed in Section \ref{subsec:shimcurve} are assumed in the construction of Brooks \cite{brooks:shimura} 
where they are used to obtain a certain explicit formula for the Maass--Shimura operators; they follow closely the choices made by 
Mori \cite{mori} and Hashimoto \cite{hashimoto}. For the comparison between periods of CM elliptic curves and false elliptic curves in the next section
we do not need to make these extra assumptions. 
\end{rmk}

\subsubsection{Serre's tensor product construction}\label{subsubsec:serretensor}
To relate the periods of a CM elliptic curve to the periods of a false elliptic curve, we use Serre's tensor product construction 
(see \cite[\S 1.7.4]{chai-conrad-oort}) which we briefly recall. If $E_0/F$ is a 
CM elliptic curve as in Section~\ref{subsubsec:cmec-per} with complex multiplication by $\cO_\cK$, then by Serre's tensor product construction applied 
to $\cO_B \otimes_{\cO_\cK} E_0$, there exists a false elliptic scheme $\cA/\cO_{F,(p)}$ such that 
$\cA(R) = \cO_B \otimes_{\cO_\cK} \cE_0(R)$ for any $\cO_{F,(p)}$-algebra $R$; we let  the false elliptic curve $A/F$ be the
generic fiber of $\cA$. We will now relate the complex and $p$-adic periods of $A$ 
to the complex and $p$-adic periods of $E_0$, respectively. 

We have 
$$
\Omega^1(\cA / \cO_{F, (p)}) = \cO_B \otimes_{\cO_{\cK}} \Omega^1(\cE_0 / \cO_{F, (p)}) = 
\left ( \cO_B \otimes_{\cO_{\cK}} \cO_{F, (p)} \right ) \otimes_{\cO_{F, (p)}} \cO_{F, (p)}\cdot \omega_{E_0}. 
$$
We continue to write $e$ for the image in $\cO_B \otimes_{\cO_{\cK}} \cO_{F, (p)}$ of the idempotent $e\in \cO_B\otimes \cO_{F,(p)}$,
and let $\omega_{A} = e \otimes \omega_{E_0} \in e \Omega^1(\cA / \cO_{F, (p)})$. Then $\omega_A$ is an $\cO_{F,(p)}$-basis of
$e\Omega^1(\cA/\cO_{F,(p)})$.  On the other hand, 
$$
\Omega^1(A / \C) = \cO_B \otimes_{\cO_\cK} \Omega^1(E_0/\C) = (\cO_B \otimes_{\cO_\cK}\cO_{F,(p)}) \otimes_{\cO_{F,(p)}} \C \cdot \omega_{\C},  
$$
where $\omega_{\C} \in \Omega^1(E_0 / \C)$ is the $1$-form defined in Section~\ref{subsubsec:cmec-per}. 
Then $\omega_{A,\C} = e\otimes dz$ is a $\C$-basis of $e\Omega^1(A/\C)$. The complex period 
$\Omega_{A, \infty} \in \C$ is then defined via  
$\omega_{A} = \Omega_{A, \infty} \cdot (2\pi i) \cdot \omega_{A, \C}$. It follows that $\Omega_{A, \infty} = \Omega_\infty$.

The comparison of the $p$-adic periods of $E_0$ and $A$ is similar. We first note that over $R$ we have
$$
\cA[p^\infty]^0  = \cO_B\otimes_{\cO_\cK}\cE_0[p^\infty]^0 = \cO_B\otimes_{\cO_\cK}\cE_0[\grp_v^\infty] \isom 
(\cO_B\otimes_{\cO_\cK}\Zp)\otimes_\Zp\mu_{p^\infty},
$$
where $\Zp$ is a $\cO_\cK$-module via $\cO_\cK\hra \cO_{\cK,\grp} = \Zp$.
It follows that the pullback of $e\otimes dt/t$ is just 
$\omega_{A,\can} = e\otimes\omega_{\can} \in e\Omega(\cA/R) = e(\cO_B\otimes_{\cO_\cK}\Zp)\otimes_\Zp\Omega^1(\cE_0/R)$.
Here we have used the fixed embedding $M\hookrightarrow \Qp$ to identify $e$ with an element of $\cO_B\otimes \Zp$.
The $p$-adic period 
$\Omega_{A,p} \in R^\times$ is then defined via  
$\omega_{A} = \Omega_{A, p} \cdot \omega_{A,\can}$. Clearly, $\Omega_{A,p} = \Omega_p$.

%
%
\section{Anticyclotomic $p$-adic $L$-functions}\label{bsd-padicLfunc}
In this section we recall the $p$-adic $L$-functions that appear in the Iwasawa-Greenberg main conjectures for the Selmer groups
$\H^1_{\SelMin^\Sigma}(\cK, M)$ and $\H^1_{\SelMin^\Sigma}(\cK, \CM)$ defined in Sections~\ref{subsubsec:iwasawa-selstruct} and~\ref{subsubsec:bigsel}, respectively. These conjectures and their current status are recalled in
the section following this one.  We continue with most of the notation and conventions of the preceding discussion.

\subsection{Anticyclotomic $p$-adic $L$-functions}\label{subsec:acpadicLfunc} Let $\cK$ be an imaginary quadratic field of discriminant $-D_\cK<0$ such that $p$ splits in $\cK$.
Let $\Sigma_{cc}$ be the set of those continuous characters $\psi:G_{\cK}\rightarrow\Q_p^\times$ that satisfy
\begin{itemize}
\item $\psi^\tau = \psi^{-1}$;
\item $\psi$ factors through $\Gamma$;
\item $\psi$ is crystalline at both $v$ and $\bar v$;
\item $\psi$ has Hodge--Tate weights $-n<0$ and $n>0$ at $v$ and $\bv$, respectively.
\end{itemize}
Such a character $\psi$ is the $p$-adic avatar of an algebraic Hecke character $\psi^\alg$ of 
$\A_{\cK}^\times$ as follows.
Let $\rec_{\cK} \colon \A_{\cK}^\times\rightarrow G_{\cK}^{\ab}$
be the reciprocity map of class field theory, normalized so that uniformizers map to geometric Frobenius elements. Then $\psi^\alg$ is given by
$$
\psi^\alg(x) = \psi(\rec_{\cK}(x))x_v^{-n}x_\bv^{n}x_\infty^n\bar x_\infty^{-n}.
$$
To make sense of this expression we have to explain how this is to be seen as a $\C^\times$-valued character. The
quantity $\psi(\rec_{\cK}(x))x_v^{-n}x_\bv^{n}$, which is {\it a priori} $\Q_p^\times$-valued (since 
${\cK}_v$ and ${\cK}_\bv$ are both just $\Qp$),
actually takes values in $\bQ^\times\subset \bQp$; we use the chosen embedding $\bQ = \overline{\cK}\stackrel{\iota_v}{\hra} \overline{\cK}_v = \bQp$
to identify $\bQ^\times$ as a subgroup of $\bQ_p^\times$. We then use the fixed embedding $\bQ\subset\C$ to see
$\psi(\rec_{\cK}(x))x_v^{-n}x_\bv^{n}$ as $\C$-valued and to identify ${\cK}\otimes\R$ with $\C$.
The character $\psi^\alg$ then has infinity type $(n,-n)$ in the sense that $\psi^\alg_\infty(z) = z^n\bar z^{-n}$.

Let $f\in S_2(\Gamma_0(N))$ be a newform and let $\cO$ and $L$ be as in Section \ref{subsubsec:newforms}.
We assume that $L$ is so large that it contains the Hilbert class field of $\cK$. We also assume that \eqref{gen-H} holds for $\cK$ and $N$. 
Let $B$ be the quaternion algebra of discriminant $N^-$ and let $\O_{B,N^+}\subset \O_B$ be 
the orders as in Section \ref{subsec:shimcurve}. The Jacquet--Langlands correspondence implies the existence of a classical holomorphic quaternionic 
modular form $f_B$ (that is, a modular form for the arithmetic subgroup $\Gamma_0^B(N^+)$ defined by the units of the Eichler order $\O_{B,N^+}$)
of weight $2$ and trivial character and having the same Hecke eigenvalues as $f$ at all primes $\ell\nmid N^-$.
Let $\grp\subset \Z(f)$ be the prime determined by the fixed embedding $\Q(f)\hookrightarrow \bQ_p$ (so $L$ is a finite extension of $\Q(f)_\grp$).
We now require that \eqref{good} hold; that is, $p\nmid N$.
Then the form $f_B$ can be normalized so that it is defined over $\Z(f)_{(\grp)}$ --  that is, it is identified with a global section 
$\omega_f$ of $\Omega^1(X^*_{N^+,N^-}/\Z(f)_{(\grp)})$ -- and non-zero modulo~$\grp$.

Let $R$ be the completion of the ring of integers of the subfield of $\bQ_p$ generated by the $p$-adic field $L$ from Section~\ref{subsec:new-galrep-modabvar} and the maximal unramified
extension $\Q_p^{ur}\subset\bQp$ of $\Qp$. Note that $R$ is a complete DVR.
As explained in \cite{bertolini-darmon-prasanna}, \cite{brooks:shimura}, and \cite{burungale:mu2}, under the hypotheses \eqref{H} and \eqref{good} there
exists $L_p(f) \in \Lambda_R : = \Lambda\widehat\otimes_\O R = R[\![\Gamma]\!]$
such that for $\psi\in\Sigma_{cc}$ with $n\equiv 0 \pmod{p-1}$, the image $L_p(f,\psi): = \widehat\psi(L_p(f)) \in R$ of $L_p(f)$ under the $R$-linear
extension of the character $\psi$ to a continuous homomorphism $\widehat\psi:\Lambda_R \rightarrow R$ satisfies
\begin{equation}\label{Lp-interp}
L_p(f,\psi) = E_\bv(f,\psi)^2 \cdot t_{\cK} \cdot \frac{C(f,\psi)}{\alpha(f,f_B) W(f,\psi)} \cdot \Omega_p^{4n}\cdot \frac{L(f,\psi^\alg,1)}{\Omega_\infty^{4n}}.
\end{equation}
The various factors appearing in this expression are:
\begin{itemize}
\item $E_\bv (f,\psi) = (1-\psi^{\alg}(\varpi_\bv)a_p p^{-1} + \psi^{\alg}(\varpi_\bv)^2 p^{-1})$, with $\varpi_\bv\in \cK_\bv$ a uniformizer;
\item $t_{\cK}$ is a power of $2$ that depends only on $\cK$ (unimportant since $p$ is assumed odd);
\item $\ds C(f,\psi) = \frac{1}{4}\pi^{2n-1}\Gamma(n)\Gamma(n+1)w_\cK\sqrt{D_{\cK}}\prod_{\ell|N^-} \frac{\ell-1}{\ell+1}$,
where $w_{\cK}$ is the number of roots of unity in $\cK$;
\item $W(f,\psi) = W_f N_{{\cK}/\Q}(\grb)\psi^\alg(x_\grb) N^n b_N^{-2n}$, where
$\grb\subset \O_{\cK}$ and $b_N\in\O_{\cK}$ are as in \cite[Prop.~8.3]{brooks:shimura}, $x_\grb \in \A_\cK^{\infty,\times}$ is such that
$\ord_w(x_{\grb,w}) = \ord_w(\grb)$ for all finite places $w$ of $\cK$, and
$W_f\in\C^\times$ is a complex number of norm one as in the remarks following the proof of Lemma 8.4 of 
\cite{brooks:shimura};  both $\grb$ and $b_N$ can be taken to be coprime to $p$;
\item $\alpha(f,f_B) = \frac{\langle f,f\rangle_{\Gamma_0(N)}}{\langle f_B,f_B\rangle_{\Gamma_0^B(N^+)}}$ is a ratio of Petersson norms, which are normalized
so that  $\langle g,g\rangle_\Gamma = \int_{\Gamma\backslash\mathfrak{h}} g(z)\overline{g(z)}\frac{dx dy}{y^2}$;  in \cite{brooks:shimura} this ratio was shown to belong to 
$L$;
\item $\Omega_p\in R^\times$ and $\Omega_\infty\in\C^\times$ are, respectively, the $p$-adic and complex periods of a `false elliptic curve' with CM by $\O_\cK$
as defined at the start of \cite[\S8.4]{brooks:shimura}; as explained in Section \ref{subsubsec:serretensor} these can be taken to be the respective $p$-adic
and complex periods of an elliptic curve with complex multiplication by $\cO_\cK$; 
\item $L(f,\psi^\alg,s) = L(\pi_{\cK}\times\psi^\alg, s-1/2) = L(V_f^\vee\otimes\psi,s)$, where $\pi_{\cK}$ is the base change of $\pi$ to $\GL_2(\A_{\cK})$ and,
following our earlier convention, we use geometric conventions for the $L$-function of the $G_{\cK}$-representation $V_f^\vee\otimes\psi$.
\end{itemize}

\begin{rmk}\hfill
\begin{itemize}
\item[\rm (a)] To see the interpolated values $L_p(f,\psi)$ as belonging to $R$, one first recognizes many of the quantities on the right-hand side of \eqref{Lp-interp}
as algebraic, that is, as belonging to $\bQ\subset \C$. These are then viewed as belonging to $\bQp$ via an embedding extending the fixed inclusion
$\Q(f)\hookrightarrow L$.
\item[\rm (b)] To compare the interpolation formula \eqref{Lp-interp} with that in \cite{brooks:shimura} and \cite{burungale:mu2},
one should take $\chi^{-1}=\psi^\alg|\cdot|_{\A_\cK}$ in
the formulas in {\it loc.~cit}. Note that $\chi$ has infinity type $(-2-j,j)$ with $j=n+1$ (this is infinity type $(2+j,-j)$ with the conventions of
{\it loc.~cit.}) and $L(f,\chi^{-1},0) = L(f,\psi^\alg,1)$.
\item[\rm (c)] In \cite{bertolini-darmon-prasanna} and \cite{brooks:shimura} the $p$-adic $L$-function is only constructed as a continuous function (though the construction clearly gives a measure).
The measure is made explicit, for example, in \cite{burungale:mu2}.
\end{itemize}
\end{rmk}

In this paper we make use of two important results about $L_p(f)$, due to Burungale and Brooks, respectively.  The first of these
is the vanishing of the $\mu$-invariant of $L_p(f)$ and the second is a formula for the value
$L_p(f,1)$ under the trivial character (which does not belong to $\Sigma_{cc}$). 
We recall these results in the following two sections.

\subsubsection{Vanishing of the anticyclotomic $\mu$-invariant}

The choice of topological generator $\gamma\in \Gamma$ determines a continuous isomorphism $\Lambda_R \isoarrow R[\![T]\!]$ such that $\gamma-1\mapsto T$;
we use this to identify each $\lambda\in\Lambda_R$ with a power series $\lambda(T)$.
Each element $\lambda\in\Lambda_R$ has a unique expression $\lambda(T) = \varpi_R^{\mu(\lambda)} p_\lambda(T) u_\lambda(T)$ with
$\varpi_R\in R$ a fixed uniformizer, $\mu(\lambda)\in \Z$ a non-negative integer, $p_\lambda$ a monic polynomial of minimal degree, and $u_\lambda\in R[\![T]\!]^\times$.
The integer $\mu(\lambda)$ - the $\mu$-invariant of $\lambda$ - and the degree of $p_\lambda$ are independent of the choices of $\gamma$ and $\varpi_R$.

Under the additional assumption that
\begin{equation}\label{sqfree}
\text{$N$ is squarefree}\tag{$\square$-free}
\end{equation}
Burungale \cite{burungale:mu2} has shown that the $\mu$-invariant of $L_p(f)$ vanishes:

\begin{prop}[{\cite[Thm.~B]{burungale:mu2}}]\label{mu=0 prop} If $p\geq 3$ and \eqref{irredK}, \eqref{split}, \eqref{H}, \eqref{good}, and \eqref{sqfree} hold,  then $\mu(L_p(f))=~0$.
\end{prop}

Let $\Sigma$ be a finite set of places of $\cK$ that do not divide $p$. We also define an incomplete $p$-adic $L$-function $L^\Sigma_p(f)$ by removing the
Euler factors at those places in $\Sigma$:
$$
L_p^\Sigma(f) = L_p(f) \times \prod_{w\in\Sigma} P_w(\eps^{-1}\Psi^{-1}(\Frob_w)) \in R[\![\Gamma]\!].
$$
Then $L_p^\Sigma(f,\psi): = \widehat\psi(L_p^\Sigma(f))$, $\psi\in\Sigma_{cc}$, satisfies the interpolation formula \eqref{Lp-interp} but
with $L(f,\psi^\alg,1)$
replaced with the incomplete $L$-value $L^\Sigma(f,\psi^\alg,1)$ on the right-hand side.

\begin{rmk}
If $w=(\ell)$ is an inert place in $\cK$, then $\Psi(\Frob_w) = 1$ and in this case
$P_w(\eps^{-1}\Psi^{-1}(\Frob_w)) = (1-a_\ell(f)\ell^{-1}+\ell^{-1})(1+a_\ell(f)\ell^{-1}+\ell^{-1})$,
which can contribute to the $\mu$-invariant of $L_p^\Sigma(f)$. In particular, the $\mu$-invariant of the incomplete $L$-function $L_p^\Sigma(f)$
may be greater than that of $L_p(f)$. This does not happen for the usual cyclotomic $p$-adic $L$-function.
\end{rmk}

\subsubsection{A formula for $L_p(f,1)$} 
We recall Brooks' formula \cite[Prop.~8.13]{brooks:shimura} 
(see also \cite[Prop.2.6.1]{skinner:conversegz}) 
for the value $L_p(f,1)$ of $L_p(f)$ at the trivial character (which is outside the range 
of interpolation). In order to this, suppose that $N$ and $\cK$ satisfy \eqref{gen-H}. 
Recall that there is a logarithm map 
$$
\log_{\omega_f}: J(X_{N^+,N^-}^*)(\cK_v)\otimes_\Zp\cO \rightarrow \cK_v
$$ 
such that $d\log_{\omega_f} = \omega_f \in \Omega^1(J(X_{N^+,N^-}^*)/\Zp)\otimes_\Zp\cO = \Omega^1(J(X_{N^+,N^1}^*/\cO)$.
Recall that $\omega_f\in  \Omega^1(J(X_{N^+,N^-}^*)/\Zp)\otimes_\Zp\cO$ is the $\cO$-basis element associated with $f_B$,
as defined above.  Let $\ell_0\nmid pN$ be a prime that splits in $\cK$ and such that $1-a_{\ell_0} + \ell_0\neq 0$ (there
are a positive proportion of such $\ell_0$). In Section \ref{subsec:cmpts} we defined a Heegner point 
$x_{\cK}^{N^+,N^-} \in J(X_{N^+,N^-}^*)(\cK)$. Then Brooks' formula is:

\begin{prop}[{\cite[Prop.~8.13]{brooks:shimura}}]\label{prop:brooks} Suppose \eqref{sqfree} holds and $N$ and $\cK$ satisfy \eqref{gen-H}. Then  
$$
L_p(f, 1) = \frac{1}{(1-a_{\ell_0}+\ell_0)^2}\cdot \left ( \frac{1 + p - a_p}{p}\right )^2  \left( \log_{\omega_f} x_\cK^{N^{+}, N^-} \right )^2,  
$$
where the equality is up to a $p$-adic unit. 
\end{prop}

\noindent 
Brooks' formula is actually given in terms of the logarithm of a point 
$$
\tilde x_\cK^{N^+,N^-} = \sum_{\sigma\in\Gal(H/\cK)}\epsilon_f [x]^\sigma \in J(X_{N^+,N^-}^*)(\cK')\otimes\Z(f),
$$ 
where $[x]$ is as in the definition of $x_\cK^{N^+,N^-}$ but $\epsilon_f$ is a certain projector in a ring correspondences. 
The relation with $x_\cK^{N^+,N^-}$ is  $\eps_f\cdot x_\cK^{N^-,N^-} = (1-a_{\ell_0}+\ell_0) \cdot \tilde x_\cK^{N^+,N^-}$ from which we find
$$
 \log_{\omega_f} x_\cK^{N^{+}, N^-} =  \log_{\omega_f} \eps_f\cdot x_\cK^{N^{+}, N^-} = (1-a_{\ell_0}+\ell_0) \cdot \log_{\omega_f} \tilde x_\cK^{N^{+}, N^-}.
$$
Additional comments on the comparison of these cohomological trivializations can be found in \cite[p.13]{skinner:conversegz}.

We now rewrite Brook's formula in a form more directly comparable to the expression \eqref{eq:modabvar-sha2} at the end of Section 3.
We do this in the case that $\overline V$ is an irreducible $\kappa$-representation of $G_\Q$.
Then we can -- and do -- take $A_f$ and the quotient map $\pi:J(X_{N^+,N^-}^*)\ra A_f$ to be $(\Z(f),\grp)$-optimal
in the sense of \cite{zhang:selheeg}. This means that $\pi$ is the composition of an optimal quotient $\pi_0:J(X_{N^+,N^-}^*)\rightarrow A_0$
and an isogeny $\phi:A_0\rightarrow A_f$ such that the image of the map $T_p A_0 \rightarrow T_p A_f$ induced by $\phi$ is not contained in 
$\grp T_p A_f$.  We note that \eqref{irredK} also holds and that we may choose $\ell_0\nmid pN$ such that $\ell_0$ splits in $\cK$ and
\eqref{eq:l0good} holds. Then the formula in Proposition \ref{prop:brooks} can be expressed as:

\begin{prop}\label{prop:brooks-Af}
Suppose \eqref{good} and \eqref{irredK} hold. 
Then with the above choice of $A_f$ and $\pi$, up to a unit in $\cO$ we have that $L_p(f,1)$ equals
$$
\left ( \frac{1 + p - a_p}{p}\right )^2  \log_{\omega_{A_f}} \left (z_\cK^{N^{+}, N^-} \right )^2,
$$
up to a $p$-adic unit.
\end{prop}

\begin{proof} To prove this proposition it suffices to prove that $\pi^*(\omega_{A_f})$ is an $\cO^\times$-multiple of $\omega_f$.
By the hypothesis \eqref{good}, each of $J(X_{N^+,N^-})$, $A_0$, and $A_f$ has good reduction at $p$.
Let $\mathcal{J}$, $\cA_0$, and $\cA$ be their respective N\'eron models over $\Zp$; these are abelian schemes. 
The maps $\pi$, $\pi_0$, and $\phi$ extend to morphisms of these N\'eron models. Furthermore, as $\pi_0$ is an optimal quotient
(so the kernel of $\pi_0$ is also an abelian variety) and since $p-1>1$, the image of 
$\pi_0^*:\Omega^1(\cA_0/\Zp)\rightarrow \Omega^1(\mathcal{J}/\Zp)$ is a $\Zp$-direct summand; this follows from
\cite[Thm.~4, p.~187]{neronmodels} (cf. \cite[Cor.~A.1]{abbes-ullmo:manin}). 
So it suffices to show that $\phi^*(\omega_{A_f})$ is part of an $\cO$-basis of $\Omega^1(\cA_0/\Zp)\otimes_\Zp\cO$. 

We have that 
$$
\Omega^1(\cA_f/\Zp)\otimes_\Zp\cO/p = \Omega^1(\cA_f/\F_p)\otimes_\Zp\cO/p = \mathrm{Lie}_{\F_p}(\cA_f[p]^0)^\vee\otimes_\Zp \cO/p.
$$
Using the prime factorization $p\Z(f) = \prod_{\grq\mid p} \grq^{e_\grq}$, the right-hand side can be written as
$$
\prod_{\grq\mid p} \mathrm{Lie}_{\F_p}(\cA_f[\grq^{e_\grq}]^0)^\vee\otimes_\Zp \cO/p.
$$
By the choice of $\omega_{A_f}$, 
the image of $\omega_{A_f}$ in $\Omega^1(\cA_f/\Zp)\otimes_\Zp\cO/p$ is identified with a 
$\cO/p$-basis element of the $\grp$-summand (that is, for $\grq=\grp$). So it suffices
to prove that $\phi^*$ induces an injection
$$
\mathrm{Lie}_{\F_p}(\cA_f[\grp^{e_\grp}]^0)^\vee\otimes_\Zp \cO/p \hookrightarrow \mathrm{Lie}_{\F_p}(\cA_0[p]^0)^\vee\otimes_\Zp \cO/p.
$$
For this, we note that since the image of $T_p A_0$ is not contained in $\grp T_p A_f$, 
and since $A_f[\grp]$ is an irreducible $\kappa$-representation of $G_\Q$, the induced homomorphism
$$
A_0[p] = T_p A_0/pT_pA_0 \rightarrow T_pA_f/\grp^{e_\grp} T_p A_f = A_f[\grp^{e_\grp}]
$$
is surjective. Hence the morphism $\cA_0[p]\ra \cA_f[\grp^{e_\grp}]$ is surjective and so, too, is the induced
map of the Lie algebras of their connected subgroups.
\end{proof}

\subsection{The two-variable Rankin--Selberg $p$-adic $L$-function}\label{subsec:twovarpadicLfunc}
The $p$-adic $L$-function $L^\Sigma_p(f)$ is the specialization of a $p$-adic $L$-function $\CL_p^\Sigma(f)\in\Lambda_{\cK,R}: = R[\![\Gamma_{\cK}]\!]$ first constructed by
Hida \cite{hida:padic2}. In what follows we recall this $p$-adic $L$-function and explain its relation to $L_p^\Sigma(f)$.

Let $\Lambda^+ = \O[\![\Gamma_\cK^+]\!]$. Let
$$
u = \eps(\gamma^+).
$$
We say that a continous $\O$-homomorphism $\lambda:\Lambda^+\rightarrow \bQp$ is algebraic of weight $n$ if $\lambda(\gamma_+) = u^{-n}$
for some integer $n\geq 0$.
Let $\bK_\infty^v\subset \bK_\infty$ be the maximal subfield unramified at $v$, and let $\Gamma_v = \Gal(\bK_{\infty}^v/\cK)$. Then $\Gamma_v\cong\Zp$.
Let $\CI = \O[\![\Gamma_v]\!]$. The composite of the canonical homomorphisms $\Gamma_\cK^+\hookrightarrow \Gamma_\cK\twoheadrightarrow\Gamma_v$ makes $\CI$ into
a $\Lambda^+$-homomorphism. We say that a continuous $\O$-homomorphism $\lambda:\CI\rightarrow \bQp$ is algebraic of weight $k=2n+1\geq 1$ if
its restriction to $\Lambda^+$ is algebraic of weight $n$.

Let $\Theta_v \colon \A_\cK^\times\rightarrow \Gamma_v$ be the composition of the reciprocity map $\rec_\cK:\A_\cK^\times\rightarrow G_\cK^{\ab}$
of class field theory with the canonical projection $G_\cK^{\ab}\twoheadrightarrow \Gamma_v$.
For each non-zero fractional ideal $\gra$ of $\cK$, we
let $x_{\gra}\in \A_{\cK}^{\infty,\times}$ be a finite idele of $\cK$ such that $\ord_w(x_{\gra,v}) = \ord_w(\gra)$ for all finite places $w$ of $\cK$.
Then the formal $q$-expansion
$$
\ff = \sum_{n=1}^\infty \bb(n) q^n \in \CI[\![q]\!], \ \ \bb(n) = \sum_{\gra \subset \O_\cK,  N(\gra)=n \atop (\gra,\grp_\bv)=1,} \Theta_v(x_\gra),
$$
is an ordinary $\CI$-adic eigenform of tame level $D_\cK$, in the sense that
if $\lambda:\CI\rightarrow \bQp$ is an algebraic homomorphism of weight $k=2n+1$ with $n\equiv 0\pmod{p-1}$, then
$f_\lambda = \sum_{n=1}^\infty \lambda(\bb(n))q^n$ is the $q$-expansion of a $p$-ordinary $p$-stabilized newform of weight $k$ and level $D_\cK p$.
The form $f_\lambda$ can be identified as follows. The condition that $\lambda$ be algebraic of weight $k$ implies that
$\rho_\lambda = \lambda|_{\Gamma_v}$, viewed as a continous $\bQ_p^\times$-valued character of $G_\cK$, has Hodge-Tate weights $0$ and $k-1$
at $v$ and $\bv$, respectively. Associated with $\rho_\lambda$ is an algebraic Hecke character $\rho_\lambda^\alg$ with
infinity type $(0,1-k)$ that is unramified at $v$ and $\bar v$. The ordinary eigenform $f_\lambda$ is the ordinary $p$-stabilization 
of the newform $f_\lambda^0$ of weight $k$ and level $D_\cK$ associated with $\rho_\lambda^\alg$ (so, in particular,
$L(f_\lambda^0,s) = L(\rho_\lambda^\alg,s)$).

Let $\Lambda^\Hida=\O[\![1+p\Zp]\!]$ (this is denoted $\Lambda$ in \cite{hida:padic2}). We give $\Lambda^+$ the structure of a $\Lambda^\Hida$-algebra
by $1+\Zp\ni u\mapsto u^{-1}\gamma_+^{-1}\in\Lambda^+$; this is an isomorphism. We also endow $\CI$ with the structure of a $\Lambda^\Hida$-algebra by
$1+\Zp\ni u\mapsto u\gamma_+^{-2}\in \CI$. With the latter structure, our definition of an algebraic homomorphism of weight $k$ is consistent
with that in \cite{hida:padic2}. Note that $\CI$ is a finite, flat $\Lambda^\Hida$-algebra.

For a $p$-adic $\O$-algebra $A$,
let $\CX(A) = \Hom_{\O-\alg, cts}(A,\bQp)$ be the set of all continuous $\O$-algebra homomorphisms.
Hida \cite[Thm.~5.1b]{hida:padic2} has shown that there is an element $D$ in the fraction ring of $\CI\widehat\otimes_\O \Lambda^+$ that,
when viewed as a $p$-adic analytic function on $\CX(\CI)\times\CX(\Lambda^+)$, has the following properties.
The function $D$ is finite at each point $(\lambda,\lambda') \in \CX(\CI)\times\CX(\Lambda^+)$ with $\lambda$ and $\lambda'$
algebraic homomorphisms of respective weights $k$ and $n$ and satisfying $1\leq n\leq k-1$, and
if $k \equiv 1\pmod{2(p-1)}$ and $n\equiv 0 \pmod{p-1}$, then
\begin{equation}\label{Dp-interp}\begin{split}
D(\lambda,\lambda') = (D_\cK N)^{k/2} & N^{n}\Gamma(n)\Gamma(n+1) W_{f} 
\frac{E(1+n)}{S(\lambda)} \frac{\CalD(1+n,f_\lambda,f)}{\Omega(f_\lambda,f,n)},
\end{split}
\end{equation}
where
\begin{itemize}
\item $S(\lambda) = (1-p^{k-1}\rho_\lambda^\alg(\varpi_v)^{-2})(1-p^{k-2}\rho_\lambda^\alg(\varpi_v)^{-2})$;
\item $E(s) = (1-\rho_\lambda^\alg(\varpi_v)a_p(f)p^{-s} + \rho_\lambda^\alg(\varpi_v)p^{1-s})$;
\item $\CalD(f_\lambda,f,s) = \zeta^{pND_\cK}(k+2-2s-2) \times \sum_{n=1}^\infty a_n(f_\lambda)a_n n^{-s} = E(s)L(f,(\rho_\lambda^\alg),s)$;
\item $\Omega(f_\lambda,f,n)= (2\pi i)^{2n-1}(2i)^{k+1}\pi^2\langle f_\lambda^0,f_\lambda^0\rangle_{\Gamma_1(D_\cK)}$.
\end{itemize}

Let $\CI_R = \CI\widehat\otimes_\cO R = R[\![\Gamma_v]\!]$ and let $L_p^-(\cK)\in \CI$ be the anticyclotomic $p$-adic $L$-function of the imaginary quadratic field $K$.
The $p$-adic $L$-function $L_p^-(\cK)$  is a specialization of Katz's two-variable $p$-adic $L$-function and
satisfies the following interpolation property: for an integer $k\geq 0$ such that $k\equiv 1\pmod{2(p-1)}$
and $\lambda\in\CX(\CI)$ an algebraic homomorphism of weight $k$,
$$
L^-_p(\cK,\lambda) : =\lambda(L_p^-(\cK)) = S(\lambda)\Omega_p^{2k-2} \frac{w_\cK(2\pi)^{k-2}\Gamma(k)}{2 D_\cK^{(k-1)/2}\Omega_\infty^{2k-2}}
L(\rho_\lambda^\alg\rho_\lambda^{\alg,-\tau},1).
$$
Here the superscript `-$\tau$' denotes composing the inverse character with the action on $\A_\cK^\times$ of
the nontrivial automorphism $\tau$ of $K$. The complex and $p$-adic periods $\Omega$ and $\Omega_p$ are the periods of 
an elliptic curve with CM by $\cO_\cK$ and can be fixed to be the same as those appearing in
\eqref{Lp-interp}.  Appealing to the well-known relation
\begin{equation*}\begin{split}
\langle f_\lambda^0,f_\lambda^0\rangle_{\Gamma_1(D_\cK)} & = \Gamma(k)D_\cK^2 2^{-2k}\pi^{-1-k} \cdot\res_{s=k} \CalD(s,f_\lambda,f_\lambda^\tau) \\
& = \Gamma(k)D_\cK^2 2^{-2k}\pi^{-1-k} \cdot \frac{2\pi h_\cK}{w_\cK D_\cK^{1/2}}L(\rho_\lambda^\alg\rho_{\lambda}^{\alg,-\tau},1),
\end{split}
\end{equation*}
where $h_\cK$ is the class number of $\cK$,
the interpolation formula for $L_p^-(\cK)$ can be rewritten as
\begin{equation}\label{LpK-interp}
 L_p^-(\cK, \lambda) = w_\cK^2 S(\lambda) \Omega_p^{2k-2} \frac{\pi^{2k-2} 2^{3k-4}}{h_\cK D_\cK^{k/2+1}\Omega_\infty^{2k-2}}
 \langle f_\lambda^0,f_\lambda^0\rangle_{\Gamma_1(D_\cK)}.
\end{equation}
Note that $\rho_\lambda^{\alg}\rho_\lambda^{\alg,-\tau}$ is an anticyclotomic character with infinity type $(k-1,1-k)$.

Suppose 
\begin{equation}\label{units}
p\nmid w_\cK. 
\tag{units}
\end{equation}
We put
$$
\CL_p(f) = \left (\frac{h_\cK}{w_\cK}L_p^-(\cK)\otimes 1 \right ) D.
$$
We consider $\CL_p(f)$ as an element of the fraction field of $\Lambda_\cK = \O[\![\Gamma_\cK]\!]$ via the isomorphism
$\Gamma_\cK\isoarrow \Gamma_v\oplus\Gamma_\cK^+$ that is the direct sum of the canonical projections
to $\Gamma_v$ and $\Gamma_\cK^+$, respectively.  Then for a finite set $\Sigma$ of finite places
of $K$ not dividing $p$ we put
$$
\CL_p^\Sigma(f) = \CL_p(f) \times\prod_{w\in\Sigma}  P_w(\eps^{-1}\Psi_\cK(\frob_w)).
$$
If $\Sigma$ contains all places dividing $ND_\cK$ that do not divide $p$, then Wan has shown that
$\CL_p^\Sigma(f) \in \Lambda_\cK$ \cite{wan:rankin}. For such $\Sigma$ we let
$\CL_p^\Sigma(f)^-\in \Lambda$ be the image of $\CL_p^\Sigma(f)$ under the quotient map
$\Lambda_\cK\twoheadrightarrow\Lambda$ induced by the canonical projection $\Gamma_\cK\twoheadrightarrow\Gamma$.

\subsection{Relating $\CL^\Sigma_p(f)^-$ to $L_p^\Sigma(f)$}\label{subsec:spec}
Let $\psi\in\Sigma_{cc}$. By abuse of notation we
also denote by $\widehat\psi$ the homomorphism of $\Lambda_\cK$ obtained by composition with the projection
$\Lambda_\cK\twoheadrightarrow\Lambda$ and we put $\CL_p^\Sigma(f,\psi)^- = \widehat\psi(\CL^\Sigma_p(f))$. 
Then the homomorphism $\widehat\psi\in\CX(\Lambda_\cK)$ corresponds via the isomorphism
$\Lambda_\cK\isoarrow\CI\widehat\otimes_\O\Lambda^+$ to the point $(\lambda,\lambda')\in \CX(\CI)\times\CX(\Lambda^+)$
with $\lambda|_{\Gamma_v} = \psi_n\eps^{-n}$ (and so algebraic of weight $2n+1$) and
with $\lambda'|_{\Gamma_\cK^+} = \eps^n$ (and so algebraic of weight $n$). Here $n$ is such that $\psi^\alg$ has infinity type $(n,-n)$.
From the interpolation formulas for $D(\lambda,\lambda')$
and $L^-_p(\cK,\lambda)$ we then find
\begin{equation*}\label{CalLp-interp}
\begin{split}
\CL_p^\Sigma(f,\psi)^- = \widehat\psi(\CL^\Sigma_p(f)) = w_\cK N^{2n+1/2} & D_\cK^{-1} \Gamma(n)\Gamma(n+1)W_f i^{-1} 2^{2n}\pi^{2n-1} \\
& \times E_{\bv}(f,\psi)^2 \Omega_p^{4n} \frac{L^\Sigma(f,\psi_n^\alg,1)}{\Omega^{4n}}.
\end{split}
\end{equation*}
In particular, for $\psi\in\Sigma_{cc}$,
\begin{equation}\label{CL-=Lp}
\CL_p^\Sigma(f,\psi)^- = c(f,\psi)L_p(f,\psi),
\end{equation}
where
\begin{equation*}
c(f,\psi) = t_K^{-1} D_\cK^{-3/2} N^{3n+1/2} \prod_{\ell\mid N^-}\left(\frac{\ell+1}{\ell-1}\right) 2^{2n+2} i^{-1} N_{\cK/\Q}(\grb) \psi^\alg(x_\grp) b_N^{-2n}\alpha(f,f_b).
\end{equation*}
The following lemma allows us to pass from \eqref{CL-=Lp} to a relation between $\CL_p^\Sigma(f,\psi)^-$ and $L_p^\Sigma(f)$.

\begin{lem}\label{c(f,psi)-lem}
There exist $c\in R \left [\frac{1}{p} \right ]^\times$ and $\CU\in \Lambda^\times$ such that 
$c(f,\psi) = c\widehat\psi(\CU)$ for all $\psi\in \Sigma_{cc}$ with $n\equiv 0 \pmod{p-1}$.
\end{lem}

\begin{proof}
We let
$$
c = t_\cK^{-1} D_\cK^{-3/2} N^{1/2} \prod_{\ell\mid N^-}\left(\frac{\ell+1}{\ell-1}\right) 2^2 i^{-1} N_{\cK/\Q}(\grb) \alpha(f,f_B) \in R[\frac{1}{p}]^\times.
$$
We can assume that $x_\grb$ is chosen so that $x_{\grb,v} = 1 = x_{\grb,\bv}$. 
Let $\gamma_\grb\in \Gamma$ be the image of $\rec_\cK(x_\grb)$. Then $\psi^\alg(x_\grb)  = \psi(\gamma_\grb)$.
Let $\gamma_v\in \Gamma$ be a topological generator of the image of the inertia group $I_v$.
For $\psi$ to be crystalline at $v$ with Hodge-Tate weight $-n$ means that $\psi(\gamma_v) = \eps(\gamma_v)^n$.
Let $a_v\in \Z_p^\times$ such that $\eps(\gamma_v) = (1+p)^{a_v}$ and put 
$$
\CU = \gamma_v^{(2\log_p(2)+ 3\log_p(N) +\log_p(b_N))/a_v} \gamma_\grb \in \Gamma\subset \Lambda^\times.
$$
Here $\log_p$ is the Iwasawa $p$-adic logarithm. Also, we are viewing $b_N\in \cO_{\cK,(p)}^\times$ as an element of $\Z_p^\times$
via the identification $\Zp=\cO_{\cK,v}$ (which comes from the hypothesis that $p$ splits in $\cK$).
Then if $n\equiv 0 \pmod{p-1}$, $\CU$ satisfies
$$
\widehat\psi(\CU) = (1+p)^{n(2\log_p(2)+ 3\log_p(N) +\log_p(b_N))}\psi^\alg(x_\grb) = 2^{2n}N^{3n}b_N^n\psi^\alg(x_\grb),
$$
and so $c\widehat\psi(\CU) = c(f,\psi)$.
\end{proof}

Combining Lemma \ref{c(f,psi)-lem} with \eqref{CL-=Lp} we find that 
$\CL_p^\Sigma(f,\psi)^- = \widehat\psi(c\CU L_p^\Sigma(f))$ for all 
$\psi\in\Sigma_{cc}$ with $n\equiv 0\pmod{p-1}$.  If $\Sigma$ contains all the places dividing $ND_\cK$ not dividing $p$,
then $\CL_p^\Sigma(f,\psi)^- = \widehat\psi(\CL_p^\Sigma(f)^-)$ for all these $\psi$.
Since the set of kernels of the homomorphisms $\widehat\psi$ of $\Lambda_R$ for such $\psi$ 
are Zariski dense in $\mathrm{Spec}(\Lambda_R)$,  we conclude that we then have $\CL_p^\Sigma(f)^- = c\CU L_p(f)$ in
$\Lambda_R$. Hence:

\begin{coro}\label{Lpf = CLp^- almost} If $\Sigma$ contains all places of $\cK$ diving $ND_\cK$ that do not divide $p$, then 
$(L_p^\Sigma(f)) = (\CL_p^\Sigma(f)^-)$ in $\Lambda_R\left [\frac{1}{p} \right ]$.
\end{coro}

\begin{rmk}
It is possible to define $\cL_p^\Sigma(f)^-$ for all finite sets $\Sigma$ of places of $\cK$ not dividing $p$ and to directly show,
again making use of \eqref{CL-=Lp} and Lemma \ref{c(f,psi)-lem}, that the conclusion of Corollary \ref{Lpf = CLp^- almost} holds
for all such $\Sigma$. We will not need this.
\end{rmk}

%
%
\section{Main Conjectures}\label{bsd-acmc}
Recall that in Section \ref{subsubsec:iwasawa-selstruct} we defined the Iwasawa-theoretic 
Selmer groups $\H^1_{\SelMin^\Sigma}(\cK, M)$ and $\H^1_{\SelGr^\Sigma}(\cK, \CM)$. 
Let
$$
X_{\ac}^\Sigma(M) = \Hom_{\cO}(\H^1_{\SelMin^\Sigma}(\cK, M), L / \cO) \ \ \text{and} \ \
X_{\Gr}^\Sigma(\CM) = \Hom_{\cO}(\H^1_{\SelGr^\Sigma}(\cK, \CM), L / \cO)
$$ 
be their Pontrjagin duals.
We now recall the Iwasawa-Greenberg main conjectures for these groups together with some recent results towards proving these conjectures.
To do this, recall that $\Lambda = \O[\![\Gamma]\!]$ and $\Lambda_\cK = \O[\![\Gamma_\cK]\!]$ and that $R$ is the valuation ring of the completion
of the maximal unramified extension of $L$. Let $\Lambda_R = R[\![\Gamma]\!]$ and $\Lambda_{\cK,R} = R[\![\Gamma_\cK]\!]$.
The groups $X_{\ac}^\Sigma(M)$ and
$X_{\Gr}^\Sigma(\CM)$ are, respectively, finite $\Lambda$- and $\Lambda_\cK$-modules.  

\subsection{The Iwawasa--Greenberg main conjectures}\label{subsec:mainconj}
Let $\Sigma$ be a finite set of finite places of $\cK$ that do not divide $p$. The main conjectures are easy to state:

\begin{conj}[Main Conjecture for $\CM$]\label{conj:greenberg-conj}
Let $\CL_p^{\Sigma}(f)$ be Hida's two variable $p$-adic $L$-function recalled in Section~\ref{subsec:twovarpadicLfunc}. Then
$$
\chr_{\Lambda_\cK} \left ( X_{\Gr}^\Sigma(\CM)\right ) \Lambda_{\cK,R} = (\CL_p^\Sigma(f)) \subset \Lambda_{\cK,R}.
$$
\end{conj}

Note that implicit in the statement of this conjecture is that $\cL_p^\Sigma(f)$ belongs to $\Lambda_{\cK,R}$ and not just its field of fractions.  In Section \ref{subsec:twovarpadicLfunc} this is explained
to be known to hold at least when $\Sigma$ is sufficiently large.

\begin{conj}[Main Conjecture for $M$] \label{conj:ACconj}
Let $L_p^\Sigma(f)\in\Lambda_R$ be the $p$-adic $L$-function defined in Section~\ref{subsec:acpadicLfunc}. Then
$$
\chr_{\Lambda} \left ( X_{\ac}^{\Sigma}(M)\right )\Lambda_R = (L_p^\Sigma(f))\subset \Lambda_R.
$$
\end{conj}

\begin{rmk}\label{rmk:chr-ideal}
Note that we always have equalities of ideals
$$
\chr_{\Lambda_{\cK,R}}\left( X_{\Gr}^\Sigma(\CM)\otimes_{\Lambda_\cK}\Lambda_{\cK,R}\right) = \chr_{\Lambda_\cK} \left ( X_{\Gr}^\Sigma(\CM)\right ) \Lambda_{\cK,R} 
$$
and
$$
\chr_{\Lambda_{R}}\left( X_{\ac}^\Sigma(M)\otimes_{\Lambda}\Lambda_{R}\right) = \chr_{\Lambda} \left ( X_{\ac}^\Sigma(M)\right ) \Lambda_{R}.
$$
\end{rmk} 

Recently substantial progress has been made toward these conjectures in \cite{wan:rankin} and \cite{wan:rankin-ss} by
following the strategy from \cite{skinner-urban:gl2} of exploiting congruences between suitable Eisenstein series and cuspforms, this time on $\GU(3,1)$.

\begin{thm}[{\cite[Thm.1.1]{wan:rankin}},{\cite[Thm.1.1]{wan:rankin-ss}}]\label{MCthm1} 
Suppose \eqref{irredK}, \eqref{split}, \eqref{gen-H}, \eqref{good}, and \eqref{sqfree} hold. Suppose also that there is at least one prime divisor of $N$ non-split in $\mathcal{K}$
and that $\Sigma$ contains all places dividing $ND_\cK$. Then
$$
\text{$ \chr_{\Gamma_{\cK}}\left (X_{\Gr}^\Sigma(\CM) \right )\Lambda_{\cK,R}\left [1/p
\right ] \subseteq (\CL_p^\Sigma(f)) \subset \Lambda_{\cK,R}\left [1/p \right ]$.}
$$
\end{thm}

\noindent Combining this with Corollary \ref{bigSelmercontrol-cor} and Corollary \ref{Lpf = CLp^- almost}, we conclude:

\begin{thm}\label{ACthm1} Suppose \eqref{irredK}, \eqref{split}, \eqref{gen-H}, \eqref{good}, and \eqref{sqfree} hold. 
Suppose also that there is at least one prime divisor of $N$ non-split in $\mathcal{K}$ and that $\Sigma$ contains all places dividing $ND_\cK$. Then
$$
\text{$\chr_\Lambda \left (X_{\ac}^{\Sigma}(M)\right )\Lambda_R \left [1/p \right ] \subset (L_p^\Sigma(f)) \subset \Lambda_R \left [1/p \right ]$.}
$$
\end{thm}

This theorem can be strengthened upon combination with Proposition \ref{Selmersurjprop} (in particular, the surjectivity of \eqref{restrict-eq1}) and
Burungale's $\mu=0$ result of Proposition \ref{mu=0 prop}:

\begin{thm}\label{ACthm2} Suppose \eqref{irredK}, \eqref{split}, \eqref{gen-H}, \eqref{good}, and \eqref{sqfree} hold. 
Suppose also that there is at least one prime divisor of $N$ non-split in $\mathcal{K}$. Then
$$
\text{$\chr_\Lambda \left (X_{\ac}^{\Sigma}(M)\right )\Lambda_R \subset (L_p^\Sigma(f)) \subset \Lambda_R$.}
$$
\end{thm}

\begin{proof}
Let $\Sigma_1\subset\Sigma_2$ be two finite sets of places of $\cK$ dividing $p$ and with $\Sigma_2$ containing all the places dividing $ND_\cK$.
Then the surjectivity of \eqref{restrict-eq1} yields 
\begin{equation*}
\begin{split}
\chr_{\Lambda}\left(X_{\ac}^{\Sigma_2}(M)\right) & = \chr_{\Lambda}\left(X_{\ac}^{\Sigma_1}(M)\right)\chr_{\Lambda}
\left(\prod_{w\in\Sigma_2\backslash \Sigma_1} \Hom_{\cO}(\frac{\H^1(\cK_w,M)}{\H^1_{\SelMin}(\cK_w,M)},L/\cO)\right) \\
& =  \chr_{\Lambda}\left(X_{\ac}^{\Sigma_1}(M)\right)\prod_{w\in\Sigma_2\backslash \Sigma_1} \left( P_w(\eps^{-1}\Psi^{-1}(\Frob_w))\right).
\end{split}
\end{equation*}
Comparing this with the definition of $L_p^{\Sigma_2}(f)$ we then see that the hypothesis that $\Sigma$ contain the places
dividing $ND_\cK$ can be removed from
Theorem \ref{ACthm1}. In the case where $\Sigma=\emptyset$ the resulting inclusion of ideals can then be improved to an inclusion in
$\Lambda_R$ in light of the $\mu=0$ result of Proposition \ref{mu=0 prop}. The inclusion in $\Lambda_R$ for any $\Sigma$ then follows.
\end{proof}

\begin{rmk}
As a consequence of our results below, in some cases we will be able to improve the inclusion in Theorem \ref{ACthm2}
to an equality. That is, we will prove Conjecture \ref{conj:ACconj} in these cases.
\end{rmk}

\subsection{Consequences for the order of $\H^1_{\SelMin}(\cK, W)$}\label{subsec:ordselac}
We can now record the key result connecting the value $L_p(f,1)$ with the order
of $\H^1_{\SelMin}(\cK, W)$.

\begin{prop}\label{Lpf1=Sel}
Suppose \eqref{irredK}, \eqref{split}, \eqref{gen-H}, \eqref{good}, and \eqref{sqfree} hold. 
Suppose also that there is at least one prime divisor of $N$ non-split in $\mathcal{K}$. Suppose further that 
\eqref{crk1} and \eqref{surjp} also hold. 
Then $L_p(f,1)\neq 0$ and
$$
\#\O/(L_p(f,1)) \mid \# \H^1_{\SelMin}(\cK, W) \cdot C(W),
$$
where $C(W)= C^\emptyset(W)$ is as in Theorem~\ref{thm:control}.
\end{prop}

\begin{proof}
Let $(f_\ac(T)) = \chr_\Lambda\left (X_{\ac}(M)\right)$ with $f_\ac(T) \in \cO[\![T]\!]\cong \Lambda$, where the isomorphism
identifies $1+T$ with the chosen topological generator $\gamma$. Then by Theorem \ref{ACthm2}, 
$L_p(f)$ divides $F(T)$ in $R[\![T]\!]$. In particular $F(0)$ is an $R$-multiple of $L_p(f,1)$. On the other hand,
by Theorem \ref{thm:control}
$$
\#\cO/(f_\ac(0)) = \#H^1_{\SelMin}(\cK,W)\cdot C(W) < \infty.
$$
(The finiteness of the right-hand side was established in Proposition \ref{Sel-Sha prop}.)
The proposition follows.
\end{proof}

%
%
\section{Proof of the Main Theorem}\label{bsd-mainthm}

In this section we piece together the results from the previous sections to prove the main result of the paper, Theorem \ref{thm:main}.
We therefore take $E/\Q$ to be an elliptic curve as in that theorem. In particular, we assume:
\begin{itemize}
\item $E$ is a semistable elliptic curve with square-free conductor $N$;
\item $E$ has good reduction at the the prime $p$ (i.e., $p\nmid N$); 
\item if $p=3$ and $E$ has supersingular reduction at $p$, then $a_p(E)=0$;
\item the residual representation $\bar\rho_{E,p}:G_\Q\ra \Aut(E[p])$ is irreducible;
\item  $\ord_{s=1}L(E,s) = 1$.  
\end{itemize}
The results of Gross--Zagier and Kolyvagin then imply that $\rk_\Z E(\Q) =1$ and $\Sha(E/\Q)$ 
is finite. To prove Theorem \ref{thm:main} we must show that the same power of $p$ appears in both sides of \eqref{eq:bsd-formula}.
Since Conjecture \ref{conj:bsd} is isogeny invariant  --  more precisely, the ratio of both sides of 
\eqref{eq:bsd-formula}, when both are finite, is an invariant of the isogeny class of $E$ (cf.~\cite[Thm.~2.1]{tate:bsd}) -- we may further assume:
\begin{itemize}
\item $E$ admits an optimal quotient map $\pi:J(X_0(N))\rightarrow E$ (that is, such that the kernel of $\pi$ is connected).
\end{itemize}   

\subsection{The Birch and Swinnerton-Dyer Conjecture} We will eventually deduce both the conjectured upper and lower bounds
on $\#\Sha(E/\Q)[p^\infty]$ from corresponding upper and lower bounds for $\#\Sha(E/\cK)[p^\infty]$ for suitable imaginary quadratic fields $\cK$.
For this reason we find it helpful to recall the general Birch and Swinerton-Dyer conjecture for an elliptic curve over a number field. As
stated by Tate \cite[(A)~ and~(B)]{tate:bsd}, this is:

\begin{conj}[general Birch and Swinnerton-Dyer Conjecture]\label{conj:bsd-gen}
Let $F$ be a number field and let $\cE /F$ be an elliptic curve over $F$. 
\begin{itemize}
\item[\rm (a)] The Hasse--Weil $L$-function $L(\cE/F, s)$ has an analytic continuation to 
the entire complex plane and $\ord_{s = 1}L(\cE/F, s) = \rk_{\Z} \cE(F)$;
\item[\rm (b)] The Tate-Shafarevich group $\Sha(\cE/F)$ has finite order, and 
\begin{equation}\label{eq:bsd-formula-gen}
\frac{L^{(r)}(\cE/F, 1)}{r! \cdot \Omega_{\cE/F} \cdot \Reg(\cE/F) \cdot |\Delta_F|^{-1/2}} = \frac{\# \Sha(\cE/F) \cdot \prod_{v \nmid\infty} c_v(\cE/F)}{\# \cE(F)_{\tors}^2}, 
\end{equation}
where $r = \ord_{s = 1}L(\cE/F, s)$, 
$c_v(\cE/F) = [\cE(F_v) : \cE^0(F_v)]$ is the Tamagawa number at $v$ for a finite place $v$ of $F$, $\Reg(\cE/F)$ is the regulator 
of the N\'eron-Tate height pairing on $\cE(F)$, $\Delta_F$ is the discriminant of $F$, and
$\Omega_{\cE/F}\in\C^\times$ is the period defined by 
\begin{equation}\label{eq:per}
\Omega_{\cE/F} = \mathrm{N}_{F/\Q}(\gra_\omega)\cdot \prod_{\substack{v \mid \infty \\ v \text{-real}}} \int_{\cE(F_v)} |\omega| \cdot \prod_{\substack{v \mid \infty \\ v \text{-complex}}} \left ( 2 \cdot  \int_{\cE(F_v)} \omega \wedge \overline{\omega} \right ). 
\end{equation}
Here $\omega \in \Omega^1(\tilde\cE/\cO_F)$ is any non-zero differential on the N\'eron model $\tilde\cE$ of $\cE$ over $\O_F$, and 
$\mathfrak{a}_\omega\subset F$ is the fractional ideal such that  $\mathfrak{a}_\omega\cdot\omega = \Omega^1(\tilde\cE/\O_F)$.
Also, for a finite place $v$, $\cE^0(F_v) \subset \cE(F_v)$ denotes the subgroup of local points that specialize to the identity component of the N\'eron model of $\cE$ at the place $v$.
\end{itemize}
\end{conj} 

\noindent When $F=\Q$,  we generally write $\Omega_\cE$ for $\Omega_{\cE/\Q}$.

\subsection{The Birch--Swinnertion-Dyer formula for rank zero elliptic curves}
To pass from the expected upper and lower bounds for $\#\Sha(E/\cK)[p^\infty]$ for suitable imaginary quadratic fields $\cK=\sqrt{D}$ ($D<0$) to the expected upper or lower bound
for $\#\Sha(E/\Q)[p^\infty]$, we will need to appeal to known results for the BSD formula for the $\cK$-twists\footnote{If $E$ has Weierstrass equation $y^2=x^3+Ax+B$, then 
the $\cK$-twist of $E$ is the curve $E^D/\Q$ having 
Weierstrass equation $Dy^2=x^3+Ax+B$. If $f$ is the newform associated with $E$, then the newform associated with $E^D$ is just the 
newform $f_\cK$ associated with the twist of $f$ by the quadratic character $\chi_\cK$ of $\cK$.} 
$E^D$ of $E$. The fields $\cK$ will always be chosen so that the groups
$E^D(\Q)$ have rank $0$.  We therefore recall the known results about the $p$-part of the BSD formula for rank $0$ curves.

In the theorem below we summarize the already known results on the Birch and Swinnerton-Dyer conjecture for both 
ordinary and supersingular elliptic curves of analytic rank zero. The inequality of part (i) of the theorem is a consequence of Kato's groundbreaking work on an Euler system for
elliptic curves (see~\cite[Thm.4.8]{perrin-riou:expmath}). The equality of part (ii) for the ordinary case is proved in  \cite[Thm.2]{skinner-urban:gl2}, \cite[Thm.C]{skinner:multred} 
as a consequence of the proof Iwasawa main conjecture for $\GL_2$, and that of part (iii) for the supersingular case is a consequence of the proof of Kobayashi's main conjecture 
\cite[Cor.4.8]{wan:kobayashi}.

\begin{thm}\label{thm:rank0} Let $\cE/\Q$ be an elliptic curve with good or multiplicative reduction at the odd prime $p$ and suppose that $\rhobar_{\cE, p} \colon \Gal(\overline{\Q} / \Q) \ra \Aut(\cE[p])$ is irreducible. Suppose $L(\cE,1)\neq 0$.
\begin{itemize}
\item[\rm (i)] One has
\begin{equation}\label{sha:rank0ineq}
\ord_p(\# \Sha(\cE/\Q)[p^\infty]) \leq \ord_p \left ( \frac{L(\cE/\Q, 1)}{\Omega_{\cE} \cdot \prod_{\ell} c_{\ell}(\cE)} \right ). 
\end{equation}
\item[\rm (ii)] If $\cE$ has good ordinary or multiplicative reduction at $p$ and there exists a prime $q$ of multiplicative reduction for $E$ at which the representation $\rhobar_{\cE, p}$ is ramified, then 
\begin{equation}\label{sha:rank0ord}
\ord_p(\# \Sha(\cE/\Q)[p^\infty]) = \ord_p \left ( \frac{L(\cE/\Q, 1)}{\Omega_{\cE} \cdot \prod_{\ell} c_{\ell}(\cE)} \right ).  
\end{equation}
\item[\rm (iii)] If $\cE$ is semistable or a twist of a semistable elliptic curve by a quadratic character that is unramified at the primes dividing the conductor of the semistable curve
and if $\cE$ has supersingular reduction at $p$ with $a_p(\cE)=0$, then 
\begin{equation}\label{sha-twist:exact}
\ord_p(\# \Sha(\cE/\Q)[p^\infty]) = \ord_p \left ( \frac{L(\cE/\Q, 1)}{\Omega_{\cE} \cdot \prod_{\ell} c_{\ell}(\cE)} \right ).  
\end{equation}
\end{itemize}
\end{thm}

\begin{rmk} Note that the condition that $a_p(\cE)=0$ in part (iii) of this theorem is superfluous if $p\geq 5$.
\end{rmk}

\begin{rmk}\label{rmk:wan}  Part (iii) of the above theorem is slightly more general than the  
result cited in \cite{wan:kobayashi}, where the elliptic curve is assumed to be  
semistable. However the same proof extends to the case of the quadratic twist of a semistable curve. 
The only  reason for putting the semistable assumption there is to avoid local  
triple product integrals for supercuspidal representations at split  
primes in \cite{wan:rankin}, \cite{wan:rankin-ss}, which is excluded under the assumption of  the above theorem. 
The results of Hsieh and Hung  on  
non-vanishing (or non-vanishing modulo $p$) of $L$-values and  
vanishing of $\mu$-invariants are also valid in the more general setting of the quadratic twist.
\end{rmk}

\subsection{Comparison of Tamagawa numbers and periods for quadratic twists}
We will derive both upper and lower bounds on $\#\Sha(E/\Q)[p^\infty]$ from corresponding bounds on $\#\Sha(E/\cK)[p^\infty]$ for suitable choices of imaginary quadratic fields 
$\cK=\Q ( \sqrt{-D} )$.
In order to derive from this the exact upper and lower bounds predicted by Conjecture \ref{conj:bsd}, we recall here the relations between the Tamagawa numbers and periods of $E$ over $\cK$ and the Tamagawa numbers and periods of $E$ and its quadratic twist\footnote{If $E$ has Weierstrass equation $y^2=x^3+Ax+B$, then by $E^D$ we mean the elliptic curve over $\Q$ having 
Weierstrass equation $-Dy^2=x^3+Ax+B$. This is just the $\cK$-twist of $E$. If $f$ is the newform associated with $E$, then the newform associated with $E^{D}$ is just the 
newform $f_\cK$ associated with the twist of $f$ by the quadratic character of $\cK$.} $E^{D}$.
These are discussed in detail in \cite{skinner-zhang}.  

\subsubsection{Tamagawa numbers} 
We recall \cite[Cor.9.2.]{skinner-zhang} that 
\begin{equation}\label{eq:tamK}
\ord_p\left(\prod_{w} c_w(E/\cK)\right) = \ord_p\left(\prod_{\ell} c_{\ell}(E/\Q) \cdot \prod_{\ell} c_{\ell}(E^D / \Q)\right).
\end{equation}
On the left-hand side $w$ is running over finite places of the imaginary quadratic field $\cK=\Q(\sqrt{D})$ of discriminant $D<0$, and on
the right-hand side $\ell$ is running over all primes. Here we have also used that $E$ has good reduction at $p$ and so has Tamagawa
number $1$ at any place dividing $p$.

\subsubsection{Periods and comparisons} 
For the elliptic curve $E/\Q$, one defines its real period to be 
$$
\ds \Omega_{E} = \int_{E(\R)} |\omega|.
$$ 
Here $\omega_E$ is a N\'eron differential (a $\Z$-basis
of the module of differentials of the N\'eron model of $E$ over $\Z$). If $f \in S_2(\Gamma_0(N))$ is the normalized cuspidal eigenform 
associated with $E$ (that is, satisfying $L(f,s) = L(E,s)$) then, as explained in \cite[\S 9.2]{skinner-zhang}, one also
defines canonical periods $\Omega_f^+,\Omega_f^- \in \C^\times$; these are defined up to $\Z_{(p)}^\times$-multiples (this can be
refined to $\Z^\times$-multiples, but we do not need this here). Furthermore, as recalled in \cite[\S 9.3]{skinner-zhang}, one also 
has the {congruence (or Hida) period}
$$
\Omega_f^{\mathrm{cong}}= \frac{\langle f, f\rangle}{\eta_f} \in \overline{\Q}_p^\times,
$$
where $\langle f, f\rangle$ is the Petersson norm and $\eta_f$ is the congruence number of $f$ recalled in {\it loc.~cit.};
to make sense of this definition we use our chosen isomorphism $\bQp\cong\C$. Recall that $\Omega_f$ is only 
well-defined up to $\Z_p^\times$-multiple. We next recall the relations between these various periods.

It is explained in \cite[\S 3.3]{skinner:multred} (see also \cite[Prop.~3.1]{greenberg-vatsal}) that 
\begin{equation}\label{eq:omegaE=omegaf}
\Omega_E = -2\pi i \Omega_f^+ \ \ \text{and} \ \ \Omega_{E^D}=-2\pi i\Omega_{f_\cK}^+, 
\end{equation}
up to $\Z_{(p)}^\times$-multiples. 
By \cite[Lem.~9.5]{skinner-zhang} the congruence period can be chosen to satisfy 
\begin{equation}\label{eq:congperiod}
\Omega_f^{\mathrm{cong}} = i(2\pi)^2\Omega_f^+\Omega_f^-.
\end{equation}
Let $f_\cK$ be the newform associated with the twist of $f$ by the quadratic character of the imaginary quadratic field $\cK$.
Under the hypothesis that $\bar\rho_{E,p}$ is irreducible, it is shown in \cite[Lem.~9.6]{skinner-zhang} that if $p\nmid D$ then 
\begin{equation}\label{eq:period of twist}
\Omega_{f}^{\pm} = \sqrt{-D}\cdot \Omega_{f_\cK}^{\mp},
\end{equation}
up to a $\Z_p^\times$-multiple. Note that the statement of {\it loc.~cit.} omits the factor $\sqrt{-D}$ (which equals the Gauss sum $\tau(\chi_{\cK})$ up to sign); this is because it is assumed
there that $p$ splits in $\cK$ (also assumed here) and so $\sqrt{-D}\in \Z_p^\times$.  In particular, combining \eqref{eq:omegaE=omegaf},
\eqref{eq:congperiod}, and \eqref{eq:period of twist} we find:
\begin{equation}\label{eq:period product}
\Omega^{\mathrm{cong}}_f = \sqrt{|D|} \cdot \Omega_E \Omega_{E^D}.
\end{equation}

\begin{rmk}\label{rem:maninconst}
The comparison between the periods $\Omega_{E}$ of $\Omega_{E^D}$ of the elliptic curves and 
$\Omega_f^+$ and $\Omega_{f_\cK}^+$ is often done in terms of what is known as the Manin constants, as explained in \cite[\S 3]{greenberg-vatsal}. 
It is known by a result of Mazur that if $p\nmid 2ND$, then $p$ does not divide either of the Manin constants (see, e.g., \cite[\S 1]{jetchev:tamagawa}).  
\end{rmk}

\subsection{The final argument} 
Let $f \in S_2(\Gamma_0(N))$ be the Hecke eigenform associated to $E$. 
Then $(V,T,W) = (V_f,T_f,W_f) = (V_p E, T_pE, E[p^\infty])$, $\cO=\Zp$, and $\overline{V} \cong E[p]$. 
Note that there exists at least one prime $q \mid N$ such that the mod $p$ representation $\bar\rho_{E,p}$ is ramified at $q$. 
If not, then Ribet's level lowering theorem \cite[Thm.1.1]{ribet:modreps} yields a cuspform $g$ of weight 2 and level 1 with mod $p$ residual representation
isomorphic to $\bar\rho_{E,p}$ (we apply Ribet's theorem to remove, one-by-one, the primes dividing $N$). This is a contradiction as there are no cuspidal eigenforms
of weight $2$ and level $1$.  Let $N = q_1 \dots q_r$ with $q_1 = q$. 

\subsubsection{Lower bounds on $\# \Sha(E/\Q)[p^\infty]$}\label{subsubsec:lower}
We first choose an auxiliary imaginary quadratic field $\cK' = \Q(\sqrt{D'})$ of discriminant $D'<0$ such that 
\begin{itemize}
\item[(a)] $N$ and $\cK'$ satisfy \eqref{gen-H};
\item[(b)] $q$ is either inert or ramified in $\cK'$;
\item[(c)] $p$ splits in $\cK'$;
\item[(d)] $L(E^{D'},1)\neq 0$.
\end{itemize}
It is easy to find $\cK'$ such that (a), (b) and (c) hold. Indeed, if $r \geq 2$ (i.e., $N$ has at least two prime divisors), then we can even guarantee that 
\eqref{H} holds by requiring that $q$ be inert in $\cK'$. In the case where $N = q$ is prime, we take $N^- = 1$ and $N^+ = q$, which only satisfies \eqref{gen-H}.
We note that for any $\cK'$ for which (a), (b), and (c) hold, the root number $w(E/\cK')$ of $E/\cK'$ is $-1$ (cf.~\eqref{sign-1}). As $w(E/\Q)=-1$ by hypothesis 
(this is a consequence of $\ord_{s=1}L(E,s)$ being odd) and 
$w(E/\cK')=w(E/\Q)w(E^{D'}/\Q)$, we have $w(E^{D'})=+1$. As (a), (b), and (c) impose only finitely many congruence conditions on the
discriminant of $\cK'$, it then follows from a result of Friedberg and Hoffstein \cite[Thm.~B]{friedberg-hoffstein:annals} that $\cK'$ can be chosen so that (d) also holds.
Note that the condition (d) means that $\ord_{s=1} L(E/\cK',s) = 1$. In particular, by the work of Gross--Zagier and Kolyvagin we know that $E(\cK')$ has
rank one and that $\Sha(E/\cK')$ is finite.

To take a first step towards a lower bound on $\#\Sha(E/\cK')[p^\infty]$ we want to appeal to the Proposition~\ref{Lpf1=Sel}. So we first check that the hypotheses of that proposition 
hold with $\cK=\cK'$. The conditions \eqref{split}, \eqref{gen-H}, \eqref{good}, and \eqref{sqfree} either follow immediately from our hypotheses on $E$ or from the choice of $\cK'$.
The condition \eqref{irredK} is an easy consequence of the hypotheses that $\bar\rho_{E,p}$ is irreducible and that $\bar\rho_{E,p}$ is ramified at the prime $q=q_1\mid\mid N$
(see~\cite[Lem.~2.8.1]{skinner:conversegz}). Finally, note that \eqref{crk1} and \eqref{surjp} also hold as $\rk E(\cK')=1$ and $\#\Sha(E/\cK')[p^\infty]<\infty$ 
(cf.~Section \ref{subsec:newform-app}). 

By Proposition~\ref{Lpf1=Sel} (with $\cK=\cK'$) 
we have
\begin{equation}\label{eq:lowerbound-1}
\ord_p\left(L_p(f,1)\right)\leq \ord_p\left(\#\H^1_{\SelMin}(\cK',E[p^\infty])\cdot C(E[p^\infty])\right).
\end{equation}
Let $z_{\cK'} = z_{\cK'}^{N^+,N^-}\in E(\cK')$ be the Heegner point defined in Section \ref{subsec:cmpts}. (Note that the 
hypothesis \eqref{irredK} ensures that there are many good auxiliary primes $\ell_0$.) 
As the hypotheses of Proposition~\ref{prop:brooks-Af} are also clearly satisfied, the left-hand side of \eqref{eq:lowerbound-1} is
$$
\ord_p(L_p(f,1)) = 2\cdot \ord_p\left(\frac{1+p-a_p}{p} \log_{\omega_E}(z_{\cK'})\right).
$$
As $L_p(f,1)\neq 0$ (also by Proposition \ref{Lpf1=Sel}), this implies that $z_{\cK'}$ has infinite order (of course, this is also a consequence of the 
general Gross--Zagier formula for $z_{\cK'}$). Let 
$$
m_{\cK'} = [E(\cK'):\Z\cdot z_{\cK'}].
$$
This index is finite.  The hypotheses of Section \ref{subsec:newform-app} are also all satisfied, so
it follows from \eqref{eq:modabvar-ec} and the definition of $C(E[p^\infty])$ that the left-hand side of \eqref{eq:lowerbound-1} is
\begin{equation*}\begin{split}
\ord_p\left(\#\H^1_{\SelMin}(\cK',E[p^\infty]\cdot C(E[p^\infty])\right) = \ord_p & \left(\#\Sha(E/\cK')\right) - 2\cdot\ord_p(m_{\cK'})  \\ &+ 2\cdot\ord_p\left(\frac{1+p-a_p}{p} \log_{\omega_E}(z_{\cK'})\right) \\
& + \ord_p \left (\prod_{w\mid N^+} c_w(E/\cK') \right ).
\end{split}\end{equation*}
We then conclude from all this that 
\begin{equation}\label{eq:shalowerK-1}
\ord_p\left(\#\Sha(E/\cK')[p^\infty]\right) \geq  2\cdot\ord_p(m_{\cK'}) - \ord_p \left (\prod_{w\mid N^+} c_w(E/\cK') \right).
\end{equation}

To pass from the inequality \eqref{eq:shalowerK-1} to one involving the derivative $L'(E/\cK',1)$ we use the variant of the Gross--Zagier formula
 for the Heegner point\footnote{Though stated in terms of $z_{\cK'}$ in \cite{zhang:selheeg}, the formula in {\it loc.~cit.} is a special case of a formula in 
 \cite{cai-tian:gz}, which is expressed in terms of a `Heegner point' defined using the map $1/m\cdot \tilde\iota_{N^+,N^-}$ from \eqref{eq:cohtriv1} in place of $\iota_{N^+,N^-}$. However, 
 it is clear from \eqref{eq:cohtriv-relation} that replacing this point with $z_{\cK'}$ only changes the formula by a $p$-adic unit.}
$z_{\cK'}$ \cite[p.~245]{zhang:selheeg}: up to a unit in $\Z_p^\times$,
\begin{equation*}
\sqrt{|D'|}\frac{L'(E/\cK',1)}{\Omega_f^{\rm cong}} = \frac{\delta(N,1)}{\delta(N^+,N^-)} \frac{\langle z_{\cK'},z_{\cK'}\rangle_{NT}}{c_E^2},
\end{equation*}
where $c_E\in \Z$ is the Manin constant of $E$ (so $p\nmid c_E$ in this case; see Remark \ref{rem:maninconst}) and $\delta(N,1)$, $\delta(N^+,N^-)$
are defined in {\it loc.~cit.} In particular, as explained on the same page of {\it loc.~cit.},
$$
\delta_{N^+,N^-} : = \frac{\delta(N,1)}{\delta(N^+,N^-)}  = \prod_{\ell\mid N^-} c_\ell(E/\cK')
$$
up to a unit in $\Z_p^\times$. 
We can then rewrite the Gross--Zagier formula for $z_{\cK'}$ as an equality up to $p$-adic unit:
\begin{equation}\label{eq:gz for K'}
\frac{L'(E,1)}{\Omega_E\cdot\Reg(E/\Q)}\cdot \frac{L(E^{D'},1)}{\Omega_{E^{D'}}} = m_{\cK'}^2 \cdot \prod_{\ell\mid N^-} c_\ell(E/\cK').
\end{equation}
Here we have used that $\langle z_{\cK'},z_{\cK'}\rangle_{NT} = m_{\cK'} \Reg(E/\cK')$, $\Reg(E/\cK') = \Reg(E/\Q)$ 
(as $E^{D'}(\Q)$ is finite), $L'(E/\cK') = L'(E,1)L(E^{D'},1)$, and the period relation \eqref{eq:period product}.
From \eqref{eq:gz for K'} we obtain
\begin{equation*}\label{eq:gz for m}
2\cdot\ord_p(m_{\cK'}) = \ord_p\left(\frac{L'(E,1)}{\Omega_E\cdot\Reg(E/\Q)}\cdot \frac{L(E^{D'},1)}{\Omega_{E^{D'}}}\right) - \ord_p\left(\prod_{\ell\mid N^-} c_\ell(E/\cK')\right).
\end{equation*}
Combining this with \eqref{eq:shalowerK-1} and \eqref{eq:tamK} (with $\cK=\cK'$) we obtain
\begin{equation*}\label{eq:shalowerK-2}
\ord_p(\#\Sha(E/\cK')[p^\infty]) \geq \ord_p\left(\frac{L'(E,1)}{\Omega_E\cdot\Reg(E/\Q)\prod_{\ell}c_\ell(E/\Q)}\cdot \frac{L(E^{D'},1)}{\Omega_{E^{D'}}\prod_\ell c_\ell(E^{D'}/\Q)}\right).
\end{equation*}
As $\Sha(E/\cK')[p^\infty] \cong \Sha(E/\Q)[p^\infty]\oplus \Sha(E^{D'}/\Q)[p^\infty]$, from 
the above lower bound on $\ord_p(\#\Sha(E/\cK')[p^\infty])$ and Theorem \ref{thm:rank0}(i) (really \eqref{sha:rank0ineq}
for $\cE = E^{D'}$) we conclude that
\begin{equation}\label{eq:shalower}
\ord_p(\#\Sha(E/\Q)[p^\infty]) \geq \ord_p\left(\frac{L'(E,1)}{\Omega_E\cdot\Reg(E/\Q)\prod_{\ell}c_\ell(E/\Q)}\right).
\end{equation}
That is, we have proved the exact conjectured lower bound on $\#\Sha(E/\Q)[p^\infty]$.

\subsubsection{Upper bounds on $\# \Sha(E/\Q)[p^\infty]$}\label{subsubsec:upper} 
Recall that $N=q_1\cdots q_r$ with $q_1=q$ such that $p\nmid c_q(E/\Q)$. 
If $r$ is odd, let $N^+ = q_1$ and $\ds N^- = N/q_1$. If $r$ is even, let $N^+ = 1$ and $N^- = N$.
We choose a second auxiliary imaginary quadratic field $\Q(\sqrt{D''})$ of discriminant $D''<0$ such that
\begin{itemize}
\item[(a)] the primes dividing $N^+$ split in $\cK''$;
\item[(b)] the primes dividing $N^-$ are all inert in $\cK''$;
\item[(c)] $p$ splits in $\cK''$;
\item[(d)] $L(E^{D''}, 1) \ne 0$. 
\end{itemize}
Note that \eqref{H} holds for $N$ and any $\cK''$ satisfying (a), (b), and (c). As with the choice of $\cK'$ above, the root number of the quadratic twist $E^{D''}$ is $+1$ and the result
of Friedberg and Hoffstein ensures that $\cK''$ can be chosen so that (d) also holds. 

Note that by (d) we have $\ord_{s=1}L(E/\cK'',s) = 1$. Let $z_{\cK''} = z_{\cK''}^{N^+,N^-}\in E(\cK'')$ be the Heegner point defined in Section \ref{subsec:cmpts} and 
let $m_{\cK''} = [E(\cK''):\Z\cdot z_{\cK''}]$. From  Theorem \ref{thm:shaK} we obtain
\begin{equation*}
\ord_p(\#\Sha(E/\cK'')[p^\infty]) \leq 2\cdot\ord_p(m_{\cK''}).
\end{equation*}
From the Gross--Zagier formula for $z_{\cK''}$ we have, just as we did for $z_{\cK'}$, that
\begin{equation*}
2\cdot\ord_p(m_{\cK''}) = \ord_p\left(\frac{L'(E,1)}{\Omega_E\cdot\Reg(E/\Q)}\cdot \frac{L(E^{D''},1)}{\Omega_{E^{D''}}}\right) - \ord_p\left(\prod_{\ell\mid N^-} c_\ell(E/\cK'')\right).
\end{equation*}
For our choice of $\cK''$ there are no $w\mid N^+$ such that $p\mid c_w(E/\cK'')$, and so the product on the left-hand side can be replaced by a product over all $w$.
Combining the above equality with the preceding inequality and appealing to \eqref{eq:tamK} (for $\cK=\cK''$) we get
\begin{equation*}
\ord_p(\#\Sha(E/\cK'')[p^\infty]) \leq \ord_p\left(\frac{L'(E,1)}{\Omega_E\cdot\Reg(E/\Q)\prod_{\ell}c_\ell(E/\Q)}\cdot \frac{L(E^{D''},1)}{\Omega_{E^{D''}}\prod_\ell c_\ell(E^{D''}/\Q)}\right)
\end{equation*}
Appealing to part (ii) or (iii) of Theorem \ref{thm:rank0} (for $\cE=E^{D''}$) then yields
\begin{equation}\label{eq:shaupper}
\ord_p(\#\Sha(E/\Q)[p^\infty]) \leq \ord_p\left(\frac{L'(E,1)}{\Omega_E\cdot\Reg(E/\Q)\prod_{\ell}c_\ell(E/\Q)}\right).
\end{equation}
That is, we have proved the exact conjectured upper bound on $\#\Sha(E/\Q)[p^\infty]$.

\subsubsection{The final step: equality} Combining \eqref{eq:shalower} and \eqref{eq:shaupper} gives
$$
\ord_p(\#\Sha(E/\Q)[p^\infty]) = \ord_p\left(\frac{L'(E,1)}{\Omega_E\cdot\Reg(E/\Q)\prod_{\ell}c_\ell(E/\Q)}\right),
$$
proving Theorem \ref{thm:main}.

\subsubsection{Some remarks on the hypotheses in Theorem \ref{thm:main}}
\label{subsubsec:remarks} 
We make a few comments on how the various hypotheses on Theorem \ref{thm:main} intervene in its proof
and some remarks on possible generalizations.
\begin{itemize}
\item[(i)] The requirement that $N$ be square-free is made in \cite{brooks:shimura} and so appears as a hypothesis of Proposition \ref{prop:brooks}.
However, this condition can likely be dropped in light of the main results of \cite{liu-zhang-zhang}. The formulas in the latter are not as explicit as the result in 
\cite{brooks:shimura} but can made so just as the general Gross--Zagier formula is made explicit in \cite{cai-tian:gz}. 

\item[(ii)] The requirement that $N$ be square-free is also made in \cite{wan:rankin} and \cite{wan:kobayashi}, but as we have already indicated (see Remark \ref{rmk:wan})
this can be replaced with the requirement that the cuspidal automorphic representation $\pi=\otimes_{\ell\leq\infty}\pi$ associated with $E$ is either a principal series or a Steinberg representation
at each $\ell\mid N$.  Further relaxing of this conditions requires a better understanding of certain local triple product integrals.

\item[(iii)] If $p=3$ and $E$ has supersingular reduction at $p$, then we have required that $a_p(E)=0$. This is only because the same hypothesis is made in \cite{wan:kobayashi}
and so $p$-part of the BSD formula in the rank $0$ case (needed for $L(E^{D''},1)$ in the argument deducing the exact upper bound) is only known for supersingular $p$ when $a_p(E)=0$.

\item[(iv)] As much as possible, we have worked in the context of the Selmer groups of a general newform $f\in S_2(\Gamma_0(N))$ of weight $2$ and trivial character. The ultimate restriction
to the case of an elliptic curve is made for two reasons: (1) the lack of a precise reference for the upper-bound on $\#\Sha(A_f/\cK)[\grp^\infty]$ coming from the Euler system of Heegner points
(the generalization of Theorem \ref{thm:shaK}), and (2) the lack of a general result about the $\grp$-part of the special value $L(f_{\cK},1)/2\pi i\Omega_{f_{\cK}}^+$ (when $L(f_{\cK},1)\neq 0$)
when $f$ is not ordinary at $\grp$.   We expect that both of these issues will be addressed in forthcoming work.

\item[(v)] The condition that $p$ be a prime of good reduction can likely be relaxed to at least a prime of multiplicative reduction. The rank $0$ special value formulas are proved, for example,
in \cite{skinner:multred}, and the results of \cite{castella:exceptional} and \cite{liu-zhang-zhang} allows $p$ to be a prime of multplicative reduction.

\item[(vi)] That $p$ is odd is, of course, an essential hypothesis of many of the results used in the course of our proof. 
\end{itemize}

\section*{Acknowledgements}
We thank Ernest H. Brooks for a careful reading of an earlier draft of this paper and various helpful discussions. We are grateful to Christophe Cornut for his interest and multiple useful conversations on the problem. 

The work of the first named author (D.J.) was supported by Swiss National Science Foundation professorship grant PP00P2-144658. The work of the second named author (C.S.) was partially supported by the grants DMS-0758379 and DMS-1301842 from the National Science Foundation and by
the Simons Investigator grant \#376203 from the Simons Foundation. He also gratefully acknowledges the hospitality of the California Institute of Technology, where
part of this work was carried out during an extended visit in 2014.

\bibliographystyle{amsalpha}
\bibliography{biblio}

\def\cprime{$'$}
\providecommand{\bysame}{\leavevmode\hbox to3em{\hrulefill}\thinspace}
\providecommand{\MR}{\relax\ifhmode\unskip\space\fi MR }
\providecommand{\MRhref}[2]{%
  \href{http://www.ams.org/mathscinet-getitem?mr=#1}{#2}
}
\providecommand{\href}[2]{#2}
\begin{thebibliography}{Wan14b}

\bibitem[AU96]{abbes-ullmo:manin}
A.~Abbes and E.~Ullmo, \emph{\`{A} propos de la conjecture de {M}anin pour les
  courbes elliptiques modulaires}, Compositio Math. \textbf{103} (1996), no.~3,
  269--286.

\bibitem[BBV15]{berti-bertolini-venerucci}
A.~Berti, M.~Bertolini, and R.~Venerucci, \emph{Congruences between modular
  forms and the {B}irch and {S}winnerton-{D}yer conjecture}, preprint (2015).

\bibitem[BDP13]{bertolini-darmon-prasanna}
M.~Bertolini, H.~Darmon, and K.~Prasanna, \emph{Generalized {H}eegner cycles
  and {$p$}-adic {R}ankin {$L$}-series}, Duke Math. J. \textbf{162} (2013),
  no.~6, 1033--1148.

\bibitem[Bir71]{birch:bsd}
B.\thinspace{}J. Birch, \emph{Elliptic curves over \protect{${\mathbf{Q}}$:
  {A}} progress report}, 1969 Number Theory Institute (Proc. Sympos. Pure
  Math., Vol. XX, State Univ. New York, Stony Brook, N.Y., 1969), Amer. Math.
  Soc., Providence, R.I., 1971, pp.~396--400.

\bibitem[BK90]{bloch-kato}
S.~Bloch and K.~Kato, \emph{{$L$}-functions and {T}amagawa numbers of motives},
  The {G}rothendieck {F}estschrift, {V}ol.\ {I}, Progr. Math., vol.~86,
  Birkh\"auser Boston, Boston, MA, 1990, pp.~333--400.

\bibitem[BLR90]{neronmodels}
S.~Bosch, W.~L{\"u}tkebohmert, and M.~Raynaud, \emph{N\'eron models},
  Springer-Verlag, Berlin, 1990. \MR{91i:14034}

\bibitem[Bro14]{brooks:shimura}
H.~Brooks, \emph{Shimura curves and special values of $p$-adic $l$-functions},
  to appear in Intl. Math. Res. Notices (2014).

\bibitem[Bur14]{burungale:mu2}
A.~Burungale, \emph{On the non-triviality of generalized {H}eegner cycles
  modulo {$p$}, {II}: {S}himura curves}, preprint (2014).

\bibitem[Buz97]{buzzard:integral}
K.~Buzzard, \emph{Integral models of certain {S}himura curves}, Duke Math. J.
  \textbf{87} (1997), no.~3, 591--612.

\bibitem[Cas15]{castella:exceptional}
F.~Castella, \emph{On the exceptional specializations of big {H}eegner points},
  preprint (2015).

\bibitem[CCO14]{chai-conrad-oort}
C.-L. Chai, B.~Conrad, and F.~Oort, \emph{Complex multiplication and lifting
  problems}, Mathematical Surveys and Monographs, vol. 195, American
  Mathematical Society, Providence, RI, 2014.

\bibitem[CST14]{cai-tian:gz}
L.~Cai, J.~Shu, and Y.~Tian, \emph{Explicit {G}ross-{Z}agier and {W}aldspurger
  formulae}, Algebra Number Theory \textbf{8} (2014), no.~10, 2523--2572.

\bibitem[FH95]{friedberg-hoffstein:annals}
S.~Friedberg and J.~Hoffstein, \emph{Nonvanishing theorems for automorphic
  {$L$}-functions on {${\rm GL}(2)$}}, Ann. of Math. (2) \textbf{142} (1995),
  no.~2, 385--423.

\bibitem[FM95]{fontainemazur:geom}
J.-M. Fontaine and B.~Mazur, \emph{Geometric {G}alois representations},
  Elliptic curves, modular forms, \& {F}ermat's last theorem ({H}ong {K}ong,
  1993), Ser. Number Theory, I, Int. Press, Cambridge, MA, 1995, pp.~41--78.

\bibitem[Gre94]{greenberg:defmotives}
R.~Greenberg, \emph{Iwasawa theory and {$p$}-adic deformations of motives},
  Motives ({S}eattle, {WA}, 1991), Proc. Sympos. Pure Math., vol.~55, Amer.
  Math. Soc., Providence, RI, 1994, pp.~193--223.

\bibitem[Gre99]{greenberg:iwasawa-elliptic}
\bysame, \emph{Iwasawa theory for elliptic curves}, Arithmetic theory of
  elliptic curves ({C}etraro, 1997), Lecture Notes in Math., vol. 1716,
  Springer, Berlin, 1999, pp.~51--144.

\bibitem[Gro91]{gross:kolyvagin}
B.\thinspace{}H. Gross, \emph{Kolyvagin's work on modular elliptic curves},
  $L$-functions and arithmetic (Durham, 1989), Cambridge Univ. Press,
  Cambridge, 1991, pp.~235--256.

\bibitem[GV00]{greenberg-vatsal}
R.~Greenberg and V.~Vatsal, \emph{On the {I}wasawa invariants of elliptic
  curves}, Invent. Math. \textbf{142} (2000), no.~1, 17--63.

\bibitem[GZ86]{gross-zagier}
B.~Gross and D.~Zagier, \emph{Heegner points and derivatives of
  \protect{${L}$}-series}, Invent. Math. \textbf{84} (1986), no.~2, 225--320.

\bibitem[Has95]{hashimoto}
K.~Hashimoto, \emph{Explicit form of quaternion modular embeddings}, Osaka J.
  Math. \textbf{32} (1995), no.~3, 533--546.

\bibitem[Hel07]{Helm}
D.~Helm, \emph{On maps between modular {J}acobians and {J}acobians of {S}himura
  curves}, Israel J. Math. \textbf{160} (2007), 61--117.

\bibitem[Hid88]{hida:padic2}
H.~Hida, \emph{A {$p$}-adic measure attached to the zeta functions associated
  with two elliptic modular forms. {II}}, Ann. Inst. Fourier (Grenoble)
  \textbf{38} (1988), no.~3, 1--83.

\bibitem[How04]{howard:heeg}
B.~Howard, \emph{The {H}eegner point {K}olyvagin system}, Compos. Math.
  \textbf{140} (2004), no.~6, 1439--1472.

\bibitem[Jet08]{jetchev:tamagawa}
D.~Jetchev, \emph{Global divisibility of {H}eegner points and {T}amagawa
  numbers}, Compos. Math. \textbf{144} (2008), no.~4, 811--826.

\bibitem[KM85]{KatzMazur}
N.\thinspace{}M. Katz and B.~Mazur, \emph{Arithmetic moduli of elliptic
  curves}, Princeton University Press, Princeton, N.J., 1985.

\bibitem[Kob13]{kobayashi:gz}
S.~Kobayashi, \emph{The {$p$}-adic {G}ross-{Z}agier formula for elliptic curves
  at supersingular primes}, Invent. Math. \textbf{191} (2013), no.~3, 527--629.

\bibitem[Kol90]{kolyvagin:euler_systems}
V.~A. Kolyvagin, \emph{Euler systems}, The Grothendieck Festschrift, Vol.\ II,
  Birkh\"auser Boston, Boston, MA, 1990, pp.~435--483.

\bibitem[Kol91a]{kolyvagin:structure_of_selmer}
\bysame, \emph{On the structure of {S}elmer groups}, Math. Ann. \textbf{291}
  (1991), no.~2, 253--259.

\bibitem[Kol91b]{kolyvagin:structureofsha}
V.\thinspace{}A. Kolyvagin, \emph{On the structure of {S}hafarevich-{T}ate
  groups}, Algebraic geometry (Chicago, IL, 1989), Springer, Berlin, 1991,
  pp.~94--121.

\bibitem[LZZ13]{liu-zhang-zhang}
Y.~Liu, S.~Zhang, and W.~Zhang, \emph{On $p$-adic {W}aldspurger formula},
  preprint (2013).

\bibitem[McC88]{mccallum:sha}
W.\thinspace{}G. McCallum, \emph{On the {S}hafarevich-{Tate} group of the
  {J}acobian of a quotient of a {F}ermat curve}, Invent. Math. \textbf{53}
  (1988), no.~3, 637--666.

\bibitem[Mil86]{milne:duality}
J.\thinspace{}S. Milne, \emph{Arithmetic duality theorems}, Academic Press
  Inc., Boston, Mass., 1986.

\bibitem[Mor11]{mori}
A.~Mori, \emph{Power series expansions of modular forms and their interpolation
  properties}, Int. J. Number Theory \textbf{7} (2011), no.~2, 529--577.

\bibitem[MR04]{mazur-rubin:kolyvagin_systems}
B.~Mazur and K.~Rubin, \emph{Kolyvagin systems}, Mem. Amer. Math. Soc.
  \textbf{168} (2004), no.~799, viii+96.

\bibitem[Nek07]{nekovar:euler}
J.~Nekov{\'a}{\v{r}}, \emph{The {E}uler system method for {CM} points on
  {S}himura curves}, {$L$}-functions and {G}alois representations, London Math.
  Soc. Lecture Note Ser., vol. 320, Cambridge Univ. Press, Cambridge, 2007,
  pp.~471--547.

\bibitem[PR87]{perrin-riou:mc}
B.~Perrin-Riou, \emph{Fonctions {$L$} {$p$}-adiques, th\'eorie d'{I}wasawa et
  points de {H}eegner}, Bull. Soc. Math. France \textbf{115} (1987), no.~4,
  399--456.

\bibitem[PR03]{perrin-riou:expmath}
Bernadette Perrin-Riou, \emph{Arithm\'etique des courbes elliptiques \`a
  r\'eduction supersinguli\`ere en {$p$}}, Experiment. Math. \textbf{12}
  (2003), no.~2, 155--186.

\bibitem[Rib90]{ribet:modreps}
K.\thinspace{}A. Ribet, \emph{On modular representations of \protect{${\rm
  {G}al}(\overline{\bf{Q}}/{\bf {Q}})$} arising from modular forms}, Invent.
  Math. \textbf{100} (1990), no.~2, 431--476.

\bibitem[RT97]{ribet-takahashi}
K.~Ribet and S.~Takahashi, \emph{Parametrizations of elliptic curves by
  {S}himura curves and by classical modular curves}, Proc. Nat. Acad. Sci.
  U.S.A. \textbf{94} (1997), no.~21, 11110--11114, Elliptic curves and modular
  forms (Washington, DC, 1996).

\bibitem[Rub00]{rubin:book}
K.~Rubin, \emph{{E}uler {S}ystems}, Princeton University Press, Spring 2000,
  {A}nnals of {M}athematics {S}tudies {\bf 147}.

\bibitem[Sai97]{saito}
T.~Saito, \emph{Modular forms and \protect{$p$}-adic {H}odge theory}, Invent.
  Math. \textbf{129} (1997), 607--620.

\bibitem[Ski14a]{skinner:conversegz}
C.~Skinner, \emph{A converse to a theorem of {G}ross, {Z}agier, and
  {K}olyvagin}, preprint (2014).

\bibitem[Ski14b]{skinner:multred}
\bysame, \emph{Multiplicative reduction and the cyclotomic main conjecture for
  $\mathbf{GL}_2$}, preprint (2014).

\bibitem[SU13]{skinner-urban:gl2}
C.~Skinner and E.~Urban, \emph{The main conjecture for $\mathbf{GL}_2$}, to
  appear in Invent. Math. (2013).

\bibitem[SZ14]{skinner-zhang}
C.~Skinner and W.~Zhang, \emph{Indivisibility of heegner points in the
  multiplicative case}, preprint (2014).

\bibitem[Tat63]{tate:duality}
J.~Tate, \emph{Duality theorems in {G}alois cohomology over number fields},
  Proc. Internat. Congr. Mathematicians (Stockholm, 1962), Inst.
  Mittag-Leffler, Djursholm, 1963, pp.~288--295. \MR{31 \#168}

\bibitem[Tat74]{tate:arithmetic}
\bysame, \emph{The arithmetic of elliptic curves}, Invent. Math. \textbf{23}
  (1974), 179--206. \MR{54 \#7380}

\bibitem[Tat66]{tate:bsd}
\bysame, \emph{On the conjectures of {B}irch and {S}winnerton-{D}yer and a
  geometric analog}, S\'eminaire Bourbaki, Vol.\ 9, Soc. Math. France, Paris,
  1965/66, pp.~Exp.\ No.\ 306, 415--440.

\bibitem[TY07]{tayloryoshida:local-global}
R.~Taylor and T.~Yoshida, \emph{Compatibility of local and global {L}anglands
  correspondences}, J. Amer. Math. Soc. \textbf{20} (2007), no.~2, 467--493.

\bibitem[Wan13]{wan:rankin}
X.~Wan, \emph{Iwasawa main conjecture for {R}ankin--{S}elberg $p$-adic
  {$L$}-functions}, preprint (2013).

\bibitem[Wan14a]{wan:rankin-ss}
\bysame, \emph{Iwasawa main conjecture for {R}ankin--{S}elberg $p$-adic
  $l$-functions: non-ordinary case}, preprint (2014).

\bibitem[Wan14b]{wan:kobayashi}
\bysame, \emph{Iwasawa main conjecture for supersingular elliptic curves},
  preprint (2014).

\bibitem[YZZ13]{yuanzhangzhang:gz}
X.~Yuan, S.-W. Zhang, and W.~Zhang, \emph{The {G}ross-{Z}agier formula on
  {S}himura curves}, Annals of Mathematics Studies, vol. 184, Princeton
  University Press, Princeton, NJ, 2013.

\bibitem[Zha01a]{zhang:gz-shimura}
S.~Zhang, \emph{Heights of {H}eegner points on {S}himura curves}, Ann. of Math.
  (2) \textbf{153} (2001), no.~1, 27--147.

\bibitem[Zha01b]{zhang:gross-zagier}
S.~W. Zhang, \emph{Gross-{Z}agier formula for $\textrm{GL}_2$}, Asian J. Math.
  \textbf{5} (2001), no.~2, 183--290.

\bibitem[Zha14]{zhang:selheeg}
W.~Zhang, \emph{Selmer groups and divisiblity of {H}eegner points}, preprint
  (2014).

\end{thebibliography}

\end{document}